\documentclass{amsart}

\usepackage{amsmath,arydshln,multirow}
\usepackage[colorlinks, linkcolor=blue,anchorcolor=Periwinkle,
    citecolor=blue,urlcolor=Emerald]{hyperref}
\usepackage[all]{xy}
\SelectTips{cm}{}
\usepackage{tikz}
\usepackage{mathtools}
\usepackage{extarrows}
\usepackage{amsfonts,amsmath,amssymb,amscd,bbm,amsthm,mathrsfs,dsfont}

\numberwithin{equation}{subsubsection}

\newtheorem{Theorem}[subsubsection]{Theorem}

\newtheorem{Proposition-Definition}[subsubsection]{Proposition-Definition}

\newtheorem{Lemma}[subsubsection]{Lemma}
\newtheorem{Proposition}[subsubsection]{Proposition}
\newtheorem{Corollary}[subsubsection]{Corollary}

\newtheorem{Example}[subsubsection]{Example}
\newtheorem{Remark}[subsubsection]{Remark}
\newtheorem{Definition}[subsubsection]{Definition}
\newcommand{\reminder}[1]{}

\begin{document}

\date{version of \today}

\title{$\mathcal G$-systems}

\author{Peigen Cao}
\address{Peigen Cao
\newline Department
of Mathematics, Zhejiang University (Yuquan Campus), Hangzhou, Zhejiang
310027,  P.R.China}
\email{peigencao@126.com}

\renewcommand{\thefootnote}{\alph{footnote}}

\renewcommand{\thefootnote}{\alph{footnote}}
\setcounter{footnote}{-1} 
\renewcommand{\thefootnote}{\alph{footnote}}
\setcounter{footnote}{-1} 

\begin{abstract}

A $\mathcal G$-system is a collection of $\mathbb Z$-bases of $\mathbb Z^n$ with some extra axiomatic conditions. There are two kinds of actions ``mutations" and ``co-Bongartz completions" naturally acting on a $\mathcal G$-system, which provide the combinatorial structure of a $\mathcal G$-system. It turns out that ``co-Bongartz completions" have good  compatibility  with ``mutations".

The constructions of ``mutations" are known before in different contexts, including cluster tilting theory, silting theory, $\tau$-tilting theory, cluster algebras, marked surfaces. We found that in addition to ``mutations", there exists another kind of actions ``co-Bongartz completions" naturally appearing in these different theories. With the help of ``co-Bongartz completions" some good combinatorial results can be easily obtained.
In this paper, we give the constructions of ``co-Bongartz completions" in different theories. Then we show that $\mathcal G$-systems naturally arise from these theories, and the ``mutations" and ``co-Bongartz completions" in different theories are compatible with those in $\mathcal G$-systems.
\end{abstract}

\maketitle
\bigskip

  \tableofcontents

\section{Introduction}

Cluster algebras, invented  \cite{FZ} by  Fomin
and  Zelevinsky around the year 2000, are commutative algebras
whose generators and relations are constructed in a recursive manner. The generators of a cluster algebra are called  {\em cluster variables}, which are grouped into overlapping {\em clusters} of the same size. One remarkable feature of cluster algebras is that they have the Laurent phenomenon, which says that for any given cluster ${\bf x}_{t_0}=\{x_{1;t_0},\cdots,x_{n;t_0}\}$, any cluster variable $x_{i;t}$ can be written as a Laurent polynomial in $x_{1;t_0},\cdots,x_{n;t_0}$.

In \cite{CL1} the authors proved that any  skew-symmetrizable cluster algebra $\mathcal A(\mathcal S)$ with principal coefficients at $t_0$ (see \cite{FZ3}) has the enough $g$-pairs property. Roughly speaking, for any cluster ${\bf x}_t$ of $\mathcal A(\mathcal S)$ and any subset $U\subseteq{\bf x}_{t_0}=\{x_{1;t_0},\cdots,x_{n;t_0}\}$, there exists a unique cluster ${\bf x}_{t^\prime}$ such that $U\subseteq{\bf x}_{t^\prime}$ and the $G$-matrices $G_t$ and $G_{t^\prime}$ satisfying that the $i$-th row vector of $G_{t^\prime}^{-1}G_t$ is a nonnegative vector for any $i$ with $x_{i;t^\prime}\notin U$. Using the terminology in \cite{CL1}, $({\bf x}_t,{\bf x}_{t^\prime})$ is called a {\em $g$-pair along ${\bf x}_{t_0}\backslash U$} and
 using the terminology in present paper,  ${\bf x}_{t^\prime}$ is called the {\em co-Bongartz completion} of $(U,{\bf x}_{t_0})$ with respect to ${\bf x}_t$  and  is denoted by ${\bf x}_{t^\prime}=\mathcal T_{(U,{\bf x}_{t_0})}({\bf x}_t)$. From the viewpoints in \cite{CL1,CL2}, some problems in cluster algebras, for example, conjectures on denominator vectors in \cite{FZ3} and unistructural conjecture in \cite{ASV} of cluster algberas would become easier to study when we consider the Laurent expansion of ${\bf x}_{t}$ with respect to the cluster ${\bf x}_{t^\prime}=\mathcal T_{(U,{\bf x}_{t_0})}$ for special choice of $U$.

The starting point of this paper is trying to give categorification explanation and topological explanation  of ``co-Bongartz completions" (or known as $g$-pairs in \cite{CL1}) in cluster algebras. Since the ``co-Bongartz completions" in cluster algebras are essentially defined by a classes of integer matrices known as $G$-matrices, this motivates us to introduce the $\mathcal G$-systems, which can be considered as a kind of axiomatization of $G$-matrices.  We found that some conjectured good combinatorial results in cluster algebras can be obtained in $\mathcal G$-systems with the help of ``co-Bongartz completions".
We construct ``co-Bongartz completions" in different contexts, including  cluster tilting theory,  silting theory,  $\tau$-tilting theory, and  unpunctured  surfaces. Then we use the $\mathcal G$-systems as a common framework to  unify ``co-Bongartz completions" in different theories. These in return give the categorification explanation and topological explanation  of ``co-Bongartz completions" in cluster algebras.

This paper is organized as follows. In Section 2 we recall the $\tau$-tilting theory introduced in \cite{AIR}, and give the definition of {\em co-Bongartz completions in triangulated categories}. In Section 3 basics on cluster tilting theory and silting theory are recalled.

In Section 4 we give a detailed discuss on co-Bongartz completions in $2$-CY triangulated categories. As applications, some combinatorial results on cluster tilting objects  are obtained. Note that this section serves  as a comparison for Section 5. In Section 5 we introduce the $\mathcal G$-systems. We prove that there are two kinds of actions ``mutations" and ``co-Bongartz completions" naturally acting on the $\mathcal G$-systems. Then we found that the similar combinatorial results can be obtained in $\mathcal G$-systems as previous section, while there is no category environment in $\mathcal G$-systems.

In Section 6 we show that $\mathcal G$-system naturally arise from cluster tilting theory, silting theory and $\tau$-tilting theory and we also show that the ``co-Bongartz completions" and ``mutations" in $\mathcal G$-systems are compatible with those in different theories.

In Subsections 7.1 and 7.2 we recall the cluster algebras. In Subsection 7.3 we first  refer to \cite{CL1} to recall the construction of the co-Bongartz completion $\mathcal T_{(U,{\bf x}_{t_0})}({\bf x}_t)$ of $(U,{\bf x}_{t_0})$ with respect to a cluster ${\bf x}_t$ (here $U\subseteq{\bf x}_{t_0}$) in a cluster algebra. Then we prove two things. Firstly, we show that $\mathcal G$-systems naturally arise from cluster algebras and the ``mutations" and ``co-Bongartz completions" in cluster algebras are compatible with the ``mutations" and ``co-Bongartz completions" in $\mathcal G$-systems. Secondly, we show that the co-Bongartz completion $\mathcal T_{(U,{\bf x}_{t_0})}({\bf x}_t)$ is uniquely determined by ${\bf x}_t$ and $U$, not depending on the choice ${\bf x}_{t_0}$.

In Section 8 we construct  ``co-Bongartz completions" on unpunctured surfaces and prove that the ``co-Bongartz completions" on unpunctured surfaces are compatible with the ``co-Bongartz completions" in the associated cluster algebras by showing that they are both compatible with the ``co-Bongartz completions" in $\mathcal G$-systems.

Finally, we want to explain why we consider ``co-Bongartz completions" but ``Bongartz completions". We refer to \cite{Q} for the terminology. Let $\mathcal A$ be a  cluster algebra with principal coefficients at $t_0$, and $x$ be a cluster variable of $\mathcal A$. The {\em degree} (i.e, $g$-vector) and {\em codegree} of $x$ with respect to the initial seed $t_0$ can be defined. If we want to consider ``Bongartz completions" in cluster algebras, we should focus on codegree. Howerver, it is more natural to deal with degree (i.e., $g$-vector) in cluster algebras and this lead us to  consider ``co-Bongartz completions" but ``Bongartz completions".

\section{Preliminaries}

\subsection{Left (right) approximation in additive category}

Let $\mathcal C$ be an additive category, and $\mathcal D$ be a full subcategory which is closed under isomorphism, direct sums and direct summands.
We call a morphism $X\overset{f}{\rightarrow}Y$ in $\mathcal C$  {\bf right minimal} if it does not have a direct summand of the form $X_1\rightarrow 0$ as a complex.
We call  $\xymatrix{X\ar[r]^f&Y}$ a {\bf right $\mathcal D$-approximation of $Y\in\mathcal C$} if $X\in\mathcal D$ and
$${\rm Hom}_{\mathcal C}(-,X)\xrightarrow[]{{\rm Hom}_{\mathcal C}(-,f)}{\rm Hom}_{\mathcal C}(-,Y)\longrightarrow 0$$
is exact as functors on $\mathcal D$. We call a  right $\mathcal D$-approximation minimal if it is right minimal.

 We call $\mathcal D$ a {\bf contravariantly finite subcategory} of $\mathcal C$ if any $Y\in\mathcal C$ has a right
$\mathcal D$-approximation.  Dually,  a left (minimal) $\mathcal D$-approximation,  and a  {\em covariantly finite
subcategory} can be defined.

It is known that minimal right (left)  $\mathcal D$-approximation is unique up to isomorphism. Furthermore, a subcategory $\mathcal D$ of $\mathcal C$ is said to be {\bf functorially finite} in $\mathcal C$ if it is both contravariantly and
covariantly finite in $\mathcal C$.

For an object $X\in \mathcal C$, denote by ${\rm add} (X)$ the category of objects consisting of the direct summands of finite direct sums of $X$, and  by ${\rm Fac}(X)$ the category of all   factor objects of finite direct sums of copies of $X$. We also denote by
\begin{eqnarray}
 X^\bot&:=&\{Y\in \mathcal C| {\rm Hom}_{\mathcal C}(X, Y)=0 \},\nonumber\\
  \prescript{\bot}{}{X}&:=&\{Y\in \mathcal C | {\rm Hom}_{\mathcal C}(Y, X)=0 \}.\nonumber
\end{eqnarray}
We also denote  by $X^\flat$ the basic object (up to isomorphism) in $\mathcal C$  such that ${\rm add}(X^\flat)={\rm add}(X)$.

\subsection{$\tau$-tilting theory}
The $\tau$-tilting theory was introduced by Adachi, Iyama and Reiten in \cite{AIR}. It extends classical tilting theory from the viewpoint of mutation.

In this subsection, we fix a finite dimensional basic algebra $A$ over a field $K$ and  denote by ${\rm mod} A$ the category of  finitely generated left $A$-modules. A module $M\in{\rm mod} A$ satisfying ${\rm Hom}_A(M,\tau M)=0$ is called a {\bf $\tau$-rigid module}.

\begin{Definition}\cite{AIR}
Consider an $A$-module $M$ and a projective $A$-module $P$. The pair $(M,P)$ is called a {\bf $\tau$-rigid pair} if
$M$ is $\tau$-rigid and ${\rm Hom}_A(P,M)=0$.
\end{Definition}

 We say that a $\tau$-rigid
pair $(M,P)$ is a {\bf support $\tau$-tilting pair} if $|M|+|P|=n=|A|$, and a {\bf almost support $\tau$-tilting pair} if $|M|+|P|=n-1$, where $|N|$ is the number of indecomposable direct summands (up to isomorphism) of $N$.  $M\in {\rm mod} A$ is called a {\bf support $\tau$-tilting module} if there exists a projective module $P$ such that $(M,P)$ is a support $\tau$-tilting pair.

\begin{Proposition}\label{prosincere}\cite{AIR} Let $(M,P)$ and $(M,Q)$ be two basic support $\tau$-tilting pairs in ${\rm mod}A$, then $P\cong Q$.
\end{Proposition}

\begin{Theorem}\cite[Theorem 2.18]{AIR} \label{thmair}Let $A$ be a finite dimensional basic $K$-algebra. Then
any basic almost support $\tau$-tilting pair is  a direct summand of exactly
two basic support $\tau$-tilting pairs in ${\rm mod} A$.
\end{Theorem}

Let $(U,P^U)$ be a basic almost support $\tau$-tilting pair in ${\rm mod} A$, and $(M,P^M)=(U,P^U)\oplus(X,P^X)$, $(M^\prime,P^{M^\prime})=(U,P^U)\oplus(Y,P^Y)$ be the two basic support $\tau$-tilting pairs containing $(U,P^U)$ as their direct summand with $(X,P^X)\ncong (Y,P^Y)$. $(M,P^M)$ is called the {\bf mutation of $(M^\prime,P^{M^\prime})$ at $(Y,P^Y)$} and we denote $(M,P^M)=\mu_{(Y,P^Y)}(M^\prime,P^{M^\prime})$.

\subsection{$g$-vectors in $\tau$-tilting theory}

Let $A$ be a finite dimensional basic algebra over $K$, and $P_1,\cdots,P_n$ be the isomorphism classes of indecomposable projective modules.  For $M\in{\rm mod}A$, take a minimal projective presentation of $M$, say
$$\xymatrix{\bigoplus\limits_{i=1}^nP_i^{b_i}\ar[r]&\bigoplus\limits_{i=1}^nP_i^{a_i}\ar[r]&M\ar[r]&0}.$$
Then the vector ${\bf g}(M):=(a_1-b_1,\cdots,a_n-b_n)^{\rm T}\in\mathbb Z^n$ is called the {\bf $g$-vector} of $M$.

For a $\tau$-rigid pair $(M,P)$, the $g$-vector of $(M,P)$ is defined to be the vector $${\bf g}(M,P):={\bf g}(M)-{\bf g}(P).$$
For a basic support $\tau$-tilting pair $(M,P^M)=\bigoplus\limits_{i=1}^n(M_i,P^M_i)$, the {\bf $G$-matrix} of $(M,P^M)$ is defined to be the matrix $G_{(M,P^M)}=({\bf g}(M_1,P^M_1),\cdots,{\bf g}(M_n,P^M_n))$.

\begin{Theorem}\cite[Theorem 5.1]{AIR} Let $(M,P^M)=\bigoplus\limits_{i=1}^n(M_i,P^M_i)$ be a basic support $\tau$-tilting pair in ${\rm mod}A$, then ${\bf g}(M_1,P^M_1),\cdots,{\bf g}(M_n,P^M_n)$ forms a $\mathbb Z$-basis of $\mathbb Z^n$.
\end{Theorem}

\begin{Theorem}\cite[Theorem 5.5]{AIR}\label{thmairinj} The map $(M,P)\mapsto {\bf g}(M,P)$ gives a injection from the set of isomorphism classes of $\tau$-rigid pairs to $\mathbb Z^n$.
\end{Theorem}

\subsection{Functorially finite classes}

A {\bf torsion pair} $(\mathfrak{T},\mathfrak{F})$ in  ${\rm mod} A$ is a pair of subcategories of ${\rm mod}A$ satisfying that

(i) ${\rm Hom}_A(T,F)=0$ for any $T\in \mathfrak{T},F\in \mathfrak{F}$;

(ii) for any module $X\in {\rm mod} A$, there exists a short exact sequence
 $0\rightarrow X^{\mathfrak{T}}\rightarrow X\rightarrow X^{\mathfrak{F}}\rightarrow0$ with $X^{\mathfrak{T}}\in\mathfrak{T}$ and $X^{\mathfrak{F}}\in \mathfrak{F}$. This short exact sequence is called the {\bf canonical sequence} for $X$.

The subcategory $\mathfrak T$ (respectively, $\mathfrak F$) in a torsion pair $(\mathfrak T,\mathfrak F)$ is called a {\bf torsion class} (respectively, {\bf torsionfree class}) of ${\rm mod} A$.

We say that $X\in\mathfrak T$ is {\bf Ext-projective} in $\mathfrak{T}$, if ${\rm Ext}_A^1(X,\mathfrak T)=0$. We denote by $\mathcal P(\mathfrak T)$ the direct sum of one copy of each of the indecomposable Ext-projective objects in $\mathfrak T$ up to isomorphism.

We denote by ${\rm s}\tau$-${\rm tilt}A$ the set of basic support $\tau$-tilting modules in ${\rm mod}A$, and by ${\rm f}$-${\rm tors}A$ the set of functorially finite torsion classes in ${\rm mod}A$.
\begin{Theorem}\label{proorder}\cite[Theorem 2.7]{AIR}
There is a bijection
$${\rm s}\tau\text{-}{\rm tilt}A\longleftrightarrow {\rm f\text{-}tors}A$$
given by ${\rm s}\tau\text{-}{\rm tilt}A\ni M\mapsto {\rm Fac}(M)\in {\rm f\text{-}tors}A$ and ${\rm f\text{-}tors}A\ni \mathfrak T\mapsto\mathcal P(\mathfrak T)\in {\rm s}\tau\text{-}{\rm tilt}A$.
\end{Theorem}

\begin{Proposition}
\cite{AS} If $U\in {\rm mod}A$ is $\tau$-rigid, then ${\rm Fac}(U)$ is a functorially finite torsion class and  $U\in {\rm add}(M_U)$, where $M_U:=\mathcal P({\rm Fac}(U))$ is  the support $\tau$-tilting module in ${\rm mod}A$ given by ${\rm Fac}(U)$.
\end{Proposition}

\begin{Theorem}\cite[Theorem 2.10]{AIR} Let $U$ be a $\tau$-rigid $A$-module. Then $\mathfrak T:=\prescript{\bot}{}({\tau U})$ is a functorially finite
torsion class and $T_U := \mathcal P(\mathfrak T)$ is a support $\tau$-tilting $A$-module satisfying $U\in {\rm add}(T_U)$ and $\prescript{\bot}{}({\tau U})={\rm Fac}(T_U)$.
\end{Theorem}

\begin{Proposition}\label{prominmax}\cite[Proposition 2.9]{AIR}
Let $U$ be a $\tau$-rigid module and $\mathfrak T$ be a functorially finite class. Then $U\in {\rm add}(\mathcal P(\mathfrak T))$ if and only if ${\rm Fac}(U)\subseteq \mathfrak T\subseteq \prescript{\bot}{}({\tau U})$.
\end{Proposition}

\begin{Definition}Let $U$ be a $\tau$-rigid module, the support $\tau$-tilting module $\mathcal P({\rm Fac}(U))$ is called the {\bf co-Bongartz completion} of $U$ and the  support $\tau$-tilting module $\mathcal P(\prescript{\bot}{}({\tau U}))$ is called the {\bf Bongartz completion} of $U$.
\end{Definition}

\subsection{Co-Bongartz completions in triangulated categories}

Recall that for an object $X$ in a additive category, $X^\flat$ denotes the basic object (up to isomorphism) such that ${\rm add}(X^\flat)={\rm add}(X)$.
\begin{Definition}
Let $\mathcal D$ be a $K$-linear, Krull-Schmidt, Hom-finite  triangulated category, and $U, W$ be two objects in $\mathcal D$. Consider the following triangle
$$
\xymatrix{W[-1] \ar[r]^f&U^\prime \ar[r]&X_U \ar[r]&W},
$$
where $f$ is a left minimal ${\rm add}(U)$-approximation of $W[-1]$. We call the basic object $\mathcal T_U(W):=(U\oplus X_U)^\flat$ the {\bf co-Bongartz completion} of $U$ with respect to $W$. If $U$ is an indecomposable  object, we call $\mathcal T_U(W)=(U\oplus X_U)^\flat$ the {\bf elementary co-Bongartz completion} of $U$ with respect to $W$.
\end{Definition}

\section{Cluster tilting theory and silting theory}
\subsection{Basics on cluster tilting theory}
In this subsection, we fix a $K$-linear, Krull-Schmidt, Hom-finite $2$-Calabi-Yau triangulated category $\mathcal C$ with a  basic cluster tilting object $T=\bigoplus\limits_{i=1}^nT_i$.

Recall that $\mathcal C$ is $2$-Calabi-Yau ($2$-CY for short) means that we have a  bifunctorial isomorphism
$${\rm Hom}_{\mathcal C}(X,Y)=D{\rm Hom}_{\mathcal C}(Y,X[2]),$$
for any $X,Y\in\mathcal C$, where $D={\rm Hom}_K(-,K)$.
\begin{itemize}
 \item An object $T\in\mathcal C$ is {\bf cluster tilting} if
$${\rm add}(T)=\{X\in\mathcal C|{\rm Hom}_{\mathcal C}(T,X[1])=0\}.$$
\item An object $U\in\mathcal C$ is {\bf almost cluster tilting} if there exists an indecomposable object $X\notin {\rm add}(U)$ such that $U\oplus X$ is a cluster tilting object in $\mathcal C$.
     \item An object $U\in\mathcal C$ is  {\bf rigid} if ${\rm Hom}_{\mathcal C}(U,U[1])=0$. Clearly, both cluster tilting objects and almost cluster tilting objects are rigid.
         \end{itemize}

For $M\in \mathcal C$, denote by $|M|$ the number of non-isomorphic indecomposable direct summands of $M$.
It is known from \cite[Theorem 2.4]{DK} that $|T|=|T^\prime|$ for any two cluster tilting objects in $\mathcal C$.

\begin{Proposition}\cite{KR,Pl}\label{protilt}
Let $\mathcal C$ be a $K$-linear, Krull-Schmidt, Hom-finite $2$-Calabi-Yau triangulated category with a basic cluster tilting object $T=\bigoplus\limits_{i=1}^nT_i$. Then for any  object $X\in \mathcal C$, there exists a triangle
$$\xymatrix{T_1^X\ar[r]^{\alpha_2}& T_0^X\ar[r]^{\alpha_1}& X\ar[r]^{\alpha_3}& T_1^X[1]},$$
where $\alpha_1$ is a right ${\rm add}(T)$-approximation of $X$ and $T_0^X=\bigoplus\limits_{i=1}^nT_i^{a_i}\in {\rm add}(T)$ and $T_1^X=\bigoplus\limits_{i=1}^nT_i^{b_i}\in{\rm add}(T)$.
\end{Proposition}

\begin{Theorem}\cite[Theorem 5.3]{IY} \label{thm2tilt} Let $\mathcal C$ be a $K$-linear, Krull-Schmidt, Hom-finite $2$-CY triangulated category. Then any basic almost cluster tilting object in $\mathcal C$ is a direct summand of exactly two basic cluster tilting objects in $\mathcal C$.
\end{Theorem}

Let $U$ be a basic almost cluster tilting object in $\mathcal C$, and $T=U\oplus X,\;T^\prime=U\oplus Y$ be the two basic cluster tilting objects containing $U$ as their direct summand with $X\ncong Y$. $T$ is called the {\bf mutation of $T^\prime$ at $Y$} and we denote $T=\mu_Y(T^\prime)$. Also  $T^\prime$ is called the {\bf mutation of $T$ at $X$} and is denoted by $T^\prime=\mu_X(T)$.

\subsection{$g$-vectors in cluster tilting theory}

Let $\mathcal C$ be a $K$-linear, Krull-Schmidt, Hom-finite $2$-Calabi-Yau triangulated category with a basic cluster tilting object $R=\bigoplus\limits_{i=1}^nR_i$. By Proposition \ref{protilt}, for any  object $X\in\mathcal C$, there exists a triangle
\begin{eqnarray}\label{eqndefg}
\bigoplus\limits_{i=1}^nR_i^{b_i}\overset{\alpha_2}{\longrightarrow}\bigoplus\limits_{i=1}^nR_i^{a_i}
\overset{\alpha_1}{\longrightarrow} X\overset{\alpha_3}{\longrightarrow}(\bigoplus\limits_{i=1}^nR_i^{b_i})[1],
\end{eqnarray}
where $\alpha_1$ is a right ${\rm add}(R)$-approximation of $X$.
\begin{Definition}
Keep the notations above, the {\bf $g$-vector (index)} of $X$ with respect to the basic cluster tilting object $R$ is the integer vector
$${\bf g}^R(X)=(a_1-b_1,\cdots,a_n-b_n)^{\rm T}\in\mathbb Z^n.$$
\end{Definition}

Note that even though the triangle in (\ref{eqndefg}) is not unique, the  $g$-vector (index) of $X$ is well-defined.

Let $T=\bigoplus\limits_{i=1}^nT_i$ be another basic cluster tilting object in $\mathcal C$, then we can define the {\bf $G$-matrix} $G_{T}^R$ of $T$ with respect to $R$ by
$$G_{T}^R=({\bf g}^R(T_1),\cdots,{\bf g}^R(T_n)).$$

\begin{Proposition}\label{prodk} Let $R$ be a basic cluster tilting object in $\mathcal C$.

(i) \cite{DK} For any basic cluster tilting object $T\in\mathcal C$, we have  $det(G_{T}^R)=\pm 1$;

(ii) \cite{DK} For any two rigid objects $M,N\in\mathcal C$, if ${\bf g}^R(M)={\bf g}^R(N)$, then $M\cong N$.

(iii) \cite{P} If $X\longrightarrow Y\longrightarrow Z\overset{f}{\longrightarrow} X[1]$ is a triangle, and if $f$ factors though an object in ${\rm add}(R[1])$, then ${\bf g}^R(Y)={\bf g}^R(X)+{\bf g}^R(Z)$.
\end{Proposition}

\begin{Proposition}\cite{BMR,KZ}
Let $\mathcal C$ be a $K$-linear, Krull-Schmidt, Hom-finite $2$-CY triangulated category. Let $R=\bigoplus\limits_{i=1}^nR_i$ be  a basic cluster tilting object in $\mathcal C$, and $A=A_R:={\rm End}_{\mathcal C}^{\rm op}(R)$. The functor $H:={\rm Hom}_{\mathcal D}(R,-): \mathcal C\rightarrow {\rm mod}A$ induces an equivalence of the categories
$$H:\frac{\mathcal C}{[R[1]]}\rightarrow {\rm mod}A,$$
where $[R[1]]$ is the ideal of $\mathcal C$ consisting of morphisms factor through ${\rm add}(R[1])$.
\end{Proposition}
We denote by ${\rm iso}\;\mathcal C$ the set of isomorphism classes of objects in a category $\mathcal C$. Then by the above proposition, we have a bijection
\begin{eqnarray}
\tilde H:{\rm iso}\;\mathcal C\longrightarrow {\rm iso}\;{\rm mod}A\times {\rm iso}\;{\rm add}(R),
\end{eqnarray}
given by $X=X_0\oplus X_1\mapsto(H(X_0),H(X_1[-1]))$, where $X_1$ is a maximal direct summand of $X$ that
belongs to ${\rm add}(R[1])$.

\begin{Theorem}\cite[Theorem 4.1, Corollary 5.2]{AIR}\label{thmairbijection}
Let $\mathcal C$ be a $K$-linear, Krull-Schmidt, Hom-finite $2$-CY triangulated category. Let $R=\bigoplus\limits_{i=1}^nR_i$ be  a basic cluster tilting object in $\mathcal C$, and $A=A_R:={\rm End}_{\mathcal C}^{\rm op}(R)$. Then

(i) $\tilde H$ induces a bijection
from the set of basic rigid objects in $\mathcal C$ to the set of basic $\tau$-rigid pairs in ${\rm mod}A$, which induces a bijection from the set of basic cluster tilting objects in $\mathcal C$ to the set of basic support $\tau$-tilting pairs in ${\rm mod}A$.

(ii) for any rigid object $X$ in $\mathcal C$,  the $g$-vector ${\bf g}^R(X)$ of $X$ with respect to $R$ in $\mathcal C$ is equal to the $g$-vector ${\bf g}(\tilde H(X))$ of $\tilde H(X)$ in ${\rm mod}A$.
\end{Theorem}

\subsection{Basics on silting theory}

 In this subsection,  we fix  a $K$-linear, Krull-Schmidt, Hom-finite  triangulated category $\mathcal D$ with a  basic silting object $S=\bigoplus\limits_{i=1}^nS_i$.

 Recall that $U\in\mathcal D$ is {\bf presilting} if ${\rm Hom}_{\mathcal D}(U,U[i])=0$ for any positive integer $i>0$.
A presilting object $S\in\mathcal D$ is {\bf silting} if $\mathcal D={\rm thick}(S)$.
A presilting object $U\in\mathcal D$ is {\bf almost silting} if there exists an indecomposable object $X\notin{\rm add}(U)$ such that $U\oplus X$ is silting.  We denote the set of isomorphism classes of all basic silting objects in $\mathcal D$ by ${\rm silt}(\mathcal D)$.

\begin{Theorem}\cite[Theorem 2.27]{AI}\label{thmk0basis}
Let $\mathcal D$ be a $K$-linear, Krull-Schmidt, Hom-finite  triangulated category with a  basic silting object $S=\bigoplus\limits_{i=1}^nS_i$, then  the
Grothendieck group ${\rm K}_0(\mathcal D)$ of $\mathcal D$ is a free abelian group with a basis $[S_1],\cdots,[S_n]$.
\end{Theorem}

For any $X,Y\in\mathcal D$, we write $X\ast Y$ for the full subcategory of $\mathcal D$ consisting of all objects $Z\in\mathcal D$ such that there exists a triangle
$$\xymatrix{X^\prime\ar[r]&Z\ar[r]&Y^\prime\ar[r]&X^\prime[1]},$$
where $X^\prime\in{\rm add}(X)$ and $Y^\prime\in{\rm add}(Y)$.
It is known from  \cite[Proposition 2.1(1)]{IY}, if ${\rm Hom}_{\mathcal D}(X,Y)=0$, then $X\ast Y$ is closed under direct summand.

 \begin{Proposition}\cite[Proposition 4.4]{J}\label{projasso} Let $S$ be a silting object in $\mathcal D$ and $M\in\mathcal D$. Then $M\in S\ast S[1]$ if and only if ${\rm Hom}_{\mathcal D}(S,M[i])=0$ and ${\rm Hom}_{\mathcal D}(M,S[j+1])=0$ for any $i,j>0$.
 \end{Proposition}

 \begin{Corollary}\label{corclosed}
 Let $S$ be a silting object in $\mathcal D$, then $S\ast S[1]$ is closed under extension.
 \end{Corollary}
\begin{proof}
Let $X\rightarrow Y\rightarrow Z\rightarrow X[1]$ be a triangle in $\mathcal D$ with $X, Z\in S\ast S[1]$. By Proposition \ref{projasso}, we know that ${\rm Hom}_{\mathcal D}(S,X[i])=0,\;{\rm Hom}_{\mathcal D}(X,S[j+1])=0$ and ${\rm Hom}_{\mathcal D}(S,Z[i])=0,\;{\rm Hom}_{\mathcal D}(Z,S[j+1])=0$ for any $i,j>0$.

For any $i>0$, applying the functor ${\rm Hom}_{\mathcal D}(S[-i],-)$ to the triangle $X\rightarrow Y\rightarrow Z\rightarrow X[1]$, we get the following exact sequence
$$0={\rm Hom}_{\mathcal D}(S[-i],X)\rightarrow {\rm Hom}_{\mathcal D}(S[-i],Y)\rightarrow {\rm Hom}_{\mathcal D}(S[-i],Z)=0.$$
So ${\rm Hom}_{\mathcal D}(S,Y[i])\cong {\rm Hom}_{\mathcal D}(S[-i],Y)=0$.
Similarly, we can show ${\rm Hom}_{\mathcal D}(Y,S[j+1])=0$ for any $j>0$. By Proposition \ref{projasso}, we get $Y\in S\ast S[1]$. Thus $S\ast S[1]$ is closed under extension.
\end{proof}

Let $S$ be a basic silting object in $\mathcal D$, then the silting objects in $S\ast S[1]$  are called {\bf two-term silting objects at $S$}.
\begin{Proposition}\cite[Proposition 6.2(3)]{IY}
Let $S$ be a basic silting object in $\mathcal D$, and $A=A_S:={\rm End}_{\mathcal D}^{\rm op}(S)$. The functor $F:={\rm Hom}_{\mathcal D}(S,-): S\ast S[1]\rightarrow {\rm mod}A$ induces an equivalence of the categories
$$F:\frac{S\ast S[1]}{[S[1]]}\rightarrow {\rm mod}A,$$
where $[S[1]]$ is the ideal of $\mathcal D$ consisting of morphisms factor through ${\rm add}(S[1])$.
\end{Proposition}
We denote by ${\rm iso}\;\mathcal D$ the set of isomorphism classes of objects in a category $\mathcal D$. Then by the above proposition, we have a bijection
\begin{eqnarray}\label{eqnfbijection}
\tilde F:{\rm iso}\;S\ast S[1]\longrightarrow {\rm iso}\;{\rm mod}A\times {\rm iso}\;{\rm add}(S),
\end{eqnarray}
given by $X=X_0\oplus X_1\mapsto(F(X_0),F(X_1[-1]))$, where $X_1$ is a maximal direct summand of $X$ that
belongs to ${\rm add}(S[1])$.

\begin{Theorem}\cite{IJY}\label{thmijy} Let $S$ be a basic silting object in $\mathcal D$,  and $A=A_S:={End}_{\mathcal D}(S)^{\rm op}$. Then $\tilde F$ induces a bijection
from the set of basic presilting objects in $S\ast S[1]$ to the set of basic $\tau$-rigid pairs in ${\rm mod}A$, which induces a bijection from the set of basic silting objects in $S\ast S[1]$ to the set of basic support $\tau$-tilting pairs in ${\rm mod}A$.
\end{Theorem}
The following corollary follow from Theorem \ref{thmijy} and Theorem \ref{thmair}.
\begin{Corollary}\label{corsmutation}
Let $\mathcal D$ be a $K$-linear, Krull-Schmidt, Hom-finite  triangulated category $\mathcal D$ with a  basic silting object $S$. Then for any basic  almost silting object $U\in S\ast S[1]$, there are exactly two basic silting objects in  $S\ast S[1]$ containing $U$ as a direct summand.
\end{Corollary}

\begin{Remark}
Let $A$ be a finite-dimensional algebra over a field $K$, and $\mathcal D={\rm K}^b({\rm proj}A)$. We know that $S:=A\in\mathcal D$ is a silting object in $\mathcal D$. In this case, the result in Corollary \ref{corsmutation} was first given in \cite{DF} by  Derksen and Fei.  Lately, Adachi,  Iyama and Reiten in \cite{AIR} reprove this result in the context of $\tau$-tilting theory.
\end{Remark}

Let $S$ be a basic silting object in $\mathcal D$ and $U\in S\ast S[1]$ be a basic almost silting object. Let $T=U\oplus X$ and $T^\prime=U\oplus Y$ be the two basic silting objects in  $S\ast S[1]$ with $X\ncong Y$. $T$ is called the {\bf mutation of $T^\prime$ at $Y$} and we denote $T=\mu_Y(T^\prime)$. Also  $T^\prime$ is called the {\bf mutation of $T$ at $X$} and is denoted by $T^\prime=\mu_X(T)$.

\subsection{$g$-vectors in silting theory}

Let  $\mathcal D$  be a $K$-linear, Krull-Schmidt, Hom-finite  triangulated category with a basic silting object $S=\bigoplus\limits_{i=1}^nS_i$. By Theorem \ref{thmk0basis}, $[S_1],\cdots,[S_n]$ forms a $\mathbb Z$-basis of ${\rm K}_0(\mathcal D)\cong \mathbb Z^n$. For any $X\in \mathcal D$, say

$$[X]=g_1[S_1]+\cdots+ g_n[S_n]\in{\rm K}_0(\mathcal D),$$
then  the vector ${\bf g}^S(X)=(g_1,\cdots,g_n)^{\rm}\in\mathbb Z^n$ is called the {\bf $g$-vector} of $X$ with respect to $S=\bigoplus\limits_{i=1}^nS_i$.

Let $T=\bigoplus\limits_{i=1}^nT_i$ be another basic silting object in $\mathcal D$, we call the martix
$G_T^S=({\bf g}^S(T_1),\cdots,{\bf g}^S(T_n))$ the {\bf $G$-matrix} of $T$ with respect to $S$.

\begin{Remark}\label{rmksbasis}
Keep the above notations.
By Theorem \ref{thmk0basis}, both $[S_1],\cdots,[S_n]$ and $[T_1],\cdots, [T_n]$ are $\mathbb Z$-bases of ${\rm K}_0(\mathcal D)\cong\mathbb Z^n$, we know that $det(G_T^S)=\pm 1$.
\end{Remark}

\begin{Proposition} \label{pronocom1} Let  $\mathcal D$  be a $K$-linear, Krull-Schmidt, Hom-finite  triangulated category with a basic silting object $S=\bigoplus\limits_{i=1}^nS_i$. Let $U$ be a rigid object in $S\ast S[1]$, consider the following triangle
\begin{eqnarray}\label{eqnsilting}
\xymatrix{U[-1]\ar[r]^h&S^{\prime\prime}\ar[r]^g&S^\prime\ar[r]^f&U,
}
\end{eqnarray}
 where $f$ is a  right minimal ${\rm add}(S)$-approximation of $U$. Then  $S^{\prime\prime}\in{\rm add}(S)$.
\end{Proposition}

\begin{proof}

By $U\in S\ast S[1]$ and Proposition \ref{projasso}, we know that ${\rm Hom}_{\mathcal D}(S,U[i])=0$ and ${\rm Hom}_{\mathcal D}(U,S[1+j])=0$ for any $i, j>0$. By $S^\prime\in{\rm add}(S)$, we get that ${\rm Hom}_{\mathcal D}(S,S^\prime[i])=0$ and ${\rm Hom}_{\mathcal D}(S^\prime,S[i])=0$ for any $i>0$.

From the triangle (\ref{eqnsilting}), we know that  $S^{\prime\prime}$ is an extension of $U[-1]$ and $S^\prime$. Thus ${\rm Hom}_{\mathcal D}(S^{\prime\prime},S[i])=0$ and ${\rm Hom}_{\mathcal D}(S,S^{\prime\prime}[1+i])=0$ for any $i>0$.

Now we show that ${\rm Hom}_{\mathcal D}(S,S^{\prime\prime}[1])=0$. Applying the functor ${\rm Hom}_{\mathcal D}(S,-)$ to the triangle (\ref{eqnsilting}), we get the following exact sequence.
$$\xymatrix{{\rm Hom}_{\mathcal D}(S,S^\prime)\ar[rr]^{{\rm Hom}_{\mathcal D}(S,f)}&&{\rm Hom}_{\mathcal D}(S,U)\ar[r]&{\rm Hom}_{\mathcal D}(S,S^{\prime\prime}[1])\ar[r]&{\rm Hom}_{\mathcal D}(S,S^\prime[1])=0
}.$$
Since $f$ is a  right minimal ${\rm add}(S)$-approximation of $U$, we know that ${\rm Hom}_{\mathcal D}(S,f)$ is surjective. So ${\rm Hom}_{\mathcal D}(S,S^{\prime\prime}[1])=0$.

By $U\in S\ast S[1]$, there exists a triangle of the form $$\xymatrix{S_U^1\ar[r]&S_U^0\ar[r]&U\ar[r]&S_U^1[1]},$$
where $S_U^0, S_U^1\in{\rm add}(S)$. Since ${\rm Hom}_{\mathcal D}(S^{\prime\prime},S[i])=0$ for any $i>0$, we have ${\rm Hom}_{\mathcal D}(S^{\prime\prime},S_U^0[i])=0$ and ${\rm Hom}_{\mathcal D}(S^{\prime\prime},(S_U^1[1])[i])=0$ for any $i> 0$. Thus we can get ${\rm Hom}_{\mathcal D}(S^{\prime\prime},U[i])=0$ for any $i> 0$.

By the triangle (\ref{eqnsilting}), and ${\rm Hom}_{\mathcal D}(S^{\prime\prime},S^\prime[i])=0$ and $${\rm Hom}_{\mathcal D}(S^{\prime\prime},(U[-1])[i+1])={\rm Hom}_{\mathcal D}(S^{\prime\prime},U[i])=0,$$ for any $i>0$, we can get ${\rm Hom}_{\mathcal D}(S^{\prime\prime},S^{\prime\prime}[i+1])=0$ for any $i>0$.

Now we show that ${\rm Hom}_{\mathcal D}(S^{\prime\prime},S^{\prime\prime}[1])=0$. Applying the functors ${\rm Hom}_{\mathcal D}(-,U)$ and ${\rm Hom}_{\mathcal D}(S^{\prime\prime},-)$ to the triangle (\ref{eqnsilting}), we get the following two exact sequences.
\begin{eqnarray}
&&\xymatrix{{\rm Hom}_{\mathcal D}(S^\prime,U)\ar[rr]^{{\rm Hom}_{\mathcal D}(g,U)}&&{\rm Hom}_{\mathcal D}(S^{\prime\prime},U)\ar[r]&{\rm Hom}_{\mathcal D}(U[-1],U)=0
}\nonumber\\
&&\xymatrix{{\rm Hom}_{\mathcal D}(S^{\prime\prime},S^\prime)\ar[rr]^{{\rm Hom}_{\mathcal D}(S^{\prime\prime},f)}&&{\rm Hom}_{\mathcal D}(S^{\prime\prime},U)\ar[r]&{\rm Hom}_{\mathcal D}(S^{\prime\prime},S^{\prime\prime}[1])\ar[r]&{\rm Hom}_{\mathcal D}(S^{\prime\prime},S^\prime[1])=0
}\nonumber
\end{eqnarray}
For any $a\in {\rm Hom}_{\mathcal D}(S^{\prime\prime},U)$, since ${\rm Hom}_{\mathcal D}(g,U)$ is surjective, there exists $b\in {\rm Hom}_{\mathcal D}(S^\prime,U)$ such that $a=bg$. Since $f$ is a right ${\rm add}(S)$-approximation of $U$, for $b\in {\rm Hom}_{\mathcal D}(S^\prime,U)$, there exists $c\in {\rm Hom}_{\mathcal D}(S^\prime,\mathcal S^\prime)$ such that $b=fc$. Thus $a=bg=fcg$, which implies that ${\rm Hom}_{\mathcal D}(S^{\prime\prime},f)$ is surjective. Thus ${\rm Hom}_{\mathcal D}(S^{\prime\prime},S^{\prime\prime}[1])=0$.

In summary, we have ${\rm Hom}_{\mathcal D}(S\oplus S^{\prime\prime},S\oplus S^{\prime\prime}[i])=0$ for any $i>0$. So $S\oplus S^{\prime\prime}$ is presilting. Since $S$ is silting, we know that $\mathcal D={\rm thick}(S)\subseteq {\rm thick}(S\oplus S^{\prime\prime})\subseteq \mathcal D$. So $S\oplus S^{\prime\prime}$ is a silting object in $\mathcal D$. By Theorem \ref{thmk0basis}, we can get $S^{\prime\prime}\in {\rm add}(S)$.
\end{proof}

The following proposition is a silting version of Theorem \ref{thmairbijection} (ii).
 \begin{Proposition}\label{proggvector}
Let  $\mathcal D$  be a $K$-linear, Krull-Schmidt, Hom-finite  triangulated category with a basic silting object $S=\bigoplus\limits_{i=1}^nS_i$ and $A={\rm End}_{\mathcal D}(S)^{\rm op}$. Then for any presilting object $X=X_0\oplus X_1\in S\ast S[1]$, where $X_1$ is a maximal direct summand of $X$ that belongs to ${\rm add}(S[1])$,  the $g$-vector ${\bf g}^S(X)$ of $X$ with respect to $S$ in $\mathcal D$ is equal to the $g$-vector ${\bf g}(\tilde F(X))$ of the $\tau$-rigid pair $\tilde F(X)=(F(X_0),F(X_1[-1]))$ in ${\rm mod}A$.
 \end{Proposition}
 \begin{proof}
By $X=X_0\oplus X_1\in S\ast S[1]$ and $S\ast S[1]$ is closed under direct summand, we get $X_0\in S\ast S[1]$. Note that $X_1\in{\rm add}(S[1])$. Then by
 Proposition \ref{pronocom1}, we can get the following triangles
\begin{eqnarray}
X_0[-1]\rightarrow S^1_{X_0}\rightarrow S_{X_0}^0\overset{f}{\rightarrow} X_0\;\;\;\text{and}\;\;\;
X_1[-1]\rightarrow S^1_{X_1}\rightarrow 0\rightarrow X_1,\nonumber
\end{eqnarray}
where $f$ is a right minimal ${\rm add}(S)$-approximation of $X_0$, and $S_{X_0}^0, S_{X_0}^1\in{\rm add}(S)$ and $S_{X_1}^1\cong X_1[-1]\in{\rm add}(S)$. So
$${\bf g}^S(X)={\bf g}^S(X_0)+{\bf g}^S(X_1)=({\bf g}^S(S_{X_0}^0)-{\bf g}^S(S_{X_0}^1))-{\bf g}^S(S_{X_1}^1).$$

 By applying the functor $F={\rm Hom}_{\mathcal D}(S,-)$ to the triangle $X_0[-1]\rightarrow S^1_{X_0}\rightarrow S_{X_0}^0\rightarrow X_0$, we get the minimal projective present of $F(X_0)$ in ${\rm mod}A$.
$$F(S_{X_0}^1)\rightarrow F(S_{X_0}^0)\rightarrow F(X_0)\rightarrow 0.$$
 So we have $${\bf g}(\tilde F(X))={\bf g}(F(X_0))-{\bf g}(F(X_1[-1]))=({\bf g}(F(S_{X_0}^0))-{\bf g}(F(S_{X_0}^1)))-{\bf g}(F(S_{X_1}^1)).$$
Since $S_{X_0}^0, S_{X_0}^1,S_{X_1}^1\in{\rm add}(S)$, we know that
$${\bf g}^S(S_{X_0}^0)={\bf g}(F(S_{X_0}^0)),\;\;\;{\bf g}^S(S_{X_0}^1)={\bf g}(F(S_{X_0}^1)),\;\;\;{\bf g}^S(S_{X_1}^1)={\bf g}(F(S_{X_1}^1)).$$
So ${\bf g}^S(X)={\bf g}(\tilde F(X))$.
\end{proof}

\begin{Corollary}\label{corsmu}
Let  $\mathcal D$  be a $K$-linear, Krull-Schmidt, Hom-finite  triangulated category with a basic silting object $S=\bigoplus\limits_{i=1}^nS_i$. Then the map $X\mapsto {\bf g}^S(X)$ is a injection from the set of presilting objects in $S\ast S[1]$ to $\mathbb Z^n$.
\end{Corollary}
\begin{proof}
It follows from Theorem \ref{thmijy}, Proposition \ref{proggvector} and Theorem \ref{thmairinj}.
\end{proof}

\section{Co-Bongartz completions in $2$-CY triangulated categories}

 In this section,  we fix  a $K$-linear, Krull-Schmidt, Hom-finite $2$-CY triangulated category $\mathcal C$ with a  cluster-tilting object $T$.

\subsection{Co-Bongartz completions  in $2$-CY triangulated categories}
Recall that the co-Bongartz completions  in triangulated categories are defined in Subsection 2.5. In this subsection, we will study the properties of co-Bongartz completions  in $2$-CY triangulated categories.

\begin{Proposition}\label{protrigid}
Let $\mathcal C$ be a $K$-linear, Krull-Schmidt, Hom-finite $2$-CY triangulated category. If both $U$ and $W$ are rigid objects in $\mathcal C$, then ${\rm \mathcal T}_U(W)$ is a basic rigid object in $\mathcal C$.
\end{Proposition}
\begin{proof}
 Consider the following triangle
$$\xymatrix{W[-1]\ar[r]^f&U^\prime\ar[r]^g&X_U\ar[r]^h&W},$$
where $f$ is a left minimal ${\rm add}(U)$-approximation of $W[-1]$. Then ${\rm \mathcal T}_U(W)=(U\oplus X_U)^\flat$.

Applying the functor ${\rm Hom}_{\mathcal C}(-,U)$ to the above triangle, we get the following exact sequence.
$$\xymatrix{{\rm Hom}_{\mathcal C}(U^\prime,U)\ar[rr]^{{\rm Hom}_{\mathcal C}(f,U)}&&{\rm Hom}_{\mathcal C}(W[-1],U)\ar[r]&{\rm Hom}_{\mathcal C}(X_U,U[1])\ar[r]&{\rm Hom}_{\mathcal C}(U^\prime,U[1])}=0.$$
Since $f$ is a minimal left ${\rm add} (U)$-approximation of  $W[-1]$, we have ${\rm Hom}_{\mathcal C}(f,U)$ is surjective. Thus ${\rm Hom}_{\mathcal C}(X_U,U[1])=0$. Since $\mathcal C$ is $2$-CY, we also have ${\rm Hom}_{\mathcal C}(U,X_U[1])=0$.

Now we show that ${\rm Hom}_{\mathcal C}(X_U,X_U[1])=0$. Let $a\in {\rm Hom}_{\mathcal C}(X_U,X_U[1])$. Since $$ag\in{\rm Hom}_{\mathcal C}(U^\prime,X_U[1])=0,$$ there exists $c:W\rightarrow X_U[1]$ such that $a=ch$, i.e., we have the following commutative diagram.
$$\xymatrix{W[-1]\ar[r]^f&U^\prime\ar[r]^g&X_U\ar[r]^h\ar[d]^a&W\ar@{.>}[ld]_c\\
W\ar[r]^{f[1]}&U^\prime[1]\ar[r]^{g[1]}&X_U[1]\ar[r]^{h[1]}&W[1]}.$$
 Since $(h[1])c\in{\rm Hom}_{\mathcal C}(W, W[1])=0$, there exists $d:W\rightarrow U^\prime[1]$ such that $c=(g[1])d$. Since $dh\in {\rm Hom}_{\mathcal C}(X_U,U^\prime[1])=0$, we get that $a=ch=(g[1])dh=0$. So ${\rm Hom}_{\mathcal C}(X_U,X_U[1])=0$.

Thus we have proved that  $U\oplus X_U$ is a rigid object in $\mathcal C$. So  ${\rm \mathcal T}_U(W)=(U\oplus X_U)^\flat$ is a basic rigid object in  $\mathcal C$.
\end{proof}
The following theorem can be viewed as a dual version of \cite[Definition-Proposition 4.22]{J}. The dual version is necessary for our considerations. Here we provide a different proof.
\begin{Theorem}\label{thmsantwi}
Let $\mathcal C$ be a $K$-linear, Krull-Schmidt, Hom-finite $2$-CY triangulated category with a basic cluster tilting object $T=\bigoplus\limits_{i=1}^nT_i$, and $U$ be a rigid object in $\mathcal C$. Then $\mathcal T_U(T)$ is a basic cluster tilting object in $\mathcal C$.
\end{Theorem}
\begin{proof}
Consider the following  triangle
\begin{eqnarray}\label{eqntriangle}
\xymatrix{T[-1]\ar[r]^f&U^\prime\ar[r]^g&X_U\ar[r]^h&T},
\end{eqnarray}
where $f$ is a minimal left  ${\rm add} (U)$-approximation of $T[-1]$. Then $\mathcal T_U(T)=(U\oplus X_U)^\flat$.

By Proposition \ref{protrigid}, $R:=\mathcal T_U(T)$ is a basic rigid object in $\mathcal C$. So it suffices to show that if ${\rm Hom}_{\mathcal C}(R, Y[1])=0$, then $Y\in {\rm add}(R)$.

Let $Y$ be an object in $\mathcal C$ satisfying that ${\rm Hom}_{\mathcal C}(R, Y[1])=0$.
By Proposition \ref{protilt}, for the object $Y[1]$, we have a triangle
\begin{eqnarray}\label{eqnytriangle}
\xymatrix{Y\ar[r]^{\alpha_3}&T_1^{\prime}\ar[r]^{\alpha_2}&T_0^{\prime}\ar[r]^{\alpha_1}&Y[1]},
\end{eqnarray}
where $T_1^{\prime},T_0^{\prime}\in{\rm add}(T)$.

Let $\beta_1:T_0^{\prime}[-1]\rightarrow U_0$ be a left minimal ${\rm add}(U)$-approximation of $T_0^{\prime}[-1]$ and extend it to a triangle.
$$\xymatrix{T_0^{\prime}[-1]\ar[r]^{\beta_1}&U_0\ar[r]&X_0\ar[r]&T_0^{\prime}}.$$
By the triangle (\ref{eqntriangle}) and $T_0^\prime[-1]\in{\rm add}(T[-1])$, we know that $U_0\in{\rm add}(U^\prime)$ and $X_0\in{\rm add}(X_U)$.

Let $\beta_2:=\beta_1(\alpha_2[-1]): T_1^\prime[-1]\rightarrow U_0$ and extend it to a triangle
$$\xymatrix{T_1^{\prime}[-1]\ar[r]^{\beta_2}&U_0\ar[r]&R_1^\prime \ar[r]&T_1^{\prime}}.$$

{\bf Claim I:} $Y$ is isomorphism to a direct summand of $R_1^\prime$.

{\bf Claim II:} $\beta_2=\beta_1(\alpha_2[-1]): T_1^\prime[-1]\rightarrow U_0$ is a left ${\rm add}(U)$-approximation of $T_1^\prime[-1]$;

{\bf Claim III:} $R_1^\prime\in{\rm add}(R)$;

Proof of Claim I: By $\beta_2=\beta_1(\alpha_2[-1])$ and the octahedral axiom, we have the following commutative diagram and the third column is a triangle in $\mathcal C$.
$$\xymatrix{
Y[-1]\ar[r]\ar[d]&0\ar[d]\ar[r]&Y\ar[d]\ar@{=}[r]&Y\ar[d]^{\alpha_3}\\
T_1^\prime[-1]\ar[d]^{\alpha_2[-1]}\ar[r]^{\beta_2}&U_0\ar@{=}[d]\ar[r]&R_1^\prime\ar[d]\ar[r]&T_1^\prime\ar[d]^{\alpha_2}\\
T_0^\prime[-1]\ar[d]\ar[r]^{\beta_1}&U_0\ar[d]\ar[r]&X_0\ar[d]\ar[r]&T_0^\prime\ar[d]^{\alpha_1}\\
Y\ar[r]&0\ar[r]&Y[1]\ar@{=}[r]&Y[1]
}$$
Since  ${\rm Hom}_{\mathcal C}(R,Y[1])=0$ and $X_0\in{\rm add}(X_U)\subseteq {\rm add}(R)$, we know ${\rm Hom}_{\mathcal C}(X_0,Y[1])=0$. Thus the triangle $\xymatrix{Y\ar[r]&R_1^\prime\ar[r]&X_0\ar[r]&Y[1]}$ is split and thus $Y$ is isomorphism to a direct summand of $R_1^\prime$.

Proof of Claim II: It suffices to show the following sequence is exact for any $U_1\in {\rm add}(U)$.
$$\xymatrix{{\rm Hom_{\mathcal C}}(U_0,U_1)\ar[rr]^{{\rm Hom_{\mathcal C}}(\beta_2,U_1)}&&{\rm Hom_{\mathcal C}}(T_1^\prime[-1],U_1)\ar[r]&0}.$$
Applying the functor ${\rm Hom}_{\mathcal C}(-,U_1)$ to the triangle
$$\xymatrix{Y[-1]\ar[r]^{\alpha_3[-1]}&T_1^{\prime}[-1]\ar[r]^{\alpha_2[-1]}&T_0^{\prime}[-1]\ar[r]^{\alpha_1[-1]}&Y},$$
we get the exact sequence
$$\xymatrix{{\rm Hom_{\mathcal C}}(T_0^\prime[-1],U_1)\ar[rrr]^{{\rm Hom_{\mathcal C}}(\alpha_2[-1],U_1)}&&&{\rm Hom_{\mathcal C}}(T_1^\prime[-1],U_1)\ar[rrr]^{{\rm Hom_{\mathcal C}}(\alpha_3[-1],U_1)}&&&
{\rm Hom_{\mathcal C}}(Y[-1],U_1).}$$
Since $U_1\in{\rm add}(U)\subseteq{\rm add}(R)$ and ${\rm Hom}_{\mathcal C}(R,Y[1])=0$, we know that
$${\rm Hom_{\mathcal C}}(Y[-1],U_1)\cong{\rm Hom_{\mathcal C}}(Y,U_1[1])\cong D{\rm Hom_{\mathcal C}}(U_1,Y[1])=0.$$
So we have the  following exact sequence.

$$\xymatrix{{\rm Hom_{\mathcal C}}(T_0^\prime[-1],U_1)\ar[rrr]^{{\rm Hom_{\mathcal C}}(\alpha_2[-1],U_1)}&&&{\rm Hom_{\mathcal C}}(T_1^\prime[-1],U_1)\ar[r]&0.}$$
Thus for any $b\in {\rm Hom}_{\mathcal C}(T_1^\prime[-1],U_1)$, there exists $b_0\in {\rm Hom_{\mathcal C}}(T_0^\prime[-1],U_1)$ such that $b=b_0(\alpha_2[-1])$.
Since $\beta_1:T_0^{\prime}[-1]\rightarrow U_0$ is a left minimal ${\rm add}(U)$-approximation of $T_0^{\prime}$,
we have the following exact sequence.
$$\xymatrix{{\rm Hom_{\mathcal C}}(U_0,U_1)\ar[rr]^{{\rm Hom_{\mathcal C}}(\beta_1,U_1)}&&{\rm Hom_{\mathcal C}}(T_0^\prime[-1],U_1)\ar[r]&0}.$$
Thus for $b_0\in {\rm Hom_{\mathcal C}}(T_0^\prime[-1],U_1)$, there exists $c_0\in {\rm Hom_{\mathcal C}}(U_0,U_1)$ such that $b_0=c_0\beta_1$. So $$a=b_0(\alpha_2[-1])=c_0\beta_1(\alpha_2[-1])=c_0\beta_2,$$ which implies that the following sequence
$$\xymatrix{{\rm Hom_{\mathcal C}}(U_0,U_1)\ar[rr]^{{\rm Hom_{\mathcal C}}(\beta_2,U_1)}&&{\rm Hom_{\mathcal C}}(T_1^\prime[-1],U_1)\ar[r]&0},$$
is exact. So $\beta_2=\beta_1(\alpha_2[-1]): T_1^\prime[-1]\rightarrow U_0$ is a left ${\rm add}(U)$-approximation of $T_1^\prime[-1]$.

Proof of Claim III: Let $\beta_2^\prime: T_1^\prime[-1]\rightarrow U_0^\prime$ be a left minimal ${\rm add}(U)$-approximation of $T_1^\prime[-1]$ and extend it to a triangle
$\xymatrix{T_1^{\prime}[-1]\ar[r]^{\beta_2^\prime}&U_0^\prime\ar[r]&R_1^{\prime \prime}\ar[r]&T_1^{\prime}}$.

By the triangle (\ref{eqntriangle}) and $T_1^\prime[-1]\in{\rm add}(T[-1])$, we know that $U_0^\prime\in{\rm add}(U^\prime)$ and $R_1^{\prime\prime}\in{\rm add}(X_U)\subseteq{\rm add}(R)$. Since $\beta_2$ is a  left ${\rm add}(U)$-approximation of $T_1^\prime[-1]$
and $\beta_2^\prime$ is a left minimal ${\rm add}(U)$-approximation of $T_1^\prime[-1]$, we know that the triangle
$$\xymatrix{T_1^{\prime}[-1]\ar[r]^{\beta_2}&U_0\ar[r]&R_1^\prime \ar[r]&T_1^{\prime}}$$ is the direct sum of the triangle $\xymatrix{T_1^{\prime}[-1]\ar[r]^{\beta_2^\prime}&U_0^\prime\ar[r]&R_1^{\prime \prime} \ar[r]&T_1^{\prime}}$ with a triangle of the form
$$\xymatrix{0\ar[r]&U_0^{\prime \prime}\ar[r]^{id}&U_0^{\prime \prime} \ar[r]&0},$$
where $U_0^{\prime \prime}\in {\rm add}(U)$. So $R_1^\prime\cong R_1^{\prime \prime}\oplus U_0^{\prime \prime}\in {\rm add}(R)$.

By Claim I and Claim III, we get that $Y\in{\rm add}(R)$.
So ${\rm add}(R)=\{Y\in\mathcal C|{\rm Hom}_{\mathcal C}(R,Y[1])=0\}$ and thus $R=\mathcal T_U(T)$ is a basic cluster tilting object in $\mathcal C$.
\end{proof}

The following proposition indicates that mutations can be understood as elementary co-Bongartz completions.
\begin{Proposition} \label{promutwi}Let  $T=U\oplus Z\in\mathcal C$ be a basic cluster tilting object with $Z$ indecomposable, and $T^\prime=U\oplus Y=\mu_Z(T)$. Then $T^\prime$ is the elementary co-Bongartz completion of $Y$ with respect to $T$, i.e., $\mu_Z(T)\cong\mathcal T_Y(T)$.
\end{Proposition}
\begin{proof}
Consider the following triangle
\begin{eqnarray}\label{eqnmutation}
\xymatrix{T[-1]\ar[r]^f&Y^\prime\ar[r]&X_Y\ar[r]&T},
\end{eqnarray}
where $f$ is a left minimal ${\rm add}(Y)$-approximation of $T[-1]$.
Then $\mathcal T_Y(T)=(Y\oplus X_Y)^\flat$. By Theorem \ref{thmsantwi}, $\mathcal T_Y(T)$ is a basic cluster tilting object.

Since $T^\prime=U\oplus Y$ is rigid, we know that ${\rm Hom}_{\mathcal C}(U[-1],Y)=0$. Then by the fact $T[-1]=U[-1]\oplus Z[-1]$, we know the triangle (\ref{eqnmutation}) is the direct sum of the following two triangles (as complexes).
\begin{eqnarray}
&\xymatrix{U[-1]\ar[r]&0\ar[r]&U\ar[r]^{id}&U}&,\nonumber\\
&\xymatrix{Z[-1]\ar[r]^{f_Z}&Y^{\prime\prime}\ar[r]&X_0\ar[r]&Z}&.\nonumber
\end{eqnarray}
So $Y^\prime\cong Y^{\prime\prime}$ and $X_Y\cong U\oplus X_0$. Thus $${\rm add}(\mathcal T_Y(T))={\rm add}(Y\oplus X_Y)={\rm add}(Y\oplus U\oplus X_0)\supseteq{\rm add}(U\oplus Y)= {\rm add}(T^\prime).$$
Since both $\mathcal T_Y(T)$ and $T^\prime$ are basic cluster tilting objects, we get that $$\mu_Z(T)=T^\prime\cong \mathcal T_Y(T).$$
\end{proof}

\begin{Theorem}\label{thmdirectsum}Let $\mathcal C$ be a $K$-linear, Krull-Schmidt, Hom-finite $2$-CY triangulated category with a basic cluster tilting object $T=\bigoplus\limits_{i=1}^nT_i$, and  $U=V\oplus W$ be a rigid object in $\mathcal C$. For $X\in\{U,V,W\}$, denote by $M_X=\mathcal T_X(T)$, $M_{V,W}=\mathcal T_W(M_V)$ and $M_{W,V}=\mathcal T_V(M_W)$. Then $M_{V,W}\cong M_U\cong M_{W,V}$, i.e., we have the following diagram:
$$\xymatrix{M_V\ar[d]^{\mathcal T_W}&T\ar[l]_{\mathcal T_V}\ar[d]^{\mathcal T_U}\ar[r]^{\mathcal T_W}&M_W\ar[d]^{\mathcal T_V}\\
M_{V,W}\ar@{=}[r]&M_U\ar@{=}[r]&M_{W,V}}$$

\end{Theorem}
\begin{proof}
Consider the following triangles
\begin{eqnarray}
&\xymatrix{T[-1]\ar[r]^{f_U}&U^\prime\ar[r]^{g_U}&X_U\ar[r]^{h_U}&T},&\nonumber\\
&\xymatrix{T[-1]\ar[r]^{f_V}&V^\prime\ar[r]^{g_V}&X_V\ar[r]^{h_V}&T},&\nonumber\\
&\xymatrix{M_V[-1]\ar[r]^{f_{V,W}}&W^\prime\ar[r]^{g_{V,W}}&X_{V,W}\ar[r]^{h_{V,W}}& M_V},&\nonumber
\end{eqnarray}
where $f_U$ (respectively, $f_V$) is a left minimal ${\rm add}(U)$-approximation (respectively, ${\rm add}(V)$-approximation)\\ of $T[-1]$ and $f_{V,W}$ is  a left minimal ${\rm add}(W)$-approximation of $M_V[-1]$.
We know that $$M_U=(U\oplus X_U)^\flat=(V\oplus W\oplus X_U)^\flat,\;\;M_V=(V\oplus X_V)^\flat\;\;\text{and }\;\;M_{V,W}=(W\oplus X_{V,W})^\flat$$ are basic cluster tilting objects in $\mathcal C$, by Theorem \ref{thmsantwi}.

We will show that ${\rm Hom}_{\mathcal C}(M_U,M_{V,W}[1])=0$. Since $M_U=(U\oplus X_U)^\flat=(V\oplus W\oplus X_U)^\flat$ and  $M_{V,W}=(W\oplus X_{V,W})^\flat$ are rigid, it suffices to show that $${\rm Hom}_{\mathcal C}(V,X_{V,W}[1])=0\;\;\text{ and }\;\;{\rm Hom}_{\mathcal C}(X_U,X_{V,W}[1])=0.$$

Now we show ${\rm Hom}_{\mathcal C}(V,X_{V,W}[1])=0$. Applying the functor ${\rm Hom}_{\mathcal C}(V,-)$ to the triangle $$\xymatrix{M_V[-1]\ar[r]^{f_{V,W}}&W^\prime\ar[r]&X_{V,W}\ar[r]& M_V},$$ we get the following exact sequence,

$$\xymatrix{{\rm Hom}_{\mathcal C}(V,W^\prime[1])\ar[r]&{\rm Hom}_{\mathcal C}(V,X_{V,W}[1])\ar[r]&{\rm Hom}_{\mathcal C}(V,M_V[1])}.$$
By $W^\prime\in {\rm add}(W)$ and the fact that both $U=V\oplus W$ and $M_V=(V\oplus X_V)^b$ are rigid, we know ${\rm Hom}_{\mathcal C}(V,W^\prime[1])=0$ and ${\rm Hom}_{\mathcal C}(V,M_V[1])=0$. So ${\rm Hom}_{\mathcal C}(V,X_{V,W}[1])=0$.

Now we show ${\rm Hom}_{\mathcal C}(X_U,X_{V,W}[1])=0$. Since $f_U:T[-1]\rightarrow U^\prime$ is a left  ${\rm add}(U)$-approximation of $T[-1]$ and $V^\prime\in{\rm add}(V)\subseteq {\rm add}(U)$, there exists $f: U^\prime\rightarrow V^\prime$ such that $f_{V}=f\circ f_U$. Thus we can get the following commutative diagram.

$$\xymatrix{T[-1]\ar[r]^{f_U}\ar@{=}[d]&U^\prime\ar[r]^{g_U}\ar[d]^f&X_U\ar[r]^{h_U}\ar@{.>}[d]^g&T\ar@{=}[d]\\
T[-1]\ar[r]^{f_V}&V^\prime\ar[r]^{g_V}&X_V\ar[r]^{h_V}&T}
$$
where $g:X_U\rightarrow X_V$ satisfies that $h_U=h_Vg$. Let $\alpha\in {\rm Hom}_{\mathcal C}(X_U,X_{V,W}[1])$, we show $\alpha=0$ by showing there exists a commutative diagram  of the following form.

$$\xymatrix{T[-1]\ar[r]^{f_U}&U^\prime\ar[r]^{g_U}&X_U\ar[rrrr]^{h_U}\ar[dd]^{\alpha}\ar[rrrd]^{g}&&&&T\ar@{=}[dd]\\
&&&W^{\prime\prime}[1]\ar@{.>}[ld]^{\beta_2}&&X_V\ar[ru]^{h_V}\ar@{.>}[ll]^{\beta_1}&\\
&&X_{V,W}[1]&&&&T\ar@{.>}[llll]^{\alpha_1}
}
$$

We have proved that ${\rm Hom}_{\mathcal C}(U, X_{V,W}[1])={\rm Hom}_{\mathcal C}(V, X_{V,W}[1])\oplus {\rm Hom}_{\mathcal C}(W, X_{V,W}[1])=0$, so $$\alpha g_U\in {\rm Hom}_{\mathcal C}(U^\prime, X_{V,W}[1])=0,$$ by $U^\prime \in {\rm add}(U)$. Thus there exists $\alpha_1:T\rightarrow X_{V,W}[1]$ such that $\alpha=\alpha_1h_U$.
Since $h_U=h_Vg$, we get that $\alpha=\alpha_1h_Vg=(\alpha_1h_V)g$.

Consider the morphism $\alpha_1h_V:X_V\rightarrow X_{V,W}[1]$, we claim that it factors an object $W^{\prime\prime}[1]\in{\rm add}(W[1])$. Consider the triangle,
\begin{eqnarray}\label{eqnpart}
\xymatrix{X_V[-1]\ar[r]^{\beta_1[-1]}&W^{\prime\prime}\ar[r]&X_{V,W}^\prime\ar[r]& X_V},
\end{eqnarray}
where $\beta_1[-1]:X_V[-1]\rightarrow W^{\prime\prime}$ is a left minimal ${\rm add}(W)$-approximation of $X_V[-1]$.
We know that this triangle is a direct summand of some copies of the triangle (as complexs) $$\xymatrix{M_V[-1]\ar[r]^{f_{V,W}}&W^\prime\ar[r]^{g_{V,W}}&X_{V,W}\ar[r]^{h_{V,W}}& M_V},$$
by $X_V\in{\rm add}(M_{V,W})$. So $W^{\prime\prime}\in {\rm add}(W^\prime)\subseteq {\rm add}(W)$ and $X_{V,W}^\prime \in {\rm add}(X_{V,W})$. Applying the functor ${\rm Hom}_{\mathcal C}(-,X_{V,W}[1])$ to the triangle
$$
\xymatrix{X_V\ar[r]^{\beta_1}&W^{\prime\prime}[1]\ar[r]&X_{V,W}^\prime[1]\ar[r]& X_V[1]},
$$
which is obtained from the triangle (\ref{eqnpart}) by rotations, we get the following exact sequence.
$$\xymatrix{{\rm Hom}_{\mathcal C}(W^{\prime\prime}[1],X_{V,W}[1])\ar[rrr]^{{\rm Hom}_{\mathcal C}(\beta_1,X_{V,W}[1])}&&&{\rm Hom}_{\mathcal C}(X_V,X_{V,W}[1])\ar[r]&{\rm Hom}_{\mathcal C}(X_{V,W}^\prime[-1],X_{V,W})=0}.$$
So for $\alpha_1h_U\in {\rm Hom}_{\mathcal C}(X_V,X_{V,W}[1])$, there exists $\beta_2\in {\rm Hom}_{\mathcal C}(W^{\prime\prime}[1],X_{V,W}[1])$ such that $$\alpha_1h_V=\beta_2\beta_1.$$
Thus we finished the proof of the claim before. We also get that

$$\alpha=\alpha_1h_Vg=(\alpha_1h_V)g=\beta_2\beta_1g.$$
Since $X_U,W^{\prime\prime}\in{\rm add}(M_U)$ and $M_U$ is rigid, we know that
$\beta_1g\in{\rm Hom}_{\mathcal C}(X_U,W^{\prime\prime}[1])=0$.
So $\alpha=\beta_2\beta_1g=0$ and thus ${\rm Hom}_{\mathcal C}(X_U,X_{V,W}[1])=0$.

Hence, ${\rm Hom}_{\mathcal C}(M_U,M_{V,W}[1])=0$.
Since both $M_U$ and $M_{V,W}$ are basic cluster tilting objects, we get that
$M_{V,W}\in{\rm add}(M_U)$ and $M_{U}\in {\rm add}(M_{V,W})$. Thus $M_U\cong M_{V,W}$.

Similarly, we can show $M_U\cong M_{W,V}$. This completes the proof.
\end{proof}

The following corollary follows directly from the above theorem.
\begin{Corollary}\label{cortwiseq}
Let $\mathcal C$ be a $K$-linear, Krull-Schmidt, Hom-finite $2$-CY triangulated category, and $T=\bigoplus\limits_{i=1}^nT_i$ be a basic cluster tilting object in $\mathcal C$. Let $U=\bigoplus\limits_{i=1}^sU_i$ be a basic rigid object in $\mathcal C$. Then $$\mathcal T_U(T)\cong\mathcal T_{U_{i_s}}\cdots\mathcal T_{U_{i_2}}\mathcal T_{U_{i_1}}(T),$$ where $i_1,\cdots,i_s$ is any permutation of $1,\cdots,s$.
\end{Corollary}

\begin{Corollary}
 For any two basic cluster tilting objects $T=\bigoplus\limits_{i=1}^n T_i$ and $M=\bigoplus\limits_{i=1}^n M_i$ in $\mathcal C$, we have
 $$M=\mathcal T_M(T)=\mathcal T_{M_{i_n}}\cdots\mathcal T_{M_{i_2}}\mathcal T_{M_{i_1}}(T),$$
 where $i_1,\cdots,i_n$ is any permutation of $1,\cdots,n$.
\end{Corollary}
\begin{proof}
We know that $\mathcal T_M(T)$ is a basic cluster tilting object satisfying $M\in{\rm add}(\mathcal T_M(T))$. Since $M$ is a basic cluster tilting object, we must have $M=\mathcal T_M(T)$. Then the result follows from Corollary \ref{cortwiseq}.
\end{proof}

\subsection{Compatibility between co-Bongartz completions and mutations}

\begin{Lemma}\label{lemalmost}
Let $\mathcal C$ be a $K$-linear, Krull-Schmidt, Hom-finite $2$-CY triangulated category with a basic cluster tilting object $T=\bigoplus\limits_{i=1}^nT_i$, and  $U=\bigoplus\limits_{i=1}^sU_i$ be a basic rigid object in $\mathcal C$. If there exists a triangle of the form
$$\xymatrix{U[-1]\ar[r]^f&T_U^1\ar[r]&T_U^0\ar[r]&U},$$
where $T_U^0, T_U^1\in {\rm add}(T)$ and $f$ is a left minimal ${\rm add}(T)$-approximation of $U[-1]$. Then $|U|=s\leq |T_U^0\oplus T_U^1|$. In particular, if $U$ is a almost cluster tilting object, then $T_U^0\oplus T_U^1$ is either almost cluster tilting or cluster tilting.
\end{Lemma}
\begin{proof}

Let $M_U=\bigoplus\limits_{i=1}^nM_i=\mathcal T_U(T)$, then we know that $M_U$ is a basic cluster tilting object satisfying $U\in{\rm add}(M_U)$, by Theorem \ref{thmsantwi}. Without loss of generality, we can assume that $U_i=M_i$ for $i=1,\cdots,s$. Consider the following triangle

$$\xymatrix{M_i[-1]\ar[r]^{f_i}&T_{M_i}^1\ar[r]&T_{M_i}^0\ar[r]&M_i},$$
where $f_i$ is a left minimal ${\rm add}(T)$-approximation of $M_i[-1]$. By $U=\bigoplus\limits_{i=1}^sM_i$, we know that
$$T_U^1\cong\bigoplus\limits_{j=1}^sT_{M_j}^1\;\;\text{ and }\;\;T_U^0\cong\bigoplus\limits_{j=1}^sT_{M_j}^0.$$

Let ${\rm K}_0(T)$ be the (split) Grothendieck group  of ${\rm add}(T)$ (as additive category). Since $M_U=\bigoplus\limits_{i=1}^nM_i$ and $T$ are  basic cluster tilting objects, we know that $[T_{M_1}^0]-[T_{M_1}^1],\cdots, [T_{M_n}^0]-[T_{M_n}^1]$ form  a basis of ${\rm K}_0(T)$, by \cite[Theorem 2.4]{DK}. In particular,  $[T_{M_1}^0]-[T_{M_1}^1],\cdots, [T_{M_s}^0]-[T_{M_s}^1]$ are linearly independent.
For each $j=1,\cdots,s$, $[T_{M_j}^0]-[T_{M_j}^1]=[T_{U_j}^0]-[T_{U_j}^1]$ can be linearly spanned by the set $\{[X]\in {\rm K}_0(T)|X\in{\rm add}(T_U^0\oplus T_U^1)\}$.
So  $|U|=s\leq |T_U^0\oplus T_U^1|$.  In particular, if $U$ is a almost cluster tilting object, then $T_U^0\oplus T_U^1$ is either almost cluster tilting or cluster tilting.
\end{proof}

\begin{Theorem}\label{thmsubseq}
Let  $T=W\oplus X\in\mathcal C$ be a basic cluster tilting object with $X$ indecomposable, and $T^\prime=W\oplus Y=\mu_X(T)$. Let $U$ be a rigid object in $\mathcal C$, and $M_U^{T}=\mathcal T_U(T)$ and $M_U^{T^\prime}=\mathcal T_U(T^\prime)$. Then either $M_U^{T}\cong M_U^{T^\prime}$ or $M_U^{T}$ and $ M_U^{T^\prime}$ are obtained from each other by once mutation, i.e., one of the following two diagram holds.
$$\xymatrix{T=W\oplus X\ar[r]^{\mu_X}\ar[d]^{\mathcal T_U}&T^\prime=W\oplus Y\ar[d]^{\mathcal T_U}\\
M_U^{T}\ar@{=}[r]&M_U^{T^\prime}}\hspace{15mm}
\xymatrix{T=W\oplus X\ar[r]^{\mu_X}\ar[d]^{\mathcal T_U}&T^\prime=W\oplus Y\ar[d]^{\mathcal T_U}\\
M_U^{T}\ar[r]^{\mu_{\ast}}&M_U^{T^\prime}}$$
\end{Theorem}

\begin{proof}
Consider the following triangles
\begin{eqnarray}
&\xymatrix{T[-1]\ar[r]^{f_T}&U_T\ar[r]&R\ar[r]&T},&\nonumber\\
&\xymatrix{T^\prime[-1]\ar[r]^{f_{T^\prime}}&U_{T^\prime}\ar[r]&S\ar[r]&T^\prime},&\nonumber
\end{eqnarray}
where $f_T$ (respectively, $f_{T^\prime}$) is a left minimal ${\rm add}(U)$-approximation of $T[-1]$ (respectively, $T^\prime[-1]$). We know that $M_U^{T}=\mathcal T_U(T)=(U\oplus R)^\flat$ and $M_U^{T^\prime}=\mathcal T_U(T^\prime)= (U\oplus S)^\flat$ are basic cluster tilting objects.
We consider the common direct summand of the above two triangles (as complexes) given by $W[-1]\in{\rm add}(T[-1])\cap {\rm add}(T^\prime[-1])$,
$$\xymatrix{W[-1]\ar[r]^{f_W}&U_W\ar[r]&Z\ar[r]&W},$$
where $f_W$ is a left minimal ${\rm add}(U)$-approximation of $W[-1]$. We know that $U_W\in{\rm add}(U_X)\cap{\rm add}(U_Y)$ and $Z\in{\rm add}(R)\cap{\rm add}(S)$.
Since $W$ is a almost cluster tilting object, we know that $U_W\oplus Z$ is either almost cluster tilting or cluster tilting, by Lemma \ref{lemalmost}. Since $U_W\oplus Z\in {\rm add}(M_U^{T})\cap{\rm add}(M_U^{T^\prime})$, we get that $M_U^{T}$ and $M_U^{T^\prime}$ contain either a common almost cluster tilting object  or a common cluster tilting object. So either $M_U^{T}\cong M_U^{T^\prime}$ or $M_U^{T}$ and $M_U^{T^\prime}$ are obtained from each other by once mutation.
\end{proof}

Let $T=\bigoplus\limits_{i=1}^nT_i$ be a basic cluster tilting object in $\mathcal C$, and $U\in\mathcal C$ be a rigid object. $U$ is said to be {\bf $T$-mutation-reachable} if there exists a basic cluster tilting object $M$ such that  $U\in{\rm add}(M)$ and $M$ can be obtained from $T$ by a sequence of mutations.

\begin{Theorem}\label{thmreachable}
Let $\mathcal C$ be a $K$-linear, Krull-Schmidt, Hom-finite $2$-CY triangulated category with a basic cluster tilting object $T=\bigoplus\limits_{i=1}^nT_i$.  Let $U$ be a rigid object in $\mathcal C$ and $T_U=\mathcal T_U(T)$. If $U$ is  $T$-mutation-reachable, so is $T_U$.
\end{Theorem}
\begin{proof}
Since $U$ is $T$-mutation-reachable, there exists a basic cluster tilting object $M$ such that $U\in{\rm add}(M)$ and $M$ can be obtained from $T$ by a sequence of mutations. We can assume that $$M=\mu_{X_{s}}\cdots\mu_{X_{2}}\mu_{X_{1}}(T),$$
 where $X_{i+1}$ is a indecomposable direct summand of $\mu_{X_i}\cdots\mu_{X_2}\mu_{X_1}(T)$ for $i=0,\cdots,s-1$.

We denote $T^{t_0}=T$ and $T^{t_i}=\mu_{X_i}\cdots\mu_{X_2}\mu_{X_1}(T)$ for $i=1,\cdots,s$.
Let $T^{t_i}_U=\mathcal T_U(T^{t_i})$ for $i=0,1,\cdots,s$.
Since  $U\in {\rm add}(M)={\rm add}(T^{t_s})$, we know that $T_U^{t_s}=\mathcal T_U(T^{t_s})=T^{t_s}=M$.

By Theorem \ref{thmsubseq}, we have the following diagram:
$$\xymatrix{T=T^{t_0}\ar[rr]^{\mu_{X_1}}\ar[d]^{\mathcal T_U}&&T^{t_1}\ar[rr]^{\mu_{X_2}}\ar[d]^{\mathcal T_U}&&T^{t_3}\ar[rr]^{\mu_{X_3}}
\ar[d]^{\mathcal T_U}&&\;\;\;\;\cdots& T^{t_{s-1}}
\ar[rr]^{\mu_{X_{s}}}\ar[d]^{\mathcal T_U}&&T^{t_s}=M\ar@{=}[d]^{\mathcal T_U}\\
T_U=T^{t_0}_U\ar[rr]^{\varphi_1}&&T^{t_1}_U\ar[rr]^{\varphi_2}&&T^{t_3}_U\ar[rr]^{\varphi_3}&&\;\;\;\;\cdots&
T^{t_{s-1}}_U
\ar[rr]^{\varphi_s}&&T^{t_s}_U=M
}$$
where each $\varphi_i:T_U^{t_{i-1}}\rightarrow T_U^{t_i}$ is either a isomorphism or a mutation, and in both cases, it can go back from $T_U^{t_i}$ to $T_U^{t_{i-1}}$ by a isomorphism or a mutation for $i=1,\cdots,s$. We denote $\psi_i:T_U^{t_{i}}\rightarrow T_U^{t_{i-1}}$ such that $\psi_i\varphi_i:T_U^{t_{i-1}}\rightarrow T_U^{t_{i-1}}$ is the identity. We know that each $\psi_i$ is also either a isomorphism or a mutation for $i=1,\cdots,s$.

So $T_U=T_U^{t_0}$ can be obtained from $T=T^{t_0}$ by the sequence $$(\mu_{X_1},\cdots,\mu_{X_s},id_{T^{t_s}},\psi_s,\cdots,\psi_2,\psi_1)$$ and this sequence can reduce to a sequence of mutations by deleting the isomorphisms appearing in it. This completes the proof.
\end{proof}

The following theorem is inspired by \cite[Conjecture 4.14(3)]{FZ2}.

\begin{Theorem}
Let $\mathcal C$ be a $K$-linear, Krull-Schmidt, Hom-finite $2$-CY triangulated category with two basic cluster tilting objects $T=\bigoplus\limits_{i=1}^nT_i$ and $M=\bigoplus\limits_{i=1}^nM_i$.  Let $U$ be the maximal direct summand of $T$ satisfying $U\in{\rm add}(T)\cap {\rm add}(M)$.  If $M$ is $T$-mutation-reachable, then there exists a sequence of mutations $(\mu_{Y_1},\cdots,\mu_{Y_m})$ such that $M=\mu_{Y_m}\cdots\mu_{Y_2}\mu_{Y_1}(T)$ and $Y_j\notin{\rm add}(U)$ for $j=1,\cdots,m$.
\end{Theorem}

\begin{proof}
The proof is similar with that of Theorem \ref{thmreachable}.
Since $M$ is $T$-mutation-reachable, we can assume that
$$M=\mu_{X_{s}}\cdots\mu_{X_{2}}\mu_{X_{1}}(T),$$
 where $X_{i+1}$ is a indecomposable direct summand of $\mu_{X_i}\cdots\mu_{X_2}\mu_{X_1}(T)$ for $i=0,\cdots,s-1$.

We denote $T^{t_0}=T$ and $T^{t_i}=\mu_{X_i}\cdots\mu_{X_2}\mu_{X_1}(T)$ for $i=1,\cdots,s$.
Let $T^{t_i}_U=\mathcal T_U(T^{t_i})$ for $i=0,1,\cdots,s$,
then $U$ is a direct summand of $T^{t_i}_U$ for $i=0,1,\cdots,s$.
Since  $U\in{\rm add}(T)\cap {\rm add}(M)$, we know that $T_U^{t_0}=T=T^{t_0}$ and $T_U^{t_s}=M=T^{t_s}$.

By Theorem \ref{thmsubseq}, we have the following diagram:
$$\xymatrix{T=T^{t_0}\ar[rr]^{\mu_{X_1}}\ar@{=}[d]^{\mathcal T_U}&&T^{t_1}\ar[rr]^{\mu_{X_2}}\ar[d]^{\mathcal T_U}&&T^{t_3}\ar[rr]^{\mu_{X_3}}
\ar[d]^{\mathcal T_U}&&\;\;\;\;\cdots& T^{t_{s-1}}
\ar[rr]^{\mu_{X_{s}}}\ar[d]^{\mathcal T_U}&&T^{t_s}=M\ar@{=}[d]^{\mathcal T_U}\\
T=T^{t_0}_U\ar[rr]^{\varphi_1}&&T^{t_1}_U\ar[rr]^{\varphi_2}&&T^{t_3}_U\ar[rr]^{\varphi_3}&&\;\;\;\;\cdots&
T^{t_{s-1}}_U
\ar[rr]^{\varphi_s}&&T^{t_s}_U=M
}$$
where each $\varphi_i:T_U^{t_{i-1}}\rightarrow T_U^{t_i}$ is either a isomorphism or a mutation.
If $\varphi_i:T_U^{t_{i-1}}\rightarrow T_U^{t_i}$ is a mutation, say $\varphi_i=\mu_{Y_{k_i}}$, then $Y_{k_i}\notin {\rm add}(U)$, by the fact that $U$ is a common direct summand of $T_U^{t_{i-1}}$ and $T_U^{t_{i}}$.

So when we go from $T=T_U^{t_0}$ to $M=T_U^{t_s}$ along the sequence
$(\varphi_1,\varphi_2,\cdots,\varphi_s)$, we will not do any mutation of the form $\mu_Y$ with $Y\in{\rm add}(U)$.
By deleting the isomorphisms appearing in the sequence $(\varphi_1,\varphi_2,\cdots,\varphi_s)$, we can get a sequence of mutations $(\mu_{Y_1},\cdots,\mu_{Y_m})$ such that $M=\mu_{Y_m}\cdots\mu_{Y_2}\mu_{Y_1}(T)$ and $Y_j\notin{\rm add}(U)$ for $j=1,\cdots,m$.
\end{proof}

\section{$\mathcal G$-system}

In this section we introduce $\mathcal G$-system.  The similar combinatorial results can be obtained in $\mathcal G$-system as  Section 4 without the category environment.

\subsection{Co-Bongartz completions and mutations in $\mathcal G$-system}
\begin{Definition}
Let $\mathbf{T}$ be an index set and $\mathcal G_\mathbf{T}$ be a collection of $\mathbb Z$-bases of $\mathbb Z^n$ indexed by $\mathbf{T}$, i.e., we have a map:
\begin{eqnarray}
\mathbf{T}&\longrightarrow& \mathcal G_\mathbf{T}\nonumber\\
t&\longmapsto& G_t=\{{\bf g}_{1;t},\cdots,{\bf g}_{n;t}\}.\nonumber
\end{eqnarray}
$\mathcal G_\mathbf{T}$ is called a {\bf $\mathcal G$-system at $t_0\in \mathbf{T}$} if it satisfies the following three conditions.

{\bf Mutation Condition:} For any $G_t=\{{\bf g}_{1;t},\cdots,{\bf g}_{n;t}\}\in\mathcal G_\mathbf{T}$ and $k\in\{1,\cdots,n\}$, there exists a $G_{t_1}\in\mathcal G_\mathbf{T}$ such that $G_t\cap G_{t_1}=G_t\backslash\{{\bf g}_{k;t}\}$.

{\bf Co-Bongartz Completion Condition:} For any $G_t\in \mathcal G_\mathbf{T}$ and a subset $J\subseteq\{{\bf g}_{1;t_0},\cdots,{\bf g}_{n;t_0}\}$, there exists a $G_{t^\prime}\in \mathcal G_\mathbf{T}$ satisfying the following two statements.

(a) $J\subseteq G_{t^\prime}=\{{\bf g}_{1;t^\prime},\cdots,{\bf g}_{n;t^\prime}\}$.

(b)  If ${\bf g}_{k;t}=r_{1k}^{t^\prime}{\bf g}_{1;t^\prime}+\cdots+r_{nk}^{t^\prime}{\bf g}_{n;t^\prime}$, then
 $r_{ik}^{t^\prime}\geq 0$ for any $i$ satisfying ${\bf g}_{i;t^\prime}\notin J$ and $k=1,\cdots,n$.

{\bf Uniqueness Condition:} For any $G_u,G_{v}\in \mathcal G_\mathbf{T}$, if $$\sum \limits_{i\in I}r_i{\bf g}_{i;u}=\sum\limits_{j\in I^\prime}r_j^\prime{\bf g}_{j;v},$$
for some $I, I^\prime\subseteq \{1,\cdots,n\}$ and $r_i,r_j^\prime> 0$ with $i\in I, j\in I^\prime$, then there exists a bijection $\sigma: I^\prime\rightarrow I$  such that $r_j^\prime=r_{\sigma(j)}$ and ${\bf g}_{j;v}={\bf g}_{\sigma(j);u}$ for any $j\in I^\prime$.
\end{Definition}

\begin{Remark}\label{rmkcondb}
Let $R_t^{t^\prime}=(r_{ij}^{t^\prime})$ be the transition matrix from the basis $\{{\bf g}_{1;t^\prime},\cdots,{\bf g}_{n;t^\prime}\}$ to the basis $\{{\bf g}_{1;t},\cdots,{\bf g}_{n;t}\}$, i.e.,
$$({\bf g}_{1;t},\cdots,{\bf g}_{n;t})=({\bf g}_{1;t^\prime},\cdots,{\bf g}_{n;t^\prime})R_t^{t^\prime}.$$
The statement (b) in Co-Bongartz Completion Condition is equivalent  to say that the $i$-th row vector of $R_t^{t^\prime}$ is a nonnegative vector for any $i$ with ${\bf g}_{i;t^\prime}\notin J$.
\end{Remark}

Let $\mathcal G_\mathbf{T}$ be a $\mathcal G$-system at $t_0$.
Each $G_t=\{{\bf g}_{1;t},\cdots,{\bf g}_{n;t}\}\in \mathcal G_\mathbf{T}$ is called a {\bf $g$-cluster}  and each vector in $G_t$ is called a {\bf $g$-vector}.  $\{{\bf g}_1,\cdots,{\bf g}_{n-1}\}$ is  called a {\bf almost $g$-cluster} if there exists a vector ${\bf g}_n\in\mathbb Z_n$ such that $\{{\bf g}_1,\cdots,{\bf g}_{n-1},{\bf g}_n\}\in \mathcal G_\mathbf{T}$. The $g$-cluster $G_{t_0}$ is called the {\bf initial $g$-cluster} of  $\mathcal G_\mathbf{T}$ and the vectors in $G_{t_0}$ are called the {\bf initial $g$-vectors}.
Sometimes we also use the notation $G_t^{t_0}$ to denote the $g$-cluster  at $t$, when $\mathcal G_\mathbf{T}$ is a $\mathcal G$-system at $t_0$.

\begin{Proposition-Definition}\label{protwog}
Let $\mathcal G_\mathbf{T}$ be a  $\mathcal G$-system at $t_0$, then for any  $G_t=\{{\bf g}_{1;t},\cdots,{\bf g}_{n;t}\}\in\mathcal G_\mathbf{T}$ and $k\in\{1,\cdots,n\}$, there exists a unique $G_{t_1}\in\mathcal G_\mathbf{T}$ such that $G_t\cap G_{t_1}=G_t\backslash\{{\bf g}_{k;t}\}$.

The unique $G_{t_1}\in\mathcal G_\mathbf{T}$ is called the {\bf mutation} of $G_t$ at ${\bf g}_{k;t}$, and is denoted by  $G_{t_1}=\mu_{{\bf g}_{k;t}}(G_t)$.
\end{Proposition-Definition}

\begin{proof}
The existence of $G_{t_1}$ follows from Mutation Condition  in the definition of $G$-system. Now we show the uniqueness.

For convenience, we write $G_t=\{{\bf g}_1,\cdots,{\bf g}_n\}$. Without loss of generality, we can assume that $k=n$. Suppose there exists $G_{t_1},G_{t_2}\in\mathcal G_\mathbf{T}$ such that $$G_t\cap G_{t_1}=\{{\bf g}_1,\cdots,{\bf g}_{n-1}\},\;\;\;\;G_t\cap G_{t_2}=\{{\bf g}_1,\cdots,{\bf g}_{n-1}\}.$$
Then there exists  ${\bf w}_1, {\bf w}_2\in\mathbb Z^n$ with ${\bf w}_1\neq{\bf g}_n\neq{\bf w}_2$ such that
$G_{t_1}=\{{\bf g}_1,\cdots,{\bf g}_{n-1},{\bf w}_1\}$ and $G_{t_2}=\{{\bf g}_1,\cdots,{\bf g}_{n-1},{\bf w}_2\}$.
Since both $G_{t_1}$ and $G_{t_2}$ are $\mathbb Z$-bases of $\mathbb Z^n$, there exist $k_i,k_i^\prime\in\mathbb Z$ for $i=1,\cdots,n$ such that
\begin{eqnarray}
{\bf g}_n&=&k_1{\bf g}_1+\cdots+k_{n-1}{\bf g}_{n-1}+k_n{\bf w}_1\nonumber\\
&=&k_1^\prime{\bf g}_1+\cdots+k_{n-1}^\prime{\bf g}_{n-1}+k_n^\prime{\bf w}_2.\nonumber
\end{eqnarray}
 Because ${\bf g}_1,\cdots,{\bf g}_{n-1},{\bf g}_n$ are linearly independent, we know that $k_n\neq 0$ and $k_n^\prime\neq 0$. Thus either $k_nk_n^\prime>0$ or $k_nk_n^\prime<0$.

Choose $N$ large enough such that $N>0$, $N+k_i>0$ and $N+k_i^\prime>0$ for $i=1,\cdots,n-1$.
We have the following equality.
\begin{eqnarray}
N{\bf g}_1+\cdots+N{\bf g}_{n-1}+{\bf g}_n&=&(N+k_1){\bf g}_1+\cdots+(N+k_{n-1}){\bf g}_{n-1}+k_n{\bf w}_1\nonumber\\
&=&(N+k_1^\prime){\bf g}_1+\cdots+(N+k_{n-1}^\prime){\bf g}_{n-1}+k_n^\prime{\bf w}_2.\nonumber
\end{eqnarray}

If $k_nk_n^\prime<0$, without loss of generality, we can assume that $k_n>0$ and $k_n^\prime<0$.
Then by the Uniqueness Condition in the definition of $\mathcal G$-system and the equality
$$N{\bf g}_1+\cdots+N{\bf g}_{n-1}+{\bf g}_n=(N+k_1){\bf g}_1+\cdots+(N+k_{n-1}){\bf g}_{n-1}+k_n{\bf w}_1,$$
we can get $k_1=\cdots=k_{n-1}=0$, $k_n=1$ and ${\bf g}_n={\bf w}_1$. This contradicts ${\bf g}_n\neq {\bf w}_1$.

If $k_nk_n^\prime>0$, then either $k_n>0, k_n^\prime>0$ or $k_n<0, \;k_n^\prime<0$. If $k_n>0, k_n^\prime>0$, by considering the equality
$$(N+k_1){\bf g}_1+\cdots+(N+k_{n-1}){\bf g}_{n-1}+k_n{\bf w}_1
=(N+k_1^\prime){\bf g}_1+\cdots+(N+k_{n-1}^\prime){\bf g}_{n-1}+k_n^\prime{\bf w}_2,$$
we can get $k_i=k_i^\prime$ for $i=1,\cdots,n$ and ${\bf w}_1={\bf w}_2$. So we get that $G_{t_1}=G_{t_2}$. If $k_n<0,\; k_n^\prime<0$, we need to consider the equality
$$(N+k_1){\bf g}_1+\cdots+(N+k_{n-1}){\bf g}_{n-1}+(-k_n^\prime){\bf w}_2
=(N+k_1^\prime){\bf g}_1+\cdots+(N+k_{n-1}^\prime){\bf g}_{n-1}+(-k_n){\bf w}_1,$$
and we can also get that $k_i=k_i^\prime$ for $i=1,\cdots,n$ and ${\bf w}_1={\bf w}_2$. So $G_{t_1}=G_{t_2}$. This completes the proof.
\end{proof}

\begin{Proposition-Definition}\label{prodef}
Let $\mathcal G_\mathbf{T}$ be a  $\mathcal G$-system at $t_0$. Then  for any $G_t\in \mathcal G_\mathbf{T}$ and a subset $J\subseteq\{{\bf g}_{1;t_0},\cdots,{\bf g}_{n;t_0}\}$, there exists a unique $G_{t^\prime}\in \mathcal G_\mathbf{T}$ satisfying the following two statements.

(a) $J\subseteq G_{t^\prime}=\{{\bf g}_{1;t^\prime},\cdots,{\bf g}_{n;t^\prime}\}$.

(b) If ${\bf g}_{k;t}=r_{1k}^{t^\prime}{\bf g}_{1;t^\prime}+\cdots+r_{nk}^{t^\prime}{\bf g}_{n;t^\prime}$,
then $r_{ik}^{t^\prime}\geq 0$ for any $i$ satisfying ${\bf g}_{i;t^\prime}\notin J$ and $k=1,\cdots,n$.

The unique $G_{t^\prime}$ is called the {\bf co-Bongartz completion} of $J$ with respect to $G_t$, and is denoted by  $G_{t^\prime}=\mathcal T_J(G_t)$. If $J$ contains only one element, say $J=\{{\bf g}_{j;t_0}\}$, then we call $G_{t^\prime}$  the {\bf elementary co-Bongartz completion} of $J=\{{\bf g}_{j;t_0}\}$ with respect to $G_t$.
\end{Proposition-Definition}

\begin{proof}
The existence of $G_{t^\prime}$ is just the Co-Bongartz Completion Condition  in the definition of $\mathcal G$-system. Now we give the reason for uniqueness.

Assume  that there exists $G_{t_1}, G_{t_2}\in\mathcal G(\mathcal T)$ satisfying (a) and (b).
By (a), we know  $J\subseteq \{{\bf g}_{1;t_1},\cdots,{\bf g}_{n;t_1}\}\cap\{{\bf g}_{1;t_2},\cdots,{\bf g}_{n;t_2}\}$.

We now show that $\{{\bf g}_{1;t_1},\cdots,{\bf g}_{n;t_1}\}\subseteq \{{\bf g}_{1;t_2},\cdots,{\bf g}_{n;t_2}\}$.
For each ${\bf g}_{j_0;t_1}\in G_{t_1}$, if ${\bf g}_{j_0;t_1}\in J$, clearly, we have ${\bf g}_{j_0;t_1}\in G_{t_2}$. So we mainly consider the case ${\bf g}_{j_0;t_1}\notin J$.

Since $J\subseteq \{{\bf g}_{1;t_1},\cdots,{\bf g}_{n;t_1}\}\cap\{{\bf g}_{1;t_2},\cdots,{\bf g}_{n;t_2}\}$, without loss of generality, we can assume  ${\bf g}_{j;t_1}={\bf g}_{j;t_0}={\bf g}_{j;t_2}$ for ${\bf g}_{j;t_0}\in J$.
By (b), we can get that
\begin{eqnarray}
{\bf g}_{k;t}&=&\sum\limits_{j=1}^nr_{jk}^{t_1}{\bf g}_{j;t_1}
=\sum\limits_{{\bf g}_{j;t_1}\in J}r_{jk}^{t_1}{\bf g}_{j;t_1}+\sum\limits_{{\bf g}_{j;t_1}\notin J}r_{jk}^{t_1}{\bf g}_{j;t_1}=\sum\limits_{{\bf g}_{j;t_0}\in J}r_{jk}^{t_1}{\bf g}_{j;t_0}+\sum\limits_{{\bf g}_{j;t_1}\notin J}r_{jk}^{t_1}{\bf g}_{j;t_1}\nonumber\\
{\bf g}_{k;t}&=&\sum\limits_{j=1}^nr_{jk}^{t_2}{\bf g}_{j;t_2}
=\sum\limits_{{\bf g}_{j;t_2}\in J}r_{jk}^{t_2}{\bf g}_{j;t_2}+\sum\limits_{{\bf g}_{j;t_2}\notin J}r_{jk}^{t_2}{\bf g}_{j;t_2}=
\sum\limits_{{\bf g}_{j;t_0}\in J}r_{jk}^{t_2}{\bf g}_{j;t_0}+\sum\limits_{{\bf g}_{j;t_2}\notin J}r_{jk}^{t_2}{\bf g}_{j;t_2},\nonumber
\end{eqnarray}
where $r_{jk}^{t_1},r_{jk}^{t_2}\geq 0$ for any $j$ with ${\bf g}_{j;t_1}\notin J$ and ${\bf g}_{j;t_2}\notin J$.

Denote by $R^{t_1}=(r_{ij}^{t_1})_{n\times n}$ and $R^{t_2}=(r_{ij}^{t_2})_{n\times n}$.
Since $G_t$ and $G_{t_1}$ are two $\mathbb Z$-bases of $\mathbb Z^n$ and
$$({\bf g}_{1;t},\cdots,{\bf g}_{n;t})=({\bf g}_{1;t_1},\cdots, {\bf g}_{n;t_1})R^{t_1},$$
we know that $R^{t_1}$ is full rank.
So $R^{t_1}$ does not have zero row. Thus for ${\bf g}_{j_0;t_1}\notin J$, there exists a $k_0\in\{1,\cdots,n\}$ such that
 $r_{j_0k_0}^{t_1}\neq 0$. One the other hand, we know that $r_{j_0k_0}^{t_1}\geq 0$, by ${\bf g}_{j_0;t_1}\notin J$. So $r_{j_0k_0}^{t_1}> 0$.

Choose $N$ large enough such that $N>0$, $N+r_{jk_0}^{t_1}>0$ and $N+r_{jk_0}^{t_2}>0$ for any $j=1,\cdots,n$.
We have the following equality.

\begin{eqnarray}\label{eqnnonne}
{\bf g}_{k_0;t}+\sum\limits_{{\bf g}_{j;t_1}\in J}N{\bf g}_{j;t_0}=\sum\limits_{{\bf g}_{j;t_1}\in J}(r_{jk_0}^{t_1}+N){\bf g}_{j;t_1}+\sum\limits_{{\bf g}_{j;t_1}\notin J}r_{jk_0}^{t_1}{\bf g}_{j;t_1}
=\sum\limits_{{\bf g}_{j;t_2}\in J}(r_{jk_0}^{t_2}+N){\bf g}_{j;t_2}+\sum\limits_{{\bf g}_{j;t_2}\notin J}r_{jk_0}^{t_2}{\bf g}_{j;t_2}.\nonumber
\end{eqnarray}
All the coefficients appearing in the above equality are non-negative. Then by $r_{j_0k_0}^{t_1}> 0$ and the Uniqueness Condition in the definition of the $\mathcal G$-system,  there  exists some ${\bf g}_{i_0;t_2}\notin J$ such that ${\bf g}_{j_0;t_1}={\bf g}_{i_0;t_2}\in G_{t_2}$. So  $\{{\bf g}_{1;t_1},\cdots,{\bf g}_{n;t_1}\}\subseteq \{{\bf g}_{1;t_2},\cdots,{\bf g}_{n;t_2}\}$.

Similarly, we can show that $\{{\bf g}_{1;t_2},\cdots,{\bf g}_{n;t_2}\}\subseteq\{{\bf g}_{1;t_1},\cdots,{\bf g}_{n;t_1}\}$. Thus we get that $$\{{\bf g}_{1;t_1},\cdots,{\bf g}_{n;t_1}\}=\{{\bf g}_{1;t_2},\cdots,{\bf g}_{n;t_2}\}.$$
This completes the proof.
\end{proof}

\begin{Remark}\label{rmkselft}
(i) If $G_{t^\prime}=\mathcal T_J(G_t)$ for some $J\subseteq G_{t_0}$, then $J\subseteq G_{t^\prime}$.

(ii) If $J\subseteq G_t\cap G_{t_0}$,  it is easy to see that $\mathcal T_J(G_t)=G_t$.
\end{Remark}

By the discussions before, there exists two kinds of actions ``mutations" and ``co-Bongartz completions" naturally acting on a $\mathcal G$-system. The following proposition indicates that in some special cases, mutations action can be realized as the action of elementary co-Bongartz completions.
\begin{Theorem}\label{thmtm}
Let $\mathcal G_\mathbf{T}$ be a  $\mathcal G$-system at $t_0$, $G_t\in \mathcal G_\mathbf{T}$, and $G_{t^\prime}=\mu_{{\bf g}_{n;t}}(G_t)$.
 If ${\bf g}_{n;t^\prime}\in G_{t_0}$, then
$\mu_{{\bf g}_{n;t}}(G_t)=G_{t^\prime}=\mathcal T_{{\bf g}_{n;t^\prime}}(G_t).$
\end{Theorem}

\begin{proof}
By $G_{t^\prime}=\mu_{{\bf g}_{n;t}}(G_t)$, we know ${\bf g}_{j;t^\prime}={\bf g}_{j;t}$ for $j=1,\cdots,n-1$. For convenience, we write $G_t=\{{\bf g}_1,\cdots,{\bf g}_{n-1}, {\bf g}_n\}$ and $G_{t^\prime}=\{{\bf g}_1,\cdots,{\bf g}_{n-1}, {\bf g}_n^\prime\}$. By ${\bf g}_{n;t^\prime}\in G_{t_0}$, we know ${\bf g}_n^\prime\in G_{t_0}$.

Let $G_{t_1}=\mathcal T_{{\bf g}_n^\prime}(G_t)$, we know that ${\bf g}_n^\prime$ is also a $g$-vector in $G_{t_1}$. Without loss of generality, we can assume that ${\bf g}_{n;t_1}={\bf g}_{n}^\prime$. By the definition of co-Bongartz completion, we know that  each ${\bf g}_k\in G_t$ has the form of
 $${\bf g}_k=r_{1k}{\bf g}_{1;t_1}+\cdots+r_{n-1;k}{\bf g}_{n-1;t_1}+r_{nk}{\bf g}_{n}^\prime,$$
where $r_{ik}\geq 0$ for $i=1,\cdots,n-1$.
Choose $N$ large enough such that $N>0$ and $N+r_{nk}>0$ for $k=1,\cdots,n-1$.
We consider the following equality for each $k\in\{1,\cdots,n-1\}$.
$${\bf g}_k+N{\bf g}_{n}^\prime=r_{1k}{\bf g}_{1;t_1}+\cdots+r_{n-1;k}{\bf g}_{n-1;t_1}+(r_{nk}+N){\bf g}_{n}^\prime,$$
where the coefficients appearing in the above equality are nonnegative.
Then by applying the Uniqueness Condition  in the definition of  the $\mathcal G$-system to the $g$-clusters $G_{t^\prime}$ and $G_{t_1}$,
we can obtain that ${\bf g}_k\in \{{\bf g}_{1;t_1},\cdots,{\bf g}_{n;t_1}\}$ for $k=1,\cdots,n-1$.
Since ${\bf g}_n^\prime={\bf g}_{n;t_1}$, we get that $$G_{t^\prime}=\{{\bf g}_{1},\cdots,{\bf g}_{n-1},{\bf g}_{n}^\prime\}=\{{\bf g}_{1;t_1},\cdots,{\bf g}_{n;t_1}\}=G_{t_1}.$$
So $\mu_{{\bf g}_{n;t}}(G_t)=G_{t^\prime}=G_{t_1}=\mathcal T_{{\bf g}_n^\prime}(G_t)$.
\end{proof}

\begin{Proposition}\label{promumat}
Let $\mathcal G_\mathbf{T}$ be a  $\mathcal G$-system at $t_0$,  $G_t\in\mathcal G_\mathbf{T}$, and $G_{t^\prime}=\mu_{{\bf g}_{n;t}}(G_t)$. Let $R_{t^\prime}^t=(r_{ij;t^\prime}^t)$ be  the transition matrix from the basis $G_t$ to the basis $G_{t^\prime}$, i,e,
$$({\bf g}_{1;t^\prime},\cdots,{\bf g}_{n-1;t^\prime},{\bf g}_{n;t^\prime})=({\bf g}_{1;t},\cdots,{\bf g}_{n-1;t},{\bf g}_{n;t})R_{t^\prime}^t.$$
Then $R_{t^\prime}^t$ has the form of
$$R_{t^\prime}^t=\begin{pmatrix}I_{n-1}&\alpha\\0&-1\end{pmatrix},$$
where $\alpha$ is a column vector in $\mathbb Z^{n-1}$.
If further, ${\bf g}_{n;t}\in G_{t_0}$, then $\alpha$ is a column vector in $\mathbb Z_{\geq 0}^{n-1}$.
\end{Proposition}
\begin{proof}
Since $G_{t^\prime}=\mu_{{\bf g}_{n;t}}(G_t)$, we know that ${\bf g}_{j;t^\prime}={\bf g}_{j;t}$ for $j=1,\cdots,n-1$. For convenience, we write $G_t=\{{\bf g}_1,\cdots,{\bf g}_{n-1}, {\bf g}_n\}$ and $G_{t^\prime}=\{{\bf g}_1,\cdots,{\bf g}_{n-1}, {\bf g}_n^\prime\}$.
Clearly, for $j\leq n-1$, we have $$r_{ij;t^\prime}^t=\begin{cases}1&i=j\\0&i\neq j\end{cases}.$$
Since both $G_t$ and $G_{t^\prime}$ are $\mathbb Z$-bases of $\mathbb Z^n$, we know that ${\rm det}(R_{t^\prime}^t)=\pm 1$. Thus we can get that $r_{nn;t^\prime}^t=\pm 1$.
If $r_{nn;t^\prime}^t=1$, we have
\begin{eqnarray}\label{eqnx}
{\bf g}_{n}^\prime=r_{1n;t^\prime}^t{\bf g}_1+\cdots+r_{n-1;n;t^\prime}^t{\bf g}_{n-1}+{\bf g}_n.
\end{eqnarray}
Choose $N$ large enough such that $N>0$, $N+r_{in;t^\prime}^t>0$ for $i=1,\cdots,n-1$.
We have the following equality.
$$N{\bf g}_1+\cdots+N{\bf g}_{n-1}+{\bf g}_{n}^\prime=(r_{1n;t^\prime}^t+N){\bf g}_1+\cdots+(r_{n-1;n;t^\prime}^t+N){\bf g}_{n-1}+{\bf g}_n,$$
where the coefficients appearing in the above equality are nonnegative. Applying the Uniqueness Condition  in the definition of $\mathcal G$-system for the $g$-clusters $G_t$ and $G_{t^\prime}$, we can get ${\bf g}_n={\bf g}_n^\prime$. This is a contradiction, so  $r_{nn;t^\prime}^t\neq1$. Thus $r_{nn;t^\prime}^t=-1$.

By $G_{t^\prime}=\mu_{{\bf g}_{n}}(G_t)$, we know that $G_t=\mu_{{\bf g}_n^\prime}(G_{t^\prime})$.
If ${\bf g}_{n;t}={\bf g}_n\in G_{t_0}$,  we get that
 $G_{t}=\mathcal T_{{\bf g}_n}(G_{t^\prime})$, by Theorem \ref{thmtm}. By the definition of co-Bongartz completion, we know that the coefficient in the equality (\ref{eqnx}) before ${\bf g}_i$ is nonnegative for $i=1,\cdots,n-1$. Thus if  ${\bf g}_{n;t}\in G_{t_0}$, then $\alpha$ is a column vector in $\mathbb Z_{\geq 0}^{n-1}$.
\end{proof}

Let $A=(a_{ij})_{n\times n}$ be a matrix over $\mathbb R$. We say that $A$ is {\bf row sign-coherent}, if each row vector of $A$ is either a nonpositive vector or a nonnegative vector.
\begin{Theorem}(Row sign-coherence)
Let $\mathcal G_\mathbf{T}$ be a  $\mathcal G$-system at $t_0$. For any $G_u, G_v\in\mathcal G_\mathbf{T}$, denote by $R_{v}^u=(r_{ij;v}^u)$  the transition matrix from the basis $G_u$ to the basis $G_{v}$, i,e,
$$({\bf g}_{1;v},\cdots,{\bf g}_{n;v})=({\bf g}_{1;u},\cdots,{\bf g}_{n;u})R_{v}^u.$$
Then the matrix $R_t^{t_0}$ is row sign-coherent for any $G_t\in\mathcal G_\mathbf{T}$.
\end{Theorem}

\begin{proof}
For any $k\in\{1,\cdots,n\}$, it suffices to show that the $k$-th row vector of $R_t^{t_0}$ is either a nonnegative vector or a nonpositive vector. Without loss of generality, we assume that $k=n$.

Let $J=\{{\bf g}_{1;t_0},\cdots,{\bf g}_{n-1;t_0}\}$ and $G_{t^\prime}=\mathcal T_J(G_t)$. We know that $J\subseteq G_{t^\prime}$, thus $G_{t^\prime}$ and $G_{t_0}$ have at least $n-1$ common $g$-vectors. So either $G_{t^\prime}=\mu_{{\bf g}_{n;t_0}}(G_{t_0})$ or $G_{t^\prime}=G_{t_0}$.

If $G_{t^\prime}=G_{t_0}$, then $R_t^{t^\prime}=R_t^{t_0}$. By Remark \ref{rmkcondb}, we know that the $n$-th row vector of  $R_t^{t^\prime}=R_t^{t_0}$ is a nonnegative vector.

If $G_{t^\prime}=\mu_{{\bf g}_{n;t_0}}(G_{t_0})$, we know that $R_{t^\prime}^{t_0}$ has the form of
$$R_{t^\prime}^{t_0}=\begin{pmatrix}I_{n-1}&\alpha\\0&-1\end{pmatrix},$$
where $\alpha$ is a column vector in $\mathbb Z_{\geq 0}^{n-1}$, by Proposition \ref{promumat}.
By  $G_{t^\prime}=\mathcal T_J(G_t)$, we know that $R_t^{t^\prime}$ has the form of
$$R_t^{t^\prime}=\begin{pmatrix}R^{11}_{(n-1)\times (n-1)}&R^{12}\\R^{21}&R^{22}\end{pmatrix},$$
where $\begin{pmatrix}R^{21}&R^{22}\end{pmatrix}$ is a  nonnegative row vector.

Since $G_t=G_{t^\prime}R_{t}^{t^\prime}=G_{t_0}R_{t^\prime}^{t_0}R_{t}^{t^\prime}$, we get that
$$R_t^{t_0}=R_{t^\prime}^{t_0}R_{t}^{t^\prime}=\begin{pmatrix}R^{11}+\alpha R^{21}&R^{12}+\alpha R^{22}\\-R^{21}&-R^{22}\end{pmatrix}.$$
So the $n$-th row vector of $R_t^{t^\prime}=R_t^{t_0}$  is the vector $\begin{pmatrix}-R^{21}&-R^{22}\end{pmatrix}$, which is a nonpositive row vector. This completes the proof.
\end{proof}

\begin{Theorem}
Let $\mathcal G_\mathbf{T}$ be a  $\mathcal G$-system at $t_0$, $G_t\in\mathcal G_\mathbf{T}$, and $J=J_1\sqcup J_2\subseteq G_{t_0}$. Then we have the following commutative diagram.
$$\xymatrix{G_{t_3}\ar[d]^{\mathcal T_{J_1}}&G_t\ar[l]_{\mathcal T_{J_2}}\ar[d]^{\mathcal T_{J}}\ar[r]^{\mathcal T_{J_1}}&G_{t_1}\ar[d]^{\mathcal T_{J_2}}\\
G_{t_4}\ar@{=}[r]&G_{t^\prime}\ar@{=}[r]&G_{t_2}}$$
\end{Theorem}
\begin{proof}
For any two $g$-clusters $G_u$ and $G_v$, we denote by $R_v^u=(r_{ij;v}^u)$ the transition matrix from the basis $G_u$ to the basis $G_v$, i.e.,
$$({\bf g}_{1;v},\cdots,{\bf g}_{n;v})=({\bf g}_{1;u},\cdots,{\bf g}_{n;u})R_v^u.$$
Without loss of generality, we can assume that $$J_1=\{{\bf g}_{p+1;t_0},{\bf g}_{p+2;t_0},\cdots,{\bf g}_{p+q;t_0}\}\;\;\text{and } \;\;\;J_2=\{{\bf g}_{p+q+1;t_0},{\bf g}_{p+q+2;t_0},\cdots,{\bf g}_{n;t_0}\}.$$ Thus $J=\{{\bf g}_{p+1;t_0}, {\bf g}_{p+2;t_0},\cdots,{\bf g}_{n;t_0}\}$.

We first show that $J=J_1\sqcup J_2\subseteq G_{t_4}$. Since $G_{t_4}=\mathcal T_{J_1}(G_{t_3})$, and $G_{t_3}=\mathcal T_{J_2}(G_{t})$, we know that $J_1\subseteq G_{t_4}$ and $J_2\subseteq G_{t_3}$.
Now we show that $J_2\subseteq G_{t_4}$. Consider the expansion of ${\bf g}_{k;t_3}\in J_2\subseteq G_{t_3}$ with respect to the basis $G_{t_4}$,
$${\bf g}_{k;t_3}=r_{1k;t_3}^{t_4}{\bf g}_{1;t_4}+\cdots+r_{nk;t_3}^{t_4}{\bf g}_{n;t_4}=\sum\limits_{{\bf g}_{j;t_4}\notin J_1}r_{jk;t_3}^{t_4}{\bf g}_{j;t_4}+\sum\limits_{{\bf g}_{j;t_4}\in J_1}r_{jk;t_3}^{t_4}{\bf g}_{j;t_4},$$
where the coefficients before ${\bf g}_{j;t_4}\notin J_1$ are nonnegative, by $G_{t_4}=\mathcal T_{J_1}(G_{t_3})$.
Choose $N$ large enough such that $N>0$ and $N+r_{jk;t_3}^{t_4}>0$ for any $j$ with ${\bf g}_{j;t_4}\in J_1$. We have the following equality,
$${\bf g}_{k;t_3}+\sum\limits_{{\bf g}_{j;t_4}\in J_1}N{\bf g}_{j;t_4}=\sum\limits_{{\bf g}_{j;t_4}\notin J_1}r_{jk;t_3}^{t_4}{\bf g}_{j;t_4}+\sum\limits_{{\bf g}_{j;t_4}\in J_1}(r_{jk;t_3}^{t_4}+N){\bf g}_{j;t_4},$$
where the coefficients appearing in the above equality are nonnegative.
Note that ${\bf g}_{k;t_3}\in J_2\subseteq G_{t_0}$, and $J_1\subseteq G_{t_0}$. By applying the Uniqueness Condition  in the definition of $\mathcal G$-system to $G_{t_0}$ and $G_{t_4}$, we can get that ${\bf g}_{k;t_3}\in G_{t_4}$. So $J_2\subseteq G_{t_4}$ and thus $J\subseteq G_{t_4}$.

By $J\subseteq G_{t_4}$ and $J_2\subseteq G_{t_3}$, we can assume that
${\bf g}_{i;t_0}={\bf g}_{i;t_4}$ and ${\bf g}_{j;t_0}={\bf g}_{j;t_3}$ for $i=p+1,\cdots,n$ and  $j=p+q+1,\cdots,n$.
Since $G_{t_4}=\mathcal T_{J_1}(G_{t_3})$, and $G_{t_3}=\mathcal T_{J_2}(G_{t})$, we know that $i$-th row vector of $R_{t_3}^{t_4}$ and $j$-th row vector of $R_t^{t_3}$ are nonnegative vectors for $i$ with ${\bf g}_{i;t_4}\notin J_1$ and $j$ with ${\bf g}_{j;t_3}\notin J_2$, by Remark \ref{rmkcondb}. Namely,   $i$-th row vector of $R_{t_3}^{t_4}$ and $j$-th row vector of $R_t^{t_3}$ are nonnegative vectors for $i\notin\{p+1,\cdots,p+q\}$ and $j\notin\{p+q+1,\cdots,n\}$.

Let $r:=n-p-q$, we write $R_t^{t_3}$ and $R_{t_3}^{t_4}$ as block matrices.
\begin{eqnarray}
R_{t_3}^{t_4}=\begin{pmatrix}S^{11}_{p\times p}&S^{12}_{p\times q}&S^{13}_{p\times r}\\
S^{21}_{q\times p}&S^{22}_{q\times q}&S^{23}_{q\times r}\\
S^{31}_{r\times p}&S^{32}_{r\times q}&S^{23}_{r\times r}\end{pmatrix}
\;\;\;\text{and}\;\;\;
R_{t}^{t_3}=\begin{pmatrix}R^{11}_{p\times p}&R^{12}_{p\times q}&R^{13}_{p\times r}\\
R^{21}_{q\times p}&R^{22}_{q\times q}&R^{23}_{q\times r}\\
R^{31}_{r\times p}&R^{32}_{r\times q}&R^{23}_{r\times r}\end{pmatrix}.\nonumber
\end{eqnarray}
We have $S^{11}, S^{12}, S^{13}, S^{31}, S^{32}, S^{33}$ and $R^{11},R^{12},R^{13}, R^{21}, R^{22}, R^{23}$ are nonnegative matrices.
Since ${\bf g}_{j;t_4}={\bf g}_{j;t_0}={\bf g}_{j;t_3}$ for $j=p+q+1,\cdots,n$, we get that $S^{13}=0$, $S^{23}=0$ and $S^{33}=I_{r}$.

We know that
\begin{eqnarray}
({\bf g}_{1;t},\cdots,{\bf g}_{n;t})&=&({\bf g}_{1;t_3},\cdots,{\bf g}_{n;t_3})R_{t}^{t_3}\nonumber\\
&=&({\bf g}_{1;t_4},\cdots,{\bf g}_{n;t_4})R_{t_3}^{t_4}R_{t}^{t_3}\nonumber
\end{eqnarray}
Thus $R_t^{t_4}=R_{t_3}^{t_4}R_{t}^{t_3}$. Denote by $R_{[1,p]}$ the submatrix of $R_t^{t_4}$ given by its first $p$ rows. We know that
\begin{eqnarray}
R_{[1,p]}&=&\begin{pmatrix}S^{11}R^{11}+S^{12}R^{21}+S^{13}R^{31}&S^{11}R^{12}+S^{12}R^{22}+S^{13}R^{32}&
S^{11}R^{13}+S^{12}R^{23}+S^{13}R^{33}\end{pmatrix}\nonumber\\
&=&\begin{pmatrix}S^{11}R^{11}+S^{12}R^{21}&S^{11}R^{12}+S^{12}R^{22}&
S^{11}R^{13}+S^{12}R^{23}\end{pmatrix}\nonumber
\end{eqnarray}
So $R_{[1,p]}$ is a nonnegative matrix, i.e., the $i$-th row vector of $R_t^{t_4}$ is a nonnegative vector for $i=1,\cdots, p$. Since $J=\{{\bf g}_{p+1;t_0}, {\bf g}_{p+2;t_0},\cdots,{\bf g}_{n;t_0}\}\subseteq G_{t_4}$ and ${\bf g}_{j;t_4}={\bf g}_{j;t_0}$ for $j=p+1,\cdots,n$, we get $G_{t_4}=\mathcal T_{J}(G_t)$ by the definition of co-Bongartz completion. On the other hand, $G_{t^\prime}=\mathcal T_{J}(G_t)$.
So $G_{t_4}=G_{t^\prime}$ by the uniqueness of co-Bongartz completion.

Similarly, we can show that $G_{t_2}=G_{t^\prime}$. This completes the proof.
\end{proof}

The following corollary follows directly from the above theorem, which says that any co-Bongartz completion can be obtained by a sequence of elementary co-Bongartz completions.

\begin{Corollary}\label{cormintwista}
Let $\mathcal G_\mathbf{T}$ be a  $\mathcal G$-system at $t_0$, and $J=\{{\bf g}_{1;t_0},\cdots,{\bf g}_{p;t_0}\}\subseteq G_{t_0}$. Then for any $G_t\in\mathcal G_\mathbf{T}$, we have $$\mathcal T_J(G_t)=\mathcal T_{{\bf g}_{i_p;t_0}}\cdots \mathcal T_{{\bf g}_{i_2;t_0}}\mathcal T_{{\bf g}_{i_1;t_0}}(G_t),$$
where $(i_1,\cdots,i_p)$ is any permutation of $1,\cdots, p$.
\end{Corollary}

\subsection{Compatibility between co-Bongartz completions and mutations in $\mathcal G$-system}
\begin{Theorem}\label{thmisomu}
Let $\mathcal G_\mathbf{T}$ be a  $\mathcal G$-system at $t_0$, $G_t\in\mathcal G_\mathbf{T}$, and $G_{t^\prime}=\mu_{{\bf g}_{n;t}}(G_t)$. Let $J$ be a subset of $G_{t_0}$, and
$G_u=\mathcal T_J(G_t),\; G_v=\mathcal T_J(G_{t^\prime})$. Then either $G_u=G_v$ or $G_u$ and $G_v$ are obtained each other by once mutation, i.e., exactly one of the following two diagrams holds.
$$\xymatrix{G_t\ar[r]^{\mu_{{\bf g}_{n;t}}}\ar[d]_{\mathcal T_J}&G_{t^\prime}\ar[d]^{\mathcal T_J}\\
G_u\ar@{=}[r]&G_v}\hspace{15mm}
\xymatrix{G_t\ar[r]^{\mu_{{\bf g}_{n;t}}}\ar[d]_{\mathcal T_J}&G_{t^\prime}\ar[d]^{\mathcal T_J}\\
G_u\ar[r]^{\mu_\ast}&G_v}$$
\end{Theorem}

\begin{proof}
For convenience, we write $G_t=\{{\bf g}_1,\cdots,{\bf g}_{n-1}, {\bf g}_n\}$ and $G_{t^\prime}=\{{\bf g}_1,\cdots,{\bf g}_{n-1}, {\bf g}_n^\prime\}$.

Consider the expansion of ${\bf g}_{k}\in G_t\cap G_{t^\prime}$ with respect the basis $G_u$ (respectively, $G_v$) for $k=1,\cdots,n-1$.
\begin{eqnarray}
{\bf g}_k&=&r_{1k;t}^u{\bf g}_{1;u}+\cdots+r_{nk;t}^u{\bf g}_{n;u}=\sum\limits_{{\bf g}_{j;u}\notin J}r_{jk;t}^u{\bf g}_{j;u}+\sum\limits_{{\bf g}_{j;u}\in J}r_{jk;t}^u{\bf g}_{j;u},\nonumber\\
{\bf g}_k&=&r_{1k;t^\prime}^v{\bf g}_{1;v}+\cdots+r_{nk;t^\prime}^v{\bf g}_{n;v}=\sum\limits_{{\bf g}_{j;v}\notin J}r_{jk;t^\prime}^v{\bf g}_{j;v}+\sum\limits_{{\bf g}_{j;v}\in J}r_{jk;t^\prime}^v{\bf g}_{j;v}.\nonumber
\end{eqnarray}
Since $G_u=\mathcal T_J(G_t)$ and $G_v=\mathcal T_J(G_{t^\prime})$, we know that the coefficients before
${\bf g}_{j;u}\notin J$ and before ${\bf g}_{j;v}\notin J$ are nonnegative.

Now we show that $G_u$ and $G_v$ have at least $n-1$ common $g$-vectors. Since $G_u=\mathcal T_J(G_t)$ and $G_v=\mathcal T_J(G_{t^\prime})$, we know that $J\subseteq G_u\cap G_v$.
Since ${\bf g}_1,\cdots,{\bf g}_{n-1}$ are linearly independent, we know the matrix $R_1=(r_{ij;t}^u)_{n\times(n-1)}$ has at most one zero row, i.e., $R_1$ has at least $(n-1)$ non-zero row. If $j_0$-th row of $R_1$ is nonzero, then there exists some $k_0\in\{1,\cdots,n-1\}$ such that $r_{j_0k_0;t}^u\neq 0$. If ${\bf g}_{j_0;u}\in J\subseteq G_u$, then  ${\bf g}_{j_0;u}\in J\subseteq G_v$. If ${\bf g}_{j_0;u}\notin J\subseteq G_u$, we get $r_{j_0k_0;t}^u>0$, by $r_{jk_0;t}^u\geq 0$ for any ${\bf g}_{j;u}\notin J$.
Choose $N$ large enough such that $N>0$, $N+r_{jk_0;t}^u>0$ and $N+r_{ik_0;t^\prime}^v>0$ for any $j$ and $i$ satisfying ${\bf g}_{j;u}\in J$ and ${\bf g}_{i;v}\in J$.
We consider the following equality.
$$
\sum\limits_{{\bf g}_{j;u}\notin J}r_{jk_0;t}^u{\bf g}_{j;u}+\sum\limits_{{\bf g}_{j;u}\in J}(r_{jk_0;t}^u+N){\bf g}_{j;u}=
\sum\limits_{{\bf g}_{j;v}\notin J}r_{jk_0;t^\prime}^v{\bf g}_{j;v}+\sum\limits_{{\bf g}_{j;v}\in J}(r_{jk_0;t^\prime}^v+N){\bf g}_{j;v}.$$
The coefficients appearing in the above equality are nonnegative and $r_{j_0k_0;t}^u>0$ (here ${\bf g}_{j_0;u}\notin J\subseteq G_u$). Applying the Uniqueness Condition  in the definition of $\mathcal G$-system for the $g$-clusters $G_u$ and $G_v$, we can get that ${\bf g}_{j_0;u}\in G_v$.
Since $R_1$ has at least $(n-1)$ non-zero row, we obtain that $G_v$ contains at least $n-1$ vectors in $G_u$.

If $G_u$ and $G_v$ contain $n-1$ common $g$-vectors, then $G_u$ and $G_v$ are obtained from each other by once mutation. If $G_u$ and $G_v$ contain $n$ common $g$-vectors, then $G_u=G_v$.
\end{proof}

Let $\mathcal G_\mathbf{T}$ be a  $\mathcal G$-system at $t_0$, and $G_t\in \mathcal G_\mathbf{T}$, and
$U$ be a subset of $g$-vectors of $\mathcal G_\mathbf{T}$. We say that $U$ is {\bf $G_t$-mutation-reachable}, if there exists  $G_{u}\in \mathcal G_\mathbf{T}$ such that $U\subseteq G_{u}$ and $G_{u}$ can be obtained from $G_t$ by a sequence of mutations.

\begin{Theorem}\label{thmtwrea}
Let $\mathcal G_\mathbf{T}$ be a  $\mathcal G$-system at $t_0$, $G_t\in \mathcal G_\mathbf{T}$, and $J\subseteq G_{t_0}$.  If $J$ is  $G_t$-mutation reachable, so is $G_{t^\prime}:=\mathcal T_J(G_t)$.
\end{Theorem}
\begin{proof}
Since $J$ is $G_t$-mutation-reachable, there exists  $G_u\in \mathcal G_\mathbf{T}$ such that $J\subseteq G_u$ and $G_u$ can be obtained from $G_t$ by a sequence of mutations. We can assume that $$G_u=\mu_{{\bf g}_{s}}\cdots\mu_{{\bf g}_{2}}\mu_{{\bf g}_{1}}(G_t),$$
 where ${\bf g}_{i+1}$ is a $g$-vector in $\mu_{{\bf g}_i}\cdots\mu_{{\bf g}_2}\mu_{{\bf g}_1}(G_t)$ for $i=0,\cdots,s-1$.

We denote $G_{u_0}=G_t$ and $G_{u_i}=\mu_{{\bf g}_i}\cdots\mu_{{\bf g}_2}\mu_{{\bf g}_1}(G_t)$ for $i=1,\cdots,s$.
Let $G_{u_i^\prime}=\mathcal T_J(G_{u_i})$ for  $i=0,1,\cdots,s$.
Since  $J\subseteq G_u=G_{u_s}$, we know that $G_{u_s^\prime}=\mathcal T_J(G_{u_s})=G_{u_s}$, by Remark \ref{rmkselft}. By Theorem \ref{thmisomu}, we have the following diagram:
$$\xymatrix{G_t=G_{u_0}\ar[rr]^{\mu_{{\bf g}_1}}\ar[d]^{\mathcal T_J}&&G_{u_1}\ar[rr]^{\mu_{{\bf g}_2}}\ar[d]^{\mathcal T_J}&&G_{u_2}\ar[rr]^{\mu_{{\bf g}_3}}
\ar[d]^{\mathcal T_J}&&\;\;\;\;\cdots&G_{u_{s-1}}
\ar[rr]^{\mu_{{\bf g}_s}}\ar[d]^{\mathcal T_J}&&G_{u_s}=G_u\ar@{=}[d]^{\mathcal T_J}\\
G_{t^\prime}=G_{u_0^\prime}\ar[rr]^{\varphi_1}&&G_{u_1^\prime}\ar[rr]^{\varphi_2}&&G_{u_2^\prime}\ar[rr]^{\varphi_3}&&\;\;\;\;\cdots&
G_{u_{s-1}^\prime}
\ar[rr]^{\varphi_s}&&G_{u_s^\prime}=G_{u_s}
}$$
where each $\varphi_i:G_{u_{i-1}^\prime}\rightarrow G_{u_i^\prime}$ is either the identity or a mutation, and in both cases, it can go back from $G_{u_i^\prime}$ to $G_{u_{i-1}^\prime}$ by the identity or a mutation for $i=1,\cdots,s$. Denote by $\psi_i:G_{u_i^\prime}\rightarrow G_{u_{i-1}^\prime}$ satisfying that $\psi_i\varphi_i:G_{u_{i-1}^\prime}\rightarrow G_{u_{i-1}^\prime}$ is the identity. We know that each $\psi_i$ is also either the identity or a mutation for $i=1,\cdots,s$.

So $G_{t^\prime}=G_{u_0^\prime}$ can be obtained from $G_t=G_{u_0}$ by the sequence $$(\mu_{{\bf g}_1},\cdots,\mu_{{\bf g}_s},id_{G_{u_s}},\psi_s,\cdots,\psi_2,\psi_1)$$ and this sequence can reduce to a sequence of mutations by deleting the identities appearing in it. This completes the proof.
\end{proof}

The following theorem is inspired by \cite[Conjecture 4.14(3)]{FZ2}.

\begin{Theorem}\label{thmconnect}

Let $\mathcal G_\mathbf{T}$ be a  $\mathcal G$-system at $t_0$, $G_t\in \mathcal G_\mathbf{T}$, and $J=G_t\cap G_{t_0}$.  If $G_t$ is $G_{t_0}$-mutation-reachable,  then there exists a sequence of mutations $(\mu_{{\bf g}_1^\prime},\cdots,\mu_{{\bf g}_m^\prime})$ such that $G_t=\mu_{{\bf g}_m^\prime}\cdots\mu_{{\bf g}_2^\prime}\mu_{{\bf g}_1^\prime}(G_{t_0})$ and ${\bf g}_j^\prime\notin J$ for $j=1,\cdots,m$.
\end{Theorem}

\begin{proof}
The proof is similar with that of Theorem \ref{thmtwrea}.
Since $G_t$ is $G_{t_0}$-mutation-reachable, we can assume that
$$G_t=\mu_{{\bf g}_s}\cdots\mu_{{\bf g}_2}\mu_{{\bf g}_1}(G_{t_0}),$$
 where ${\bf g}_{i+1}$ is a $g$-vector in $\mu_{{\bf g}_i}\cdots\mu_{{\bf g}_2}\mu_{{\bf g}_1}(G_{t_0})$ for $i=0,\cdots,s-1$.

We denote  $G_{t_i}=\mu_{{\bf g}_i}\cdots\mu_{{\bf g}_2}\mu_{{\bf g}_1}(G_{t_0})$ for $i=1,\cdots,s$.
Let $G_{t_i^\prime}=\mathcal T_J(G_{t_i})$ for  $i=0,1,\cdots,s$, then $J\subseteq G_{t_i^\prime}$ for $i=0,1,\cdots, s$.

Since  $J\subseteq G_t\cap G_{t_0}=G_{t_s}\cap G_{u_0}$, we know that
$G_{t_s^\prime}=\mathcal T_J(G_{t_s})=G_{t_s}=G_t$ and $G_{t_0^\prime}=\mathcal T_J(G_{t_0})=G_{t_0}$, by Remark \ref{rmkselft}.

By Theorem \ref{thmisomu}, we have the following diagram:
$$\xymatrix{G_{t_0}\ar[rr]^{\mu_{{\bf g}_1}}\ar@{=}[d]^{\mathcal T_J}&&G_{t_1}\ar[rr]^{\mu_{{\bf g}_2}}\ar[d]^{\mathcal T_J}&&G_{t_2}\ar[rr]^{\mu_{{\bf g}_3}}
\ar[d]^{\mathcal T_J}&&\;\;\;\;\cdots&G_{t_{s-1}}
\ar[rr]^{\mu_{{\bf g}_s}}\ar[d]^{\mathcal T_J}&&G_{t_s}=G_t\ar@{=}[d]^{\mathcal T_J}\\
G_{t_0}=G_{t_0^\prime}\ar[rr]^{\varphi_1}&&G_{t_1^\prime}\ar[rr]^{\varphi_2}&&G_{t_2^\prime}
\ar[rr]^{\varphi_3}&&\;\;\;\;\cdots&
G_{t_{s-1}^\prime}
\ar[rr]^{\varphi_s}&&G_{t_s^\prime}=G_t
}$$
where each $\varphi_i:G_{t_{i-1}^\prime}\rightarrow G_{t_i^\prime}$ is either the identity or a mutation.
If $\varphi_i:G_{t_{i-1}^\prime}\rightarrow G_{t_i^\prime}$ is a mutation, say $\varphi_i=\mu_{{\bf g}_{k_i}^\prime}$, then ${\bf g}_{k_i}^\prime \notin J$, by the fact that $J\subseteq G_{t_{i-1}^\prime}\cap G_{t_i^\prime}$.

So when we go from $G_{t_0}$ to $G_t$ along the sequence
$(\varphi_1,\varphi_2,\cdots,\varphi_s)$, we will not do any mutation of the form $\mu_{{\bf g}^\prime}$ with ${\bf g}^\prime\in J$.
By deleting the identities  appearing in the sequence $(\varphi_1,\varphi_2,\cdots,\varphi_s)$, we can get a sequence of mutations $(\mu_{{\bf g}_1^\prime},\cdots,\mu_{{\bf g}_m^\prime})$ such that $G_t=\mu_{{\bf g}_m^\prime}\cdots\mu_{{\bf g}_2^\prime}\mu_{{\bf g}_1^\prime}(G_{t_0})$ and ${\bf g}_j^\prime\notin J$ for $j=1,\cdots,m$.
\end{proof}

\section{$\mathcal G$-systems from triangulated categories and module categories}
In this section we show that $\mathcal G$-systems naturally arise  from triangulated categories and module categories, and the mutations and co-Bongartz completions in $\mathcal G$-systems are compatible with those in categories.

\subsection{$\mathcal G$-systems from cluster tilting theory}
In this subsection, we fix a $K$-linear, Krull-Schmidt, Hom-finite $2$-CY triangulated category $\mathcal C$ with a basic cluster tilting object $R=\bigoplus\limits_{i=1}^nR_i$.

For a  matrix $A=(a_{ij})_{n\times n}$ and $I,J\subseteq\{1,\cdots,n\}$, we denote by $A|_{I\times J}$ the submatrix of $A$ given by entries $a_{ij}$ with $i\in I$ and $j\in J$.
\begin{Proposition}\label{prosantwi}
 Let $T=\bigoplus\limits_{i=1}^nT_i$ and $R=\bigoplus\limits_{i=1}^nR_i$  be two basic cluster tilting objects in $\mathcal C$. Let  $U=\bigoplus\limits_{j=p+1}^nR_j$, $T^\prime=\mathcal T_U(T)$, and denote  $I=\{1,\cdots,p\}$. Then
 $G_T^R=G_{T^\prime}^RG_T^{T^\prime}$ and
 $G_T^{R}|_{I\times [1,n]}$ has the form of
$$ G_T^{R}|_{I\times [1,n]}=G_{T^\prime}^{R}|_{I\times I}Q,$$
for some $Q\in M_{|I|\times n}({\mathbb Z}_{\geq 0})$.
\end{Proposition}
\begin{proof}
Consider the following triangle,
$$\xymatrix{T[-1]\ar[r]^f&U^\prime\ar[r]&X_U\ar[r]&T
},$$
where $f$ is a left minimal ${\rm add}(U)$-approximation of $T[-1]$. Then $T^\prime=\mathcal T_U(T)=(U\oplus X_U)^\flat$ is a basic cluster tilting object in $\mathcal C$, by Theorem \ref{thmsantwi}. We know that $U$ is direct summand of $T^\prime$. Without  loss of generality, we can assume that $U=\bigoplus\limits_{j=p+1}^nT_j^\prime$.

For each $T_i[-1]$, there exists a triangle

$$\xymatrix{T_i[-1]\ar[r]^{f_i}&U_i^\prime\ar[r]&X_i\ar[r]&T_i
},$$
which is direct summand of the triangle (as a complex) $\xymatrix{T[-1]\ar[r]^f&U^\prime\ar[r]&X_U\ar[r]&T
}$.
In particular, $U_i^\prime=\bigoplus\limits_{j=p+1}^n (T_j^{\prime})^{b_{ji}}\in {\rm add}(U)$ and $X_i=\bigoplus\limits_{l=1}^n (T_l^{\prime})^{a_{li}}\in {\rm add}(X_U) \subseteq {\rm add}(T^\prime)$.
Thus $${\bf g}^{T^\prime}(T_i)=(a_{1i},\cdots,a_{pi},a_{p+1,i}-b_{p+1,i},\cdots,a_{ni}-b_{ni})^{\rm T}.$$

Since $U_i^\prime[1]\in {\rm add}(U[1])\subseteq {\rm add}(R[1])$, we know that in the triangle
$$\xymatrix{U_i^\prime\ar[r]&X_i\ar[r]&T_i\ar[r]^{f_i[1]}&U_i^\prime[1]
},$$
 $f_i[1]$ factors though an object in ${\rm add}(R[1])$. By Proposition \ref{prodk} (iii), we get that
\begin{eqnarray}
{\bf g}^{R}(T_i)&=&{\bf g}^{R} (X_i)-{\bf g}^{R}(U_i^\prime)={\bf g}^{R}(\bigoplus\limits_{l=1}^n (T_l^{\prime})^{a_{li}})-
{\bf g}^{R}(\bigoplus\limits_{j=p+1}^n (T_j^{\prime})^{b_{ji}})\nonumber\\
&=&\sum\limits_{j=1}^p a_{ji}{\bf g}^{R}(T_j^\prime)+\sum_{l=p+1}^n(a_{li}-b_{li}){\bf g}^{R}(T_j^\prime).\nonumber
\end{eqnarray}
This just means that $G_T^{R}=G_{T^\prime}^{R}G_T^{T^\prime}$.

By the fact that $\bigoplus\limits_{j=p+1}^nR_j=U=\bigoplus\limits_{j=p+1}^nT_j^\prime$, we know $G_{T^\prime}^{R}$ has the form of
$$G_{T^\prime}^{R}=\begin{pmatrix}G_{T^\prime}^{R}|_{I\times I}&0\\ \ast&I_{n-p}
\end{pmatrix}.$$
Set $Q=(a_{ji})$, then we know $G_T^{T^\prime}|_{I\times [1,n]}=Q\in M_{|I|\times n}({\mathbb Z}_{\geq 0})$, by  $a_{ji}\geq 0$ for $j=1,\cdots,p$ and $i=1,\cdots,n$.
 Thus we get the following equality $$G_T^{R}|_{I\times [1,n]}=G_{T^\prime}^{R}|_{I\times [1,n]}G_T^{T^\prime}=\begin{pmatrix}G_{T^\prime}^{R}|_{I\times I}&0
\end{pmatrix}G_T^{T^\prime}=G_{T^\prime}^{R}|_{I\times I}Q,$$ where $Q\in M_{|I|\times n}({\mathbb Z}_{\geq 0})$.
\end{proof}

\begin{Theorem}
Let $\mathcal C$ be a $K$-linear, Krull-Schmidt, Hom-finite $2$-CY triangulated category with a basic cluster tilting object $R=\bigoplus\limits_{i=1}^nR_i$ and  $\mathbf{T}$ be the set of basic cluster tilting objects of $\mathcal C$. Then $\mathcal G^R_{\mathbf{T}}:=\{G_T^R|T\in\mathbf{T}\}$ is a $\mathcal G$-system at $R$ and
 the following commutative diagrams hold.

$$\xymatrix{V\ar[rr]^{\mu_{k}}\ar@{<->}[d]&&W\ar@{<->}[d]\\
G_V^R\ar[rr]^{\mu_{k}}&&G_W^R
}\;\;\;\;\;\;\;\;\;\;\xymatrix{T\ar[rr]^{\mathcal T_U}\ar@{<->}[d]&&T^\prime\ar@{<->}[d]\\
G_T^R\ar[rr]^{\mathcal T_J}&&G_{T^\prime}^R
}$$
where $V,W,T,T^\prime\in\mathbf{T}$ and $U$ is a direct summand of $R$ and $$J=\{{\bf g}^R(U_i)|U_i \text{ is an indecomposable direct summand of }U\}.$$
\end{Theorem}
\begin{proof}
 By Proposition \ref{prodk} (i), the column vectors of $G_T^R$ form a  $\mathbb Z$-basis of $\mathbb Z^n$ for any $T\in\mathbf{T}$.

By Theorem \ref{thm2tilt} and Proposition \ref{prodk} (ii), $G_\mathbf{T}^R$ satisfies the Mutation Condition  in the definition of $\mathcal G$-system. The Uniqueness Condition follows from Proposition \ref{prodk} (ii).

We mainly show that the Co-Bongartz Completion Condition holds. For any $T\in\mathbf{T}$ and  any direct summand $U$ of $R$, set $T^\prime=\mathcal T_U(T)$. By Proposition \ref{prosantwi}, we know that
if ${\bf g}^R(T_k)=r_{1k}^{t^\prime}{\bf g}^R(T_1^\prime)+\cdots+r_{nk}^{t^\prime}{\bf g}^R(T_n^\prime)$,
then $r_{ik}^{t^\prime}\geq 0$ for any $T_i^\prime\notin{\rm add}(U)$ and $k=1,\cdots,n$. Thus the  Co-Bongartz Completion Condition holds. So  $\mathcal G^R_{\mathbf{T}}$ is a $\mathcal G$-system at $R$. The commutative diagrams are easy to check.
\end{proof}

\subsection{$\mathcal G$-systems from silting theory}
In this subsection, we fix a $K$-linear, Krull-Schmidt, Hom-finite  triangulated category $\mathcal D$ with a basic silting object $S=\bigoplus\limits_{i=1}^nS_i$.

\begin{Lemma}\label{lemclosed}
Let  $\mathcal D$  be a $K$-linear, Krull-Schmidt, Hom-finite  triangulated category with a basic silting object $S=\bigoplus\limits_{i=1}^nS_i$. Let $U$ and $T$ be two objects in $S\ast S[1]$, then ${\rm Hom}_{\mathcal D}(T, U[i])=0$ for any $i\geq 2$.
\end{Lemma}
\begin{proof}
Since $U\in S\ast S[1]$, there exists a triangle of the form
$$X\rightarrow U\rightarrow Y\rightarrow X[1],$$
where $X\in{\rm add}(S)$ and $Y\in{\rm add}(S[1])$. By Proposition \ref{projasso} and $T\in S\ast S[1]$, we know that
${\rm Hom}_{\mathcal D}(T,S[j+1])=0$ for any $j>0$. In particular, we have ${\rm Hom}_{\mathcal D}(T,X[i])=0$ and ${\rm Hom}_{\mathcal D}(T,Y[i])=0$ for any $i\geq 2$.

For each $i\geq 2$, applying the functor ${\rm Hom}_{\mathcal D}(T[-i],-)$ to the triangle $X\rightarrow U\rightarrow Y\rightarrow X[1],$ we get the following exact sequence.
$$0={\rm Hom}_{\mathcal D}(T[-i],X)\rightarrow {\rm Hom}_{\mathcal D}(T[-i],U)\rightarrow {\rm Hom}_{\mathcal D}(T[-i],Y)=0.$$
So ${\rm Hom}_{\mathcal D}(T, U[i])=0$ for any $i\geq 2$.
\end{proof}

\begin{Theorem}\label{thmstwi}
Let  $\mathcal D$  be a $K$-linear, Krull-Schmidt, Hom-finite  triangulated category with a basic silting object $S=\bigoplus\limits_{i=1}^nS_i$. Let $U\in {\rm add}(S)$ and $T\in S\ast S[1]$ be a basic silting object, then $\mathcal T_U(T)$ is a basic silting object and   $\mathcal T_U(T)\in S\ast S[1]$.
\end{Theorem}

\begin{proof}
Consider the following triangle
\begin{eqnarray}\label{eqnsilt}
\xymatrix{T[-1]\ar[r]^f&U^\prime\ar[r]^g&X_U\ar[r]^h&T},
\end{eqnarray}
where $f$ is a left minimal ${\rm add}(U)$-approximation of $T[-1]$. We know that $\mathcal T_U(T)=(U\oplus X_U)^\flat$.
Since $T$ is a silting object, it is easy to see that ${\rm thick}(\mathcal T_U(T))=\mathcal D$, by the triangle (\ref{eqnsilt}).

We first show that ${\rm Hom}_{\mathcal D}(X_U,U[1])=0$. Applying the functor ${\rm Hom}_{\mathcal D}(-,U)$ to the triangle (\ref{eqnsilt}), we get the following exact sequence
$$\xymatrix{{\rm Hom}_{\mathcal D}(U^\prime,U)\ar[rr]^{{\rm Hom}_{\mathcal D}(f,U)}&&{\rm Hom}_{\mathcal D}(T[-1],U)\ar[r]&{\rm Hom}_{\mathcal D}(X_U[-1],U)\ar[r]&{\rm Hom}_{\mathcal D}(U^\prime[-1],U)=0
}.$$
Since $f$ is a left minimal ${\rm add}(U)$-approximation of $T[-1]$, we know that ${\rm Hom}_{\mathcal D}(f,U)$ is surjective, thus ${\rm Hom}_{\mathcal D}(X_U[-1],U)\cong{\rm Hom}_{\mathcal D}(X_U,U[1])=0$.

We show that for each $i\geq 2$, ${\rm Hom}_{\mathcal D}(X_U,U[i])=0$. By Lemma   \ref{lemclosed}, we know that  ${\rm Hom}_{\mathcal D}(T, U[i])=0$ for $i\geq 2$. Since $U$ is presilting, we can get ${\rm Hom}_{\mathcal D}(U, U[j])=0$ for $j\geq 0$. In particular, we have ${\rm Hom}_{\mathcal D}(U^\prime, U[i])=0$ for $i\geq 2$. By the triangle (\ref{eqnsilt}), $X_U$ is an extension of $U^\prime$ and $X_U$, so ${\rm Hom}_{\mathcal D}(X_U,U[i])=0$ for any $i\geq 2$.

We show that ${\rm Hom}_{\mathcal D}(U\oplus X_U, T[i])=0$ for any $i>0$. By $T\in S\ast S[1]$ and Proposition \ref{projasso}, we know that ${\rm Hom}_{\mathcal D}(S,T[i])=0$ for any $i>0$. In particular, we get that  ${\rm Hom}_{\mathcal D}(U, T[i])=0$ and ${\rm Hom}_{\mathcal D}(U^\prime, T[i])=0$ for any $i>0$, by $U^\prime\in {\rm add}(U)\subseteq {\rm add}(S)$. For each $i>0$, applying the functor ${\rm Hom}_{\mathcal D}(-, T[i])$ to the triangle (\ref{eqnsilt}), we get the following exact sequence.
$$\xymatrix{0={\rm Hom}_{\mathcal D}(T,T[i])\ar[r]&{\rm Hom}_{\mathcal D}(X_U,T[i])\ar[r]&{\rm Hom}_{\mathcal D}(U^\prime, T[i])}=0.$$ So  ${\rm Hom}_{\mathcal D}(X_U, T[i])=0$ and thus ${\rm Hom}_{\mathcal D}(U\oplus X_U, T[i])=0$ for any $i>0$.

Now for each $i>0$, we show that ${\rm Hom}_{\mathcal D}(U\oplus X_U,X_U[i])=0$. Let $a\in {\rm Hom}_{\mathcal D}(U\oplus X_U,X_U[i])$, then $$(h[i])a\in {\rm Hom}_{\mathcal D}(U\oplus X_U, T[i])=0.$$
So there exists $b:U\oplus X_U\rightarrow U^\prime[i]$ such that $a=(g[i])b$, i.e., we have the following diagram.
$$\xymatrix{&&U\oplus X_U\ar[d]^a\ar@{.>}[ld]_b&\\
T[-1+i]\ar[r]^{f[i]}&U^\prime[i]\ar[r]^{g[i]}&X_U[i]\ar[r]^{h[i]}&T[i]
}$$
Since $b\in {\rm Hom}_{\mathcal D}(U\oplus X_U,U^\prime[i])=0$, we get $a=(g[i])b=0$. So  ${\rm Hom}_{\mathcal D}(U\oplus X_U,X_U[i])=0$.

Thus we have proved that ${\rm Hom}_{\mathcal D}(U\oplus X_U,(U\oplus X_U)[i])=0$ for any $i>0$. So $\mathcal T_U(T)=(U\oplus X_U)^\flat$ is a basic presilting object in $\mathcal D$. By ${\rm thick}(\mathcal T_U(T))=\mathcal D$, we obtain that $\mathcal T_U(T)$ is a basic silting object.

It remains to show that  $\mathcal T_U(T)\in S\ast S[1]$. By the triangle (\ref{eqnsilt}), we know that $X_U$ is an extension of $U^\prime$ and $T$. Then by Corollary \ref{corclosed} and the fact that $U^\prime, T\in S\ast S[1]$, we get that $X_U\in S\ast S[1]$. So $U\oplus X_U\in  S\ast S[1]$.
By ${\rm Hom}_{\mathcal D}(S,S[1])=0$ and \cite[Proposition 2.1(1)]{IY}, we know that $S\ast S[1]$ is closed under direct summand. So $\mathcal T_U(T)=(U\oplus X_U)^\flat\in S\ast S[1]$.
\end{proof}

\begin{Theorem}\label{thmgsys}
Let $\mathcal D$ be a $K$-linear, Krull-Schmidt, Hom-finite  triangulated category with a basic cluster tilting object $S=\bigoplus\limits_{i=1}^nS_i$, and  $\mathbf{T}$ be the set of basic silting objects in $S\ast S[1]$. Then $\mathcal G^S_{\mathbf{T}}:=\{G_T^S|T\in\mathbf{T}\}$ is a $\mathcal G$-system at $S$ and
 the following commutative diagrams hold.

$$\xymatrix{V\ar[rr]^{\mu_{k}}\ar@{<->}[d]&&W\ar@{<->}[d]\\
G_V^S\ar[rr]^{\mu_{k}}&&G_W^S
}\;\;\;\;\;\;\;\;\;\;\xymatrix{T\ar[rr]^{\mathcal T_U}\ar@{<->}[d]&&T^\prime\ar@{<->}[d]\\
G_T^S\ar[rr]^{\mathcal T_J}&&G_{T^\prime}^S
}$$
where $V,W,T,T^\prime\in\mathbf{T}$ and $U$ is a direct summand of $S$ and $$J=\{{\bf g}^S(U_i)|U_i \text{ is an indecomposable direct summand of }U\}.$$
\end{Theorem}

\begin{proof}
 By Remark \ref{rmksbasis}, the column vectors of $G_T^S$ form a  $\mathbb Z$-basis of $\mathbb Z^n$ for any $T\in\mathbf{T}$.

By Corollary \ref{corsmutation} and Corollary \ref{corsmu}, $\mathcal G_\mathbf{T}^S$ satisfies the Mutation Condition and Uniqueness Condition in the definition of $\mathcal G$-system.

Now we show that the Co-Bongartz Completion Condition holds.
For any $T\in\mathbf{T}$ and  any direct summand $U$ of $S$, consider the following triangle
\begin{eqnarray}\label{eqnk0}
\xymatrix{T[-1]\ar[r]^f&U^\prime\ar[r]^g&X_U\ar[r]^h&T},
\end{eqnarray}
where $f$ is a left minimal ${\rm add}(U)$-approximation of $T[-1]$. By Theorem \ref{thmstwi}, we know that $T^\prime:=\mathcal T_U(T)=(U\oplus X_U)^\flat\in\mathbf{T}$.

Consider the direct summand of the triangle (\ref{eqnk0}) (as a complex) given by $T_k[-1]$
$$\xymatrix{T_k[-1]\ar[r]^{f_k}&U_k^\prime\ar[r]&X_k\ar[r]&T_k}$$
where $f_k$ is a left minimal ${\rm add}(U)$-approximation of $T_k[-1]$. We know that $$[T_k]=[X_k]-[U_k^\prime]\in {\rm K}_0(\mathcal D),$$
so ${\bf g}^S(T_k)={\bf g}^S(X_k)-{\bf g}^S(U_k^\prime).$
We can assume that $X_k=\bigoplus\limits_{i=1}^n(T_i^\prime)^{r_i}$ and $U_k^\prime=\bigoplus\limits_{i=1}^n(T_i^\prime)^{r_i^\prime}$, where $r_i,r_i^\prime\geq 0$ and $r_i^\prime=0$ for any $T_i^\prime\notin {\rm add}(U)$. So
 \begin{eqnarray}
 {\bf g}^S(T_k)&=&(r_1-r_1^\prime){\bf g}^S(T_1^\prime)+\cdots+(r_n-r_n^\prime){\bf g}^S(T_n^\prime)\nonumber\\
 &=&\sum\limits_{T_i^\prime\notin {\rm add}(U)}r_i{\bf g}^S(T_i^\prime)+\sum\limits_{T_i^\prime\in {\rm add}(U)}(r_i-r_i^\prime){\bf g}^S(T_i^\prime).\nonumber
 \end{eqnarray}
Thus the coefficients before ${\bf g}^S(T_i^\prime)$ are nonnegative for any $T_i^\prime\notin {\rm add}(U)$.
So the Co-Bongartz Completion Condition holds. Hence, $\mathcal G^S_{\mathbf{T}}$ is a $\mathcal G$-system at $S$. The commutative diagrams are easy to check.
\end{proof}

\subsection{$\mathcal G$-systems from  $\tau$-tilting theory}

In this subsection, we fix a finite dimensional basic algebra $A$ over $K$, and let ${\rm mod}A$ be the category of the finitely generated left $A$-modules.

Let $\mathcal X$ and $\mathcal Z$ be two subcategories of ${\rm mod}A$, denote by $\mathcal X\ast \mathcal Z$ the subcategory of ${\rm mod}A$ consisting modules $Y$ such that there exists a short exact sequence in ${\rm mod}A$ of the form
$$0\rightarrow X\rightarrow Y\rightarrow Z\rightarrow 0,$$
where $X\in\mathcal X$ and $Z\in\mathcal Z$.

\begin{Definition}
Let $Q$ be a projective $A$-module, and $\mathfrak T$ be a functorially finite torsion class in ${\rm mod}A$. The full subcategory $\mathcal T_Q(\mathfrak T):={\rm Fac}(Q)\ast \mathfrak T$ is called the {\bf co-Bongartz completion} of $Q$ with respect to  $\mathfrak T$.
\end{Definition}

For a module $M\in{\rm mod}A$, recall that $M^\bot$ is the subcategory of ${\rm mod}A$ given by
$$M^\bot=\{X\in {\rm mod}A|{\rm Hom}_A(M, X)=0\}.$$

\begin{Proposition}\cite[Proposition 3.29]{J}\label{projtor} Let  $Q$ be a projective $A$-module, and $\mathfrak T$ be a functorially finite torsion class in ${\rm mod}A$, then ${\rm Fac}(Q)\ast(\mathfrak{T}\cap Q^\bot)$ is a functorially finite torsion class in ${\rm mod}A$.
\end{Proposition}

\begin{Proposition}\label{proctor}
Let $Q$ be a projective $A$-module, and $\mathfrak T$ be a functorially finite torsion class in ${\rm mod}A$, then $\mathcal T_Q(\mathfrak{T})={\rm Fac}(Q)\ast(\mathfrak{T}\cap Q^\bot)$ and $\mathcal T_Q(\mathfrak{T})$ is a functorially finite torsion class in ${\rm mod}A$.
\end{Proposition}

\begin{proof}
Clearly, ${\rm Fac}(Q)\ast(\mathfrak{T}\cap Q^\bot)\subseteq \mathcal T_Q(\mathfrak{T})={\rm Fac}(Q)\ast \mathfrak T$. Now we show that $\mathcal T_Q(\mathfrak{T})\subseteq {\rm Fac}(Q)\ast(\mathfrak{T}\cap Q^\bot)$. For any $Y\in\mathcal T_Q(\mathfrak{T})$, there exists a short exact sequence of the form
$$\xymatrix{0\ar[r]&X\ar[r]^{f_1}&Y\ar[r]^{f_2}&Z\ar[r]&0
},$$
where $X\in{\rm Fac}(Q)$ and $Z\in\mathfrak{T}$. Consider the canonical sequence of $Y$ with respect to the torsion pair $({\rm Fac}(Q),Q^\bot)$,
$$\xymatrix{0\ar[r]&U\ar[r]^{g_1}&Y\ar[r]^{g_2}&V\ar[r]&0,
}$$
where $g_1:U\rightarrow Y$ is a right minimal ${\rm Fac}(Q)$-approximation of $Y$ and $V\in Q^\bot$. Since $X\in{\rm Fac}(Q)$, we know that $f_1$ factors through $g_1$, thus we have the following commutative diagram

$$\xymatrix{0\ar@{=}[d]\ar[r]&X\ar[d]^{h_1}\ar[r]^{f_1}&Y\ar@{=}[d]\ar[r]^{f_2}&Z\ar@{.>}[d]^{h_2}\ar[r]&0\ar@{=}[d]\\
0\ar[r]&U\ar[r]^{g_1}&Y\ar[r]^{g_2}&V\ar[r]&0
}$$
By Snake Lemma, we know that $h_2:Z\rightarrow V$ is surjective. Since $Z\in\mathfrak{T}$ and $\mathfrak{T}$ is closed under factor modules, we get that $V\in\mathfrak{T}$. Thus $V\in\mathfrak{T}\cap Q^\bot$, and $Y\in {\rm Fac}(Q)\ast(\mathfrak{T}\cap Q^\bot)$. Hence, $\mathcal T_Q(\mathfrak{T})={\rm Fac}(Q)\ast(\mathfrak{T}\cap Q^\bot)$.
Then by Proposition \ref{projtor}, we know that $\mathcal T_Q(\mathfrak{T})$ is a functorially finite torsion class in ${\rm mod}A$.

\end{proof}

\begin{Definition}
Let $Q$ be a projective $A$-module, $M$ be a basic support $\tau$-tilting module and $(M,P)$ be the corresponding support $\tau$-tilting pair. The {\bf co-Bongartz completion}  $\mathcal T_Q(M)$ of $Q$ with respect to $M$ (respectively, the {\bf co-Bongartz completion} $\mathcal T_Q(M,P)$ of $Q$ with respect to $(M,P)$) is defined to be the new basic support $\tau$-tilting module $M^\prime$ (respectively, the new basic support $\tau$-tilting pair $(M^\prime,P^\prime)$) given by the following commutative diagram

$$\xymatrix{(M,P)\ar@{.>}[rrr]^{\mathcal T_Q}\ar[d]&&&(M^\prime, P^\prime)
\\
M\ar@{.>}[rrr]^{\mathcal T_Q}\ar[d]_{{\rm Fac}}&&&M^\prime=\mathcal P(\mathcal T_Q({\rm Fac}(M)))\ar[u]\\
{\rm Fac}(M)\ar[rrr]^{\mathcal T_Q}&&&\mathcal T_Q({\rm Fac}(M))\ar[u]_{\mathcal P}
}$$
\end{Definition}

\begin{Theorem}
Let $A$ be a finite dimensional basic algebra over $K$, and $\mathbf{T}$ be the set of basic support $\tau$-tilting pairs. Then $\mathcal G_\mathbf{T}:=\{G_{(M,P^M)}|(M,P^M)\in\mathbf{T}\}$ is a $\mathcal G$-system at $(A,0)\in \mathbf{T}$ and the following commutative diagrams hold.

$$\xymatrix{(\bar V_0,\bar V_1)\ar[rr]^{\mu_{k}}\ar@{<->}[d]&&(\bar W_0,\bar W_1)\ar@{<->}[d]\\
G_{(\bar V_0,\bar V_1)}\ar[rr]^{\mu_{k}}&&G_{(\bar W_0,\bar W_1)}
}\;\;\;\;\;\;\;\;\;\;\xymatrix{(\bar T_0,\bar T_1)\ar[rr]^{\mathcal T_Q}\ar@{<->}[d]&&(\bar T_0^\prime,\bar T_1^\prime)\ar@{<->}[d]\\
G_{(\bar T_0,\bar T_1)}\ar[rr]^{\mathcal T_J}&&G_{(\bar T_0^\prime,\bar T_1^\prime)}
}$$
where $(\bar V_0,\bar V_1),(\bar W_0,\bar W_1),(\bar T_0,\bar T_1),(\bar T_0^\prime,\bar T_1^\prime)\in\mathbf{T}$ and $Q$ is a direct summand of $\prescript{}{A}{A}$ and $$J=\{{\bf g}^S(Q_i)|Q_i \text{ is an indecomposable direct summand of }Q\}.$$
\end{Theorem}
\begin{proof}
Let $\mathcal D={\rm K}^b({\rm add}A)$, we know that $S:=A\in\mathcal D$ is a basic silting object in $\mathcal D$.
Let  $\mathbf{T}^\prime$ be the set of basic silting objects in $S\ast S[1]$. By Theorem \ref{thmgsys}, we know that $\mathcal G^S_{\mathbf{T}^\prime}:=\{G_T^S|T\in\mathbf{T}\}$ is a $\mathcal G$-system at $S$.

By Theorem \ref{thmijy}, Proposition \ref{proggvector} and the fact that ${\rm End}_{\mathcal D}(A)^{\rm op}\cong A$, there exists a bijection $\tilde F: \mathbf{T}^\prime\rightarrow \mathbf{T}$ such that $G_T^S=G_{\tilde F(T)}$ for any $T\in\mathbf{T}^\prime$, where $G_{\tilde F(T)}$ is the $G$-matrix of the support $\tau$-tilting pair $\tilde F(T)$ in ${\rm mod}A$.
So $\mathcal G_\mathbf{T}=\{G_{(M,P^M)}|(M,P^M)\in\mathbf{T}\}$ is a $\mathcal G$-system at $\tilde F(S)=(A,0)\in \mathbf{T}$.

The first commutative diagram follows directly from Theorem \ref{thmairinj} and the definition of mutations. Now we mainly show that the second commutative diagram holds.  Let $Q$ be a direct summand of $\prescript{}{A}{A}$ and $J=\{{\bf g}^S(Q_i)|Q_i \text{ is an indecomposable direct summand of }Q\}$. For any $G_{(\bar T_0,\bar T_1)}\in\mathcal G_\mathbf{T}$, let $G_{(\bar T_0^\prime,\bar T_1^\prime)}=\mathcal T_J(G_{(\bar T_0,\bar T_1)})$.
By Theorem \ref{thmijy}, there exist unique $T, T^\prime\in\mathbf{T}^\prime$ such that $\tilde F(T)=(\bar T_0,\bar T_1)$ and $\tilde F(T^\prime)=(\bar T_0^\prime,\bar T_1^\prime)$.
By Proposition \ref{proggvector}, we know that $G_{T^\prime}^S=\mathcal T_J(G_T^S)$, i.e., we have the following diagram.
$$\xymatrix{G_{(\bar T_0,\bar T_1)}\ar[r]^{\mathcal T_J}\ar@{=}[d]&G_{(\bar T_0^\prime,\bar T_1^\prime)}\ar@{=}[d]\\
G_T^S\ar[r]^{\mathcal T_J}&G_{T^\prime}^S
}$$

 Let $U$ be the direct summand of $S$ such that ${\rm Hom}_{\mathcal D}(S, U)\cong Q$. By Theorem \ref{thmgsys}, we get $T^\prime=\mathcal T_U(T)$, i.e., we have the following commutative diagram.
$$\xymatrix{T\ar[rr]^{\mathcal T_U}\ar@{<->}[d]&&T^\prime\ar@{<->}[d]\\
G_T^S\ar[rr]^{\mathcal T_J}&&G_{T^\prime}^S
}$$

Now we show that $(\bar T_0^\prime,\bar T_1^\prime)=\mathcal T_Q(\bar T_0,\bar T_1)$. By the definition of co-Bongartz completion in ${\rm mod} A$, it suffices to show that ${\rm Fac}(\bar T_0^\prime)=\mathcal T_Q({\rm Fac}(\bar T_0))={\rm Fac}(Q)\ast {\rm Fac}(\bar T_0)$.

Consider the following triangle
\begin{eqnarray}\label{eqnkmod}
\xymatrix{T[-1]\ar[r]^f&U^\prime\ar[r]^g&X_U\ar[r]^h&T},
\end{eqnarray}
where $f$ is a left minimal ${\rm add}(U)$-approximation of $T[-1]$. So $T^\prime=\mathcal T_U(T)=(U\oplus X_U)^\flat\in\mathbf{T}$.
Applying the functor ${\rm Hom}_{\mathcal D}(S,-)$ to the triangle (\ref{eqnkmod}), we get the following exact sequence in ${\rm mod}A$.
$$\xymatrix{{\rm Hom}_{\mathcal D}(S,U^\prime)\ar[rr]^{{\rm Hom}_{\mathcal D}(S,f)}&&{\rm Hom}_{\mathcal D}(S,X_U)\ar[rr]^{{\rm Hom}_{\mathcal D}(S,g)}&&{\rm Hom}_{\mathcal D}(S,T)\ar[r]&0
}.$$
Let $M={\rm Im}({\rm Hom}_{\mathcal D}(S,f))$, then we have the following two exact sequences.
\begin{eqnarray}
\label{eqnexact1}&&\xymatrix{0\ar[r]&M\ar[r]&{\rm Hom}_{\mathcal D}(S,X_U)\ar[r]&{\rm Hom}_{\mathcal D}(S,T)\ar[r]&0
}\\
\label{eqnexact2}&&\xymatrix{{\rm Hom}_{\mathcal D}(S,U^\prime)\ar[rr]^{{\rm Hom}_{\mathcal D}(S,f)}&&M\ar[r]&0}
\end{eqnarray}
By ${\rm Hom}_{\mathcal D}(S,U^\prime)\in{\rm add}(Q)\subseteq {\rm Fac}(Q)$ and the exact sequence (\ref{eqnexact2}), we know that $M\in{\rm Fac}(Q)$. Then by the exact sequence (\ref{eqnexact1}), we know that ${\rm Hom}_{\mathcal D}(S,X_U)\in \mathcal T_Q({\rm Fac}(\bar T_0))={\rm Fac}(Q)\ast {\rm Fac}(\bar T_0).$
So ${\rm Hom}_{\mathcal D}(S,U\oplus X_U)\in \mathcal T_Q({\rm Fac}(\bar T_0))$ and we can get $\bar T_0^\prime={\rm Hom}_{\mathcal D}(S,T^\prime)\in \mathcal T_Q({\rm Fac}(\bar T_0))$. So ${\rm Fac}(\bar T_0^\prime)\subseteq \mathcal T_Q({\rm Fac}(\bar T_0))$. On the other hand, by $U,X_U\in{\rm add}(T^\prime)$, we know  $Q\cong {\rm Hom}_{\mathcal D}(S,U)\in{\rm add}(\bar T_0^\prime)$ and ${\rm Hom}_{\mathcal D}(S,X_U)\in{\rm add}(T_0^\prime)$. Thus ${\rm Fac}(Q)\subseteq {\rm Fac}(\bar T_0^\prime)$ and ${\rm Fac}({\rm Hom}_{\mathcal D}(S,X_U))\subseteq {\rm Fac}(\bar T_0^\prime)$. By the exact sequence (\ref{eqnexact1}), we know that $${\rm Fac}(\bar T_0)={\rm Fac}({\rm Hom_{\mathcal D}}(S,T))\subseteq {\rm Fac}({\rm Hom}_{\mathcal D}(S,X_U))\subseteq {\rm Fac}(\bar T_0^\prime).$$
Thus both ${\rm Fac}(Q)$ and ${\rm Fac}(\bar T_0)$ are subcategories of ${\rm Fac}(\bar T_0^\prime)$. Since ${\rm Fac}(\bar T_0^\prime)$ is a torsion class, which is closed under extension, we get that
$$\mathcal T_Q({\rm Fac}(\bar T_0))={\rm Fac}(Q)\ast {\rm Fac}(\bar T_0)\subseteq {\rm Fac}(\bar T_0^\prime).$$
So $ {\rm Fac}(\bar T_0^\prime)=\mathcal T_Q({\rm Fac}(\bar T_0))$ and thus $(\bar T_0^\prime,\bar T_1^\prime)=\mathcal T_Q(\bar T_0,\bar T_1)$. Then we have the following commutative diagram.
$$\xymatrix{(\bar T_0,\bar T_1)\ar[rr]^{\mathcal T_Q}\ar@{<->}[d]&&(\bar T_0^\prime,\bar T_1^\prime)\ar@{<->}[d]\\
T\ar[rr]^{\mathcal T_U}&&T^\prime
}$$
In summary, we get the following commutative diagram.
$$\xymatrix{(\bar T_0,\bar T_1)\ar[rr]^{\mathcal T_Q}\ar@{<->}[d]&&(\bar T_0^\prime,\bar T_1^\prime)\ar@{<->}[d]\\
G_{(\bar T_0,\bar T_1)}\ar[rr]^{\mathcal T_J}&&G_{(\bar T_0^\prime,\bar T_1^\prime)}
}$$
\end{proof}

\section{Co-Bongartz completions in cluster algebras}
In this section we show that ``co-Bongartz completion" can be defined on the set of clusters of a cluster algebra, and we show that the definition of ``co-Bongartz completion" does not depend on the choice of the initial cluster.
\subsection{Basics on cluster algebras }
Since the things we focus on does not depend on the coefficients of cluster algebras, we will restrict ourselves in  coefficient-free case.

Recall that an $n\times n$ integer matrix $B=(b_{ij})$ is called  {\bf skew-symmetrizable} if there is a positive integer diagonal matrix $S$ such that $SB$ is skew-symmetric, where $S$ is called a {\bf skew-symmetrizer} of $B$.

We take an ambient field $\mathcal F=\mathbb Q(u_1,\cdots,u_n)$  to be the field of rational functions in $n$ independent variables with coefficients in $\mathbb Z$.

\begin{Definition}
A {\bf  seed} in $\mathcal F$ is a pair $({\bf x},B)$ satisfying that

(i)  ${\bf x}=\{x_1,\cdots, x_n\}$ is a free generating set of   $\mathcal F$ over $\mathbb {Z}$. ${\bf x}$ is called the  {\bf cluster} of  $({\bf x},B)$ and $x_1\cdots,x_n$  are called {\bf cluster variables}.

(ii) $B=(b_{ij})_{n\times n}$ is a skew-symmetrizable matrix, called an {\bf exchange matrix}.
\end{Definition}

Let  $({\bf x},B)$ be a seed in $\mathcal F$, a monomial in $x_1,\cdots,x_n$ is called a {\bf cluster monomial} in ${\bf x}$.

\begin{Definition}\label{defmutation}
Let $({\bf x},B)$ be a  seed in $\mathcal F$. Define the {\bf mutation}  of  $({\bf x},B)$ at $k\in\{1,\cdots,n\}$ as a new pair $\mu_k({\bf x},B)=( {\bf x}^{\prime}, B^{\prime})$ in $\mathcal F$ given by
\begin{eqnarray}
b_{ij}^{^\prime}&=&\begin{cases}-b_{ij}~,& i=k\text{ or } j=k;\\ b_{ij}+{\rm sgn}(b_{ik}){\rm max}(b_{ik}b_{kj},0)~,&\text{otherwise}.\end{cases}\nonumber\\
x_i^{\prime}&=&\begin{cases}x_i~,&\text{if } i\neq k\\x_k^{-1}(\prod\limits_{b_{jk}>0}x_j^{b_{jk}}+ \prod\limits_{b_{jk}<0}x_j^{-b_{jk}}),~& \text{if }i=k.\end{cases}\nonumber
\end{eqnarray}
$B^\prime$ is called mutation of $B$ at $k$, and is denoted by $B^\prime=\mu_k(B)$.
\end{Definition}
It can be seen that $\mu_k({\bf x},B)$ is also a  seed and $\mu_k(\mu_k({\bf x},B))=({\bf x},B)$.

Let $\mathbb T_n$ be the  $n$-regular tree, and label the edges of $\mathbb T_n$ by  $1,\dots,n$ such that the $n$  different edges adjacent to the same vertex of $\mathbb T_n$ receive different labels.
\begin{Definition}
  A {\bf cluster pattern}  $\mathcal S$ in $\mathcal F$ is an assignment of a seed  $({\bf x}_t,B_t)$ to every vertex $t$ of the infinite $n$-regular tree $\mathbb T_n$, such that $({\bf x}_{t^\prime}, B_{t^\prime})=\mu_k( {\bf x}_t, B_t)$ for any edge $t^{~\underline{\quad k \quad}}~ t^{\prime}$.
\end{Definition}

We always denote by ${\bf x}_t=\{x_{1;t},\dots, x_{n;t}\}$ and $B_t=(b_{ij}^t)_{n\times n}$. The {\bf  cluster algebra}  $\mathcal A(\mathcal S)$  associated with a  cluster pattern $\mathcal S$ is the $\mathbb {Z}$-subalgebra of the field $\mathcal F$ generated by all cluster variables of  $\mathcal S$, i.e., $\mathcal A(\mathcal S)=\mathbb Z[{\mathcal X}]$, where  ${\mathcal X}=\bigcup_{t\in\mathbb T_n}{\bf x}_t$.

\begin{Theorem}(\cite{FZ} Laurent phenomenon) \label{laurent} Let $\mathcal A(\mathcal S)$ be a cluster algebra, then  for any cluster variable $x_{i;t}$ and seed $({\bf x}_{t_0},B_{t_0})$ of $\mathcal A(\mathcal S)$, we have
$$x_{i;t}\in \mathcal L_{t_0}=\mathbb {Z}[x_{1;t_0}^{\pm1},\dots,x_{n;t_0}^{\pm1}].$$
\end{Theorem}

\begin{Theorem}(\cite{CL} Cluster formula) Let $\mathcal A(\mathcal S)$ be a cluster algebra. For any two seeds $({\bf x}_t,B_t)$ and $({\bf x}_{t_0},B_{t_0})$, we have
$$H_{t_0}^t(B_tS^{-1}){H_{t_0}^t}^{\rm T}=B_{t_0}S^{-1}\;\;\;\text{ and }\;\;\;{\rm det}(H_{t_0}^t)=\pm 1,$$ where $S$ is a skew-symmetrizer of $B_{t_0}$ and $H_{t_0}^{t}$ is the matrix given by

$$H_{t_0}^{t}={\rm diag}(x_{1;t_0},\cdots,x_{n;t_0})\begin{pmatrix} \frac{\partial x_{1;t}}{\partial x_{1;t_0}}&\cdots& \frac{\partial x_{n;t}}{\partial x_{1;t_0}}\\
\vdots&\cdots&\vdots\\
 \frac{\partial x_{1;t}}{\partial x_{n;t_0}}&\cdots& \frac{\partial x_{n;t}}{\partial x_{n;t_0}}
\end{pmatrix}{\rm diag}(x_{1;t}^{-1},\cdots,x_{n;t}^{-1}).$$
\end{Theorem}
From the cluster formula, we can see that for a cluster algebra $\mathcal A(\mathcal S)$ with an initial seed $({\bf x}_{t_0},B_{t_0})$, $B_t$ is uniquely determined by ${\bf x}_t$.

\subsection{$C$-matrices and $G$-matrices of cluster algebras}
Since we did not talk about cluster algebras with principal coefficients, we use the notation {\em matrix pattern} to introduce the $C$-matrices and $G$-matrices of cluster algebras.
\begin{Definition}
Let $\mathbb T_n$ be the $n$-regular tree with an initial vertex $t_0\in\mathbb T_n$, and $B$ be a skew-symmetrizable integer matrix.  A {\bf matrix patter $\mathcal M(B,t_0)$ at $t_0$} is an assignment of a triple $(B_t,C_t,G_t)$ of $n\times n$ integer matrices to every vertex $t$ of the $n$-regular tree $\mathbb T_n$ such that $(B_{t_0},C_{t_0},G_{t_0})=(B,I_n,I_n)$ and for any edge $t^{~\underline{\quad k \quad}}~ t^{\prime}$, we have $B_{t^\prime}=\mu_k(B_t)$ and
the following exchange relations:
\begin{eqnarray}
c_{ij}^{t^\prime}&=&\begin{cases}-c_{ij}^t~,& j=k;\\ c_{ij}^t+{\rm sgn}(c_{ik}^t){\rm max}(c_{ik}^tb_{kj}^t,0)~,
&j\neq k.\end{cases}\nonumber\\
g_{ij}^{t^\prime}&=&\begin{cases}g_{ij}^t~,& j\neq k;\\-g_{ik}^t+\sum\limits_{b_{jk}^t>0}g_{ij}^tb_{jk}^t-\sum\limits_{c_{jk}^t>0}b_{ij}^{t_0}c_{jk}^t ~,&j=k.\end{cases}\nonumber
\end{eqnarray}
$\mu_k(B_t,C_t,G_t):=(B_{t^\prime},C_{t^\prime},G_{t^\prime})$ is called the {\bf mutation} of $(B_t,C_t,G_t)$ at $k$.
\end{Definition}
It can be checked that $\mu_k\mu_k(B_t,C_t,G_t)=(B_t,C_t,G_t)$, so the above definition is well-defined.

Let $\mathcal M(B,t_0)$ be a matrix pattern at $t_0\in\mathbb T_n$, the triple $(B_t,C_t,G_t)$ at $t$ is called a {\bf matrix seed}. We call $C_t, G_t$ the {\bf $C$-matrix, $G$-matrix} of $B$ respectively.

\begin{Definition}
Let $\mathcal A(\mathcal S)$ be a cluster algebra with an initial seed $({\bf x}_{t_0},B_{t_0})$, and $\mathcal M(B_{t_0},t_0)$ be the matrix pattern at $t_0$. The matrix $G_t$ (respectively, the matrix $C_t$) in the matrix seed $(B_t,C_t,G_t)$ is called the {\bf $G$-matrix} (respectively, {\bf $C$-matrix}) of ${\bf x}_t$ with respect to ${\bf x}_{t_0}$. Sometimes, we denote it by $G_t^{t_0}$ (respectively, $C_t^{t_0}$). The $i$-th column vector ${\bf g}_{i;t}^{t_0}$ of $G_t^{t_0}$ is called the {\bf$g$-vector} of the cluster variable $x_{i;t}$ with respect to the cluster ${\bf x}_{t_0}$ and is denoted by ${\bf g}^{t_0}(x_{i;t}):={\bf g}_{i;t}^{t_0}$. For a cluster monomial ${\bf x}_t^{\bf v}$, the $g$-vector of  ${\bf x}_t^{\bf v}$ is defined to be the vector ${\bf g}^{t_0}({\bf x}_t^{\bf v})=G_t^{t_0}{\bf v}$.
\end{Definition}
\begin{Remark}
The definition above is well-defined for cluster algebra with principal coefficients \cite{FZ3}. Combine the principal coefficients case with \cite[Proposition 6.1]{CL1}, we can get the above definition is well defined for cluster algebra with trivial coefficients.
\end{Remark}

\begin{Theorem}\cite{GHKK} Let $\mathcal A(\mathcal S)$ be a cluster algebra, and  $({\bf x}_{t_0},B_{t_0}), ({\bf x}_t,B_t)$ be two seeds of $\mathcal A(\mathcal S)$, then

(i)  each column vector of $C_t^{t_0}$ is either a nonnegative vector or a nonpositive vector.

(ii) each row vector of $G_t^{t_0}$ is either a nonnegative vector or a nonpositive vector.
\end{Theorem}

\begin{Theorem}\cite{GHKK}\label{thmghkk} Let $\mathcal A(\mathcal S)$ be a cluster algebra with an initial seed $({\bf x}_{t_0},B_{t_0})$, then the map ${\bf x}_t^{\bf v}\mapsto {\bf g}^{t_0}({\bf x}_t^{\bf v})$ is a injection from the set of cluster monomials of $\mathcal A(\mathcal S)$ to $\mathbb Z^n$.
\end{Theorem}

\begin{Proposition}\cite{NZ,CL}\label{pronzcl} Let $\mathcal A(\mathcal S)$ be a cluster algebra, and  $({\bf x}_{t_0},B_{t_0}), ({\bf x}_t,B_t)$ be two seeds of $\mathcal A(\mathcal S)$. Then $G_t(B_tS^{-1})G_t^{\rm T}=B_{t_0}S^{-1}$ and $S^{-1}G_t^{\rm T}SC_t=I_n$, where $S$ is any skew-symmetrizer of $B_{t_0}$.
\end{Proposition}

\subsection{Co-Bongartz completions in cluster algebras} In this subsection we show that ``co-Bongartz completion" can be defined on the set of clusters of a cluster algebra, and we show that the definition of ``co-Bongartz completion" does not depend on the choice of the initial cluster.
\begin{Theorem}\cite[Theorem 4.8, Theorem 5.5]{CL1}\label{thmcaogpair}
 Let $\mathcal A(\mathcal S)$ be a cluster algebra, and  $({\bf x}_{t_0},B_{t_0}), ({\bf x}_t,B_t)$ be two seeds of $\mathcal A(\mathcal S)$. Then for any subset $U\subseteq {\bf x}_{t_0}$, there exists a unique cluster ${\bf x}_{t^\prime}$ satisfying the following two statements.

(a) $U\subseteq{\bf x}_{t^\prime}$.

(b) If $R_t^{t^\prime}$ is the matrix such that $G_t^{t_0}=G_{t^\prime}^{t_0}R_t^{t^\prime}$, then the $i$-th row vector of $R_t^{t^\prime}$ is in $\mathbb Z_{\geq 0}^n$ for any $i$ satisfying $x_{i;t^\prime}\notin U$.
\end{Theorem}

\begin{Remark}
In \cite{CL1} the authors prove the existence of ${\bf x}_{t^\prime}$ under the fixed initial cluster ${\bf x}_{t_0}$. What we want to do in this subsection is to show that ${\bf x}_{t^\prime}$ does not depend on the choice of the initial cluster.
\end{Remark}
\begin{Definition} Keep the notations in above theorem. The cluster ${\bf x}_{t^\prime}$ is called the {\bf co-Bongartz completion} of $(U,{\bf x}_{t_0})$ with respect to ${\bf x}_t$, and is denoted by ${\bf x}_{t^\prime}=\mathcal T_{(U,{\bf x}_{t_0})}({\bf x}_t)$.
In the terminology in \cite{CL1}, $({\bf x}_t,{\bf x}_{t^\prime})$ is called a {\bf $g$-pair along ${\bf x}_{t_0}\backslash U$}.
\end{Definition}

{\bf Construction:} Now we  {\em summarize the construction of ${\bf x}_{t^\prime}=\mathcal T_{(U,{\bf x}_{t_0})}({\bf x}_t)$} from the proof of \cite[Theorem 4.8]{CL1}. The main idea essentially comes from \cite{M} by Muller.
Let $(B_t,C_t,G_t)$ obtained from $(B_{t_0},I_n,I_n)$ by mutations along
the sequence $(k_1,\cdots,k_s)$, i.e.,
$$(B_t,C_t,G_t)=\mu_{k_s}\cdots \mu_{k_2}\mu_{k_1}(B_{t_0},I_n,I_n).$$
Let ${\bf c}_i$ be the $k_i$-th column vector of the $C$-matrix in the matrix seed $\mu_{k_{i-1}}\cdots\mu_{k_2}\mu_{k_1}(B_{t_0},I_n,I_n)$, thus we have a sequence of vectors $({\bf c}_1,\cdots,{\bf c}_s)$.
Then we delete the vector ${\bf c}_i=(c_{1i},\cdots,c_{ni})^{\rm T}$ such that $c_{ji}\neq 0$ for some $j$ satisfying ${x}_{j;t_0}\in U$.
Thus we get a subsequence $({\bf c}_{i_1},\cdots,{\bf c}_{i_m})$, where $m\leq s$. This sequence of vectors will induce a sequence of mutations by the following method. If ${\bf c}_{i_1}$ is the $j_1$-th column vector of the $C$-matrix in
$(B_{t_0},I_n,I_n)$, then we mutate the matrix seed $(B_{t_0},I_n,I_n)$ at $j_1$. If  ${\bf c}_{i_2}$ is the $j_2$-th column vector of the $C$-matrix in
$\mu_{j_1}(B_{t_0},I_n,I_n)$, then we mutate the matrix seed $\mu_{j_1}(B_{t_0},I_n,I_n)$ at $j_2$. Continue this, we will get a sequence $(j_1,\cdots,j_m)$. It turns out $({\bf x}_{t^\prime},B_{t^\prime})=\mu_{j_m}\cdots\mu_{j_2}\mu_{j_1}({\bf x}_{t_0}, B_{t_0})$ and
$(B_{t^\prime}, C_{t^\prime},G_{t^\prime})=\mu_{j_m}\cdots\mu_{j_2}\mu_{j_1}(B_{t_0},I_n,I_n)$.

We give an example to illustrate the above process.

\begin{Example}\label{example1}
 Let ${\bf x}_{t_0}=(x_1,x_2,x_3)$ and $B_{t_0}=\begin{pmatrix}0&1&-1\\-1&0&1\\1&-1&0 \end{pmatrix}$. Let $({\bf x}_t,B_t)=\mu_2\mu_1\mu_3\mu_2({\bf x}_{t_0},B_{t_0})$  and $U=\{x_3\}\subseteq {\bf x}_{t_0}$. Then we know  $(B_t,C_t,G_t)=\mu_2\mu_1\mu_3\mu_2(B,I_3,I_3)$. For convenience, we write the triple $(B_t,C_t,G_t)$ as a big matrix $\begin{pmatrix}B_t\\C_t\\G_t\end{pmatrix}$.

\begin{eqnarray}
\begin{pmatrix} 0&1&-1\\-1&0&1\\1&-1&0\\  \hdashline[2pt/2pt] 1&0&0\\0&1&0\\0&0&1\\
\hdashline[2pt/2pt] 1&0&0\\0&1&0\\0&0&1\end{pmatrix}\xrightarrow{\mu_2}
\begin{pmatrix} 0&-1&0\\1&0&-1\\0&1&0\\ \hdashline[2pt/2pt] 1&0&0\\0&-1&1\\0&0&1\\
\hdashline[2pt/2pt] 1&0&0\\0&-1&0\\0&1&1\end{pmatrix}\xrightarrow{\mu_3}
\begin{pmatrix}0&-1&0\\1&0&1\\0&-1&0\\  \hdashline[2pt/2pt] 1&0&0\\0&0&-1\\0&1&-1\\
\hdashline[2pt/2pt] 1&0&0\\0&-1&-1\\0&1&0\end{pmatrix}\xrightarrow{\mu_1}
\begin{pmatrix}0&1&0\\-1&0&1\\0&-1&0\\  \hdashline[2pt/2pt] -1&0&0\\0&0&-1\\0&1&-1\\
\hdashline[2pt/2pt] -1&0&0\\0&-1&-1\\0&1&0\end{pmatrix}\xrightarrow{\mu_2}
\begin{pmatrix}0&-1&1\\1&0&-1\\-1&1&0\\  \hdashline[2pt/2pt] -1&0&0\\0&0&-1\\0&-1&0\\
\hdashline[2pt/2pt] -1&0&0\\0&0&-1\\0&-1&0\end{pmatrix}.\nonumber
\end{eqnarray}
The sequence of mutations $(\mu_2,\mu_3,\mu_1,\mu_2)$ induces a sequence of vectors $\begin{pmatrix}0\\1\\0 \end{pmatrix},\begin{pmatrix}0\\1\\1 \end{pmatrix},\begin{pmatrix}1\\0\\0 \end{pmatrix},\begin{pmatrix}0\\0\\1 \end{pmatrix}$. Delete the vectors such that the third component is non-zero, we get the subsequence of vectors
 $\begin{pmatrix}0\\1\\0 \end{pmatrix},\begin{pmatrix}1\\0\\0 \end{pmatrix}$. This subsequence induces a sequence of mutations $(\mu_2,\mu_1)$.
 \begin{eqnarray}
 \begin{pmatrix} 0&1&-1\\-1&0&1\\1&-1&0\\  \hdashline[2pt/2pt] 1&0&0\\0&1&0\\0&0&1\\
\hdashline[2pt/2pt] 1&0&0\\0&1&0\\0&0&1\end{pmatrix}\xrightarrow{\mu_2}
\begin{pmatrix} 0&-1&0\\1&0&-1\\0&1&0\\ \hdashline[2pt/2pt] 1&0&0\\0&-1&1\\0&0&1\\
\hdashline[2pt/2pt] 1&0&0\\0&-1&0\\0&1&1\end{pmatrix}
\xrightarrow{\mu_1}
\begin{pmatrix} 0&1&0\\-1&0&-1\\0&1&0\\ \hdashline[2pt/2pt] -1&0&0\\0&-1&1\\0&0&1\\
\hdashline[2pt/2pt] -1&0&0\\0&-1&0\\0&1&1\end{pmatrix}\nonumber
 \end{eqnarray}
 So $({\bf x}_{t^\prime},B_{t^\prime})=\mu_1\mu_2({\bf x}_{t_0},B_{t_0})$ and thus ${\bf x}_{t^\prime}=\mathcal T_{(x_3,{\bf x}_{t_0})}({\bf x}_t)=(\frac{x_1+x_2+x_3}{x_1x_2},\frac{x_1+x_3}{x_2},x_3)$.
\end{Example}

\begin{Theorem}\label{thmmatrix}
Let $\mathcal A(\mathcal S)$ be a cluster algebra, and  $\mathbf{T}=\mathbb T_n$ the $n$-regular tree. Then for any seed $({\bf x}_{t_0},B_{t_0})$ of $\mathcal A(\mathcal S)$, we have that  $\mathcal G_{\mathbf{T}}=\{G_t^{t_0}|t\in \mathbf T\}$ is a $\mathcal G$-system at $t_0$ and the following commutative diagrams hold.

$$\xymatrix{({\bf x}_t,B_t)\ar@{<->}[d]\ar@{<->}[r]^{\mu_k}&({\bf x}_{t^\prime},B_{t^\prime})\ar@{<->}[d]\\
G_t^{t_0}\ar@{<->}[r]^{\mu_k}&G_{t^\prime}^{t_0}
}\;\;\;\;\;\;\;\;\;\;\xymatrix{{\bf x}_t\ar[rr]^{\mathcal T_{(U,{\bf x}_{t_0})}}\ar@{<->}[d]&&{\bf x}_{t^\prime}\ar@{<->}[d]\\
G_{t}^{t_0}\ar[rr]^{\mathcal T_J}&&G_{t^\prime}^{t_0}}
$$
where $U$ is a subset of ${\bf x}_{t_0}$, $J=\{{\bf g}^{t_0}(x_{i;t_0})|x_{i;t_0}\in U\}$ and $\mu_k$ and $\mathcal T_J$ in the second row of the commutative diagrams are the mutation and co-Bongartz completion  in a $\mathcal G$-system.
\end{Theorem}
\begin{proof}
 By Proposition \ref{pronzcl}, the column vectors of $G_t$ form a  $\mathbb Z$-basis of $\mathbb Z^n$ for any $t\in\mathbf{T}$. The  Mutation Condition, Uniqueness Condition and  Co-Bongartz Completion Condition follow from   Theorem \ref{thmghkk} (iii), and Theorem \ref{thmcaogpair} directly. So $\mathcal G_{\mathbf{T}}$ is a $\mathcal G$-system at $t_0$. The commutative diagrams are easy to check.
\end{proof}

The following result is a direct corollary of Theorem \ref{thmmatrix}, Theorem \ref{thmconnect}, and one can also refer to \cite[Theorem 6.2]{CL1}.
\begin{Corollary}
\label{corconnect}
Let $\mathcal A(\mathcal S)$ be a cluster algebra with an initial seed $({\bf x}_{t_0},B_{t_0})$. Let ${\bf x}_{t_1}$ and ${\bf x}_{t_2}$ be two clusters of $\mathcal A(\mathcal S)$ satisfying $x_{i;t_1}=x_{i;t_2}$ for any $i\in I\subseteq \{1,\cdots,n\}$. Then there exists a sequence $(k_1,\cdots,k_s)$ with $k_j\notin I$ for $j=1,\cdots,s$ such that $({\bf x}_{t_2}, B_{t_2})=\mu_{k_s}\cdots\mu_{k_2}\mu_{k_1}({\bf x}_{t_1},B_{t_1})$.
\end{Corollary}

For a  matrix $A=(a_{ij})_{n\times n}$ and $I,J\subseteq\{1,\cdots,n\}$, we denote by $A|_{I\times J}$ the submatrix of $A$ given by entries $a_{ij}$ with $i\in I$ and $j\in J$. We denote by $[A]_+^{k\bullet}$ and $[A]_+^{\bullet k}$ the matrices given by $$([A]_+^{k\bullet})_{ij}:=\begin{cases}0&i\neq k,\\ [a_{kj}]_+&i=k. \end{cases}\;\text{ and }
([A]_+^{\bullet k})_{ij}:=\begin{cases}0&j\neq k,\\ [a_{ik}]_+&j=k. \end{cases}
$$
\begin{Lemma}\cite[(4.2)]{NZ} \label{lemnz}
Let $B$ be a skew-symmetrizable matrix, and $B^\prime=\mu_k(B)$. Suppose that $t_0^{~\underline{\quad k \quad}}~ u$ are two adjacent vertices in $\mathbb T_n$, and  $\mathcal M(B,t_0)$ (respectively, $\mathcal M(B^\prime,u)$) be a matrix pattern at $t_0$ (respectively, at $u$). Then
$$G_t^{u}=(J_k+[\varepsilon B_{t_0}]_+^{\bullet k})G_t^{t_0}+B_{t_0}[-\varepsilon G_t^{t_0}]_+^{k \bullet},$$
for any choice sign $\varepsilon=\pm 1$, where $J_k=I_n-2E_{kk}$.
\end{Lemma}

\begin{Theorem}\label{thmnotdepend}
Let $\mathcal A(\mathcal S)$ be a  cluster algebra and $({\bf x}_{t_0},B_{t_0}), ({\bf x}_{v},B_{v})$ be any two seeds with $x_{i;t_0}=x_{i;v}$ for any $i\in I\subseteq \{1,\cdots,n\}$. Let $U=\{x_{i;t_0}|i\in I\}\subseteq {\bf x}_{t_0}\cap{\bf x}_{v}$, then for any cluster ${\bf x}_t$, we have
$$\mathcal T_{(U,{\bf x}_{t_0})}({\bf x}_t)=\mathcal T_{(U,{\bf x}_v)}({\bf x}_t).$$
\end{Theorem}
From the above theorem, we can see that  the cluster ${\bf x}_{t^\prime}:=\mathcal T_{(U,{\bf x}_{t_0})}({\bf x}_t)$ is uniquely determined by $U$ and does not depend on the choice of the initial cluster containing $U$. So there is no confusion to denote by ${\bf x}_{t^\prime}=\mathcal T_{U}({\bf x}_t)$, which is called the {\bf co-Bongartz completion} of $U$ with respect to ${\bf x}_t$.

\begin{proof}
By Corollary \ref{corconnect}, it suffices to show that for any $t_0^{~\underline{\quad k \quad}}~ u$~\;with $k\notin I$, we have $$\mathcal T_{(U,{\bf x}_{t_0})}({\bf x}_t)=\mathcal T_{(U,{\bf x}_u)}({\bf x}_t).$$ Let ${\bf x}_{t^\prime}=\mathcal T_{(U,{\bf x}_{t_0})}({\bf x}_t)$. We show ${\bf x}_{t^\prime}=\mathcal T_{(U,{\bf x}_u)}({\bf x}_t)$ by showing $G_{t^\prime}^u=\mathcal T_{J}(G_{t}^{u})$, where $$J=\{{\bf g}^{u}(x_{i;u})|x_{i;u}\in U\}=\{{\bf e}_i|i\in I\}.$$

Without loss of generality, we can assume that  $I=\{p+1,\cdots, n\}$ and $I^c:=\{1,\cdots,p\}$. By ${\bf x}_{t^\prime}=\mathcal T_{(U,{\bf x}_{t_0})}({\bf x}_t)$, we know that $G_{t^\prime}^{t_0}=\mathcal T_{J}(G_{t}^{t_0})$, where $J=\{{\bf g}^{t_0}(x_{i;t_0})|x_{i;t_0}\in U\}=\{{\bf e}_i|i\in I\}$.
 By the definition of co-Bongartz completion in $\mathcal G$-system, we have the following two facts.

Fact (i):  $J$ is a subset of $G_{t^\prime}^{t_0}$ (as a set of column vectors). Without loss of generality, we assume that $G_{t^\prime}^{t_0}$ has the form of
$$G_{t^\prime}^{t_0}=\begin{pmatrix}G_{t^\prime}^{t_0}|_{I^c\times I^c} &0\\
\ast& I_{n-p}\end{pmatrix}.$$
Fact (ii): there exist some $Q\in M_{|I^c|\times n}(\mathbb Z_{\geq 0})$ such that
\begin{equation}\label{eqng00a}
G_t^{t_0}|_{I^c\times [1,n]}=G_{t^\prime}^{t_0}|_{I^c\times I^c}Q,
\end{equation}
 By Fact (ii) and $k\in I^c$, the $k$-th row vector of $G_t^{ B_{t_0};t_0}$ and the $k$-th row vector of $G_{t^\prime}^{ B_{t_0};t_0}$ have the same sign, say the sign is $\varepsilon_k\in \{\pm 1\}$.

By Lemma \ref{lemnz}, we know that
\begin{eqnarray}
\label{eqng1}  G_t^{u}&=&(J_k+[\varepsilon_k B_{t_0}]_+^{\bullet k})G_t^{t_0}+B_{t_0}[-\varepsilon_k G_t^{t_0}]_+^{k \bullet}=(J_k+[\varepsilon_k B_{t_0}]_+^{\bullet k})G_t^{t_0};\\
\label{eqng2}   G_{t^\prime}^{u}&=&(J_k+[\varepsilon_k B_{t_0}]_+^{\bullet k})G_{t^\prime}^{t_0}+B_{t_0}[-\varepsilon_k G_{t^\prime}^{t_0}]_+^{k \bullet}=(J_k+[\varepsilon_k B_{t_0}]_+^{\bullet k})G_{t^\prime}^{t_0}.
\end{eqnarray}
We also know that the matrix $J_k+[\varepsilon_k B_{t_0}]_+^{\bullet k}$ has the form of
$$J_k+[\varepsilon_k B_{t_0}]_+^{\bullet k}=\begin{pmatrix}(J_k+[\varepsilon_k B_{t_0}]_+^{\bullet k})|_{I^c\times I^c}&0\\ \ast &I_{n-p}  \end{pmatrix}.$$
Thus by the equalities  (\ref{eqng1}) and (\ref{eqng2}), we know that
\begin{eqnarray}
G_t^{u}|_{I^c\times [1,n]}&=&(J_k+[\varepsilon_k B_{t_0}]_+^{\bullet k})|_{I^c\times I^c}G_{t}^{t_0}|_{I^c\times [1,n]};\nonumber\\
 G_{t^\prime}^{u}|_{I^c\times I^c}&=&(J_k+[\varepsilon_k B_{t_0}]_+^{\bullet k})|_{I^c\times I^c}G_{t^\prime}^{t_0}|_{I^c\times I^c}.\nonumber
\end{eqnarray}
Then by the equality (\ref{eqng00a}), we can get
\begin{eqnarray}
G_t^{u}|_{I^c\times [1,n]}&=&(J_k+[\varepsilon_k B_{t_0}]_+^{\bullet k})|_{I^c\times I^c}G_{t}^{t_0}|_{I^c\times [1,n]}\nonumber\\
&=&(J_k+[\varepsilon_k B_{t_0}]_+^{\bullet k})|_{I^c\times I^c}\left(G_{t^\prime}^{t_0}|_{I^c\times I^c}Q\right)\nonumber\\
&=&\left((J_k+[\varepsilon_k B_{t_0}]_+^{\bullet k})|_{I^c\times I^c}G_{t^\prime}^{t_0}|_{I^c\times I^c}\right)Q\nonumber\\
&=& G_{t^\prime}^{u}|_{I^c\times I^c}Q,\nonumber
\end{eqnarray}
where $Q\in M_{|I^c|\times n}(\mathbb Z_{\geq 0})$.

By the equality (\ref{eqng2}), we can see that $J=\{{\bf e}_i|i\in I\}=\{{\bf e}_{p+1},\cdots,{\bf e}_n\}$ is a subset of $G_{t^\prime}^u$ (as a set of column vectors).  Hence, $G_{t^\prime}^u=\mathcal T_{J}(G_{t}^{u})$ and thus ${\bf x}_{t^\prime}=\mathcal T_{(U,{\bf x}_u)}({\bf x}_t)$. The completes the proof.
\end{proof}

\begin{Corollary}\label{corclustertwi}
 Let $\mathcal A(\mathcal S)$ be a cluster algebra, and $U$ be a subset of some cluster of $\mathcal A(\mathcal S)$. Then  ${\bf x}_{t^\prime}=\mathcal T_U({\bf x}_t)$ if and only if $U\subseteq {\bf x}_{t^\prime}$ and  the $i$-th row vector of $G_t^{t^\prime}$ is a nonnegative vector for any $i$ such that $x_{i;t^\prime}\notin U$.
\end{Corollary}

\begin{proof}

By Theorem \ref{thmnotdepend}, we know that ${\bf x}_{t^\prime}=\mathcal T_U({\bf x}_t)$ if and only if ${\bf x}_{t^\prime}=\mathcal T_{(U,{\bf x}_{t^\prime})}({\bf x}_t)$. Then the result follows.
\end{proof}
The following corollary follows directly from  Corollary \ref{cormintwista} and Theorem \ref{thmmatrix}.
\begin{Corollary}\label{corlast}
Let $\mathcal A(\mathcal S)$ be a cluster algebra, and $U=\{x_1,\cdots,x_s\}$ be a subset of some cluster of $\mathcal A(\mathcal S)$. Then for any cluster ${\bf x}_{t}$, we have
$$\mathcal T_U({\bf x}_t)=\mathcal T_{x_{i_s}}\cdots \mathcal T_{x_{i_s}}\mathcal T_{x_{i_1}}({\bf x}_t).$$
where $i_1,\cdots,i_s$ is any permutation of $1,\cdots,n$.
\end{Corollary}
By the above corollary, we know that in order to study the co-Bongartz completions in a cluster algebra, it suffices to study  the elementary co-Bongartz completions. In next section, we give a direct construction of elementary co-Bongartz completions on the unpunctured surfaces.

\section{Co-Bongartz completions on the  unpunctured surfaces}
In this section we first give the construction of ``elementary co-Bongartz completions"  on the unpunctured surfaces. Then we show that the elementary co-Bongartz completions on unpunctured surfaces are compatible with the elementary co-Bongartz completions in the associated cluster algebras. Co-Bongartz completions on unpunctured surfaces can be defined by a sequence of elementary co-Bongartz completions.
\subsection{Basics on unpunctured surfaces}
\begin{Definition}
(Bordered surface with marked points). Let $S$ be a connected oriented $2$-dimensional Riemann surface with (possibly empty) boundary. Fix a non-empty set $M$ of marked points
in the closure of $S$ with at least one marked point on each boundary component. The pair $(S,M)$
is called a {\bf bordered surface with marked points}. Marked points in the interior of $S$ are called
{\bf punctures}.
\end{Definition}
In this paper, we will only consider surfaces $(S, M)$ such that all marked points
lie on the boundary of S, and we will refer to $(S, M)$ simply as an {\bf unpunctured surface}.
For technical reasons, we require that $(S,M)$ if $S$ is a polygon, then $M$ has at least $4$ marked points on the boundary of $S$.

An {\bf arc} $\gamma$ in  $(S,M)$ is a curve in $S$, considered up to isotopy, satisfying
that
\begin{itemize}
\item the endpoints of $\gamma$ are in $M$; $\gamma$ does not cross itself, except that its endpoints may coincide;

\item except for the endpoints, $\gamma$ is disjoint from $M$ and from the boundary of $S$;
\item $\gamma$ does not cut
out a monogon or a bigon.
\end{itemize}

A {\bf boundary arc} is a curve which lies in the boundary of $S$ and connects two
marked points without passing through a third.

\begin{Definition}
(Compatibility of arcs) Two arcs are called {\bf compatible} if they do not
intersect in the interior of $S$; more precisely, there are curves in their respective isotopy
classes which do not intersect in the interior of $S$.
\end{Definition}

A {\bf triangulation} of $(S,M)$ is a  maximal collection of compatible arcs together with all
boundary arcs. The arcs of a triangulation
cut the surface into {\em triangles}.

Note that each triangulation of $(S,M)$ has $n+m$ elements, $n$ of which are arcs in $S$, and the remaining $m$ elements are boundary arcs. The number of boundary arcs $m$ is equal to the number of marked points and $n=6g+3b+m-6$, where $g$ is the genus of $S$, $b$ is the number of boundary components (see \cite[Proposition 2.10]{FST}).

Choose any triangulation $T$ of $(S,M)$, let $\tau_1,\cdots,\tau_n$ be the $n$ interior arcs of $T$ and denote the $m$ boundary arcs of $(S,M)$ by $\tau_{n+1},\cdots,\tau_{n+m}$. For any triangle $\triangle$ in $T$, define a matrix
$B^\triangle=(b_{ij}^\triangle)_{1\leq i,j\leq n}$ by
\begin{eqnarray}
b_{ij}^\triangle=\begin{cases}1&\text{if } \tau_i \text{ and } \tau_j \text{ are sides of } \triangle \text{ with } \tau_j \text{ following } \tau_i \text{ in the clockwise order};\\
-1&\text{if } \tau_i \text{ and } \tau_j \text{ are sides of } \triangle \text{ with } \tau_j \text{ following } \tau_i \text{ in the anticlockwise order};\\
0& \text{otherwise}.\end{cases}\nonumber
\end{eqnarray}
Then define the {\bf signed adjacency matrix $B_T=(b_{ij}^T)_{1\leq i,j\leq n}$ of $T$} by $b_{ij}^T=\sum\limits_\triangle b_{ij}^\triangle$, where  the sum is taken over all
triangles in $T$. Note that the matrix $B_T$ is skew-symmetric and each of its entries $b_{ij}^T$ is
either $0,-1,1,-2$, or $2$.

For each $k=1,\cdots,n$, there is a unique quadrilateral in $T\backslash\{\tau_k\}$  such that $\tau_k$ is one of the
diagonals of this quadrilateral. Let $\tau_k^\prime$ be the other diagonal in that quadrilateral. Define the {\bf flip} $\mu_k(T)$ to be the new triangulation $T\backslash\{\tau_k\}\cup \{\tau_k^\prime\}$.

\begin{Theorem}(\cite[Theorem 7.11]{FST} and \cite[Theorem 5.1]{FT})\label{thmbijection}
 Fix an unpunctured surface $(S,M)$ with an initial triangulation $T=\{\tau_1,\cdots,\tau_n,\tau_{n+1},\cdots,\tau_{n+m}\}$, where $\tau_{n+1},\cdots,\tau_{n+m}$ are boundary arcs.
Let $\mathcal A(B_T)$ be a cluster algebra with the initial exchange matrix $B_T$. Then there exists a bijection
$x$ from the set of arcs of $(S,M)$ to the set of cluster variables of  $\mathcal A(B_T)$ given by $\gamma\mapsto x_\gamma$, which induces a bijection ${\bf x}$ from the set of
triangulations  of $(S,M)$ to the set of  clusters of $\mathcal A(B_T)$ given by $T^\prime=\{\tau_1^\prime,\cdots,\tau_n^\prime,\tau_{n+1},\cdots,\tau_{n+m}
\}\mapsto{\bf x}_{T^\prime}:=\{x_{\tau_1^\prime},\cdots,x_{\tau_n^\prime}\}$
. Furthermore,  the flips in $(S,M)$ are compatible with the mutations in $\mathcal A(B_T)$ under the map ${\bf x}$.
\end{Theorem}

\subsection{$g$-vectors of arcs}
We refer to \cite{R} for this subsection. Since the signed adjacency matrix $B_T$ in this paper is $-B_T$ in \cite{R},
some notations  in this subsection may be not the same with that in \cite{R}.

Let  $T$ be a triangulation of $(S,M)$, $\gamma$ be any arc in $(S,M)$ that crosses $T$ exactly $d$ times. We  fix an orientation for $\gamma$
and we denote its starting point by $s$ and its endpoint by $r$, with $s,r\in M$.  Let
$s=p_0, p_1,\cdots, p_d, p_{d+1}=r$ be the intersection points of $\gamma$ and $T$ in order of occurrence on $\gamma$, and $\tau_{i_k}\in T$ be the arc crossing with $\gamma$ at $p_k$ for $k=1,2,\cdots,d$. One can refer to Figure \ref{fig1} for an illustration.

For $k=0,1,\cdots,d$, let $\gamma_k$ denote the segment of the path $\gamma$ from the point $p_k$ to the point
$p_{k+1}$. Each $\gamma_k$ lies in exactly one triangle $\triangle_k$ in $T$. If $1\leq k\leq d-1$, the triangle $\triangle_k$ is formed by
the arcs $\tau_{i_k},\tau_{i_{k+1}}$ and a third arc that we denote by $\tau_{[\gamma_k]}$.
In the triangle $\triangle_0$, $\tau_{i_1}$ is one of the sides.
Denote the side of $\triangle_0$ that lies clockwise of $\tau_{i_1}$ by $\tau_{i_0}$ and the other side by $\tau_{[\gamma_0]}$. Similarly,
$\tau_{i_d}$ is one of the sides of $\triangle_d$. Denote the side that lies clockwise of $\tau_{i_d}$ by $\tau_{i_{d+1}}$  and the other side
by $\tau_{[\gamma_d]}$.

\begin{figure}
\centering
\begin{tikzpicture}
[xscale=0.9,yscale=0.8,> = stealth, 
	shorten < = -8pt,
shorten >=-8pt]

\draw  (0,0) ellipse (6 and 4);

\node (v1)at (-4.5,2.6) {$\bullet$};
\node (v2)at (0,4) {$\bullet$};
\node (v3)at (4.5,2.6) {$\bullet$};
\node (v4)at (6,0) {$\bullet$};
\node (v6)at (3,-3.5) {$\bullet$};
\node (v7)at (-3,-3.5) {$\bullet$};
\node (v8)at (-6,0) {$\bullet$};
\node at (-6.8,0){$s=p_0$};
\node at (6.8,0){$r=p_5$};
\draw  (v7)--node[very near end,right]{$\tau_{i_1}$}(v1);
\draw(v7)--node[very near end,right]{$\tau_{i_2}$}(v2);
\draw(v7)--node[near start,right]{$\tau_{i_3}$}(v3);
\draw(v3)--(v6)node[very near end,left]{$\tau_{i_4}$};
\draw[red](v8)--(v4);

\node (p1) at (-3.85,0) {$\bullet$};
\node at (-3.60,0.2) {$p_1$};

\node (p2)at (-1.59,0) {$\bullet$};
\node at (-1.79,0.2) {$p_2$};

\node (p3)at (1.32,0) {$\bullet$};
\node at (1.12,0.2) {$p_3$};

\node at (3.85,0) {$\bullet$};
\node at (4.2,0.2) {$p_4$};

\node at (-6.3,1.3) {$\tau_{\gamma_{[0]}}$};
\node at (-2.5,4) {$\tau_{\gamma_{[1]}}$};

\node at (2.5,4) {$\tau_{\gamma_{[2]}}$};
\node at (6,1.5) {$\tau_{i_5}$};

\node at (5.5,-2.5) {$\tau_{\gamma_{[4]}}$};

\node at (-0.5,-4.3) {$\tau_{\gamma_{[3]}}$};

\node at (-5.8,-1.7) {$\tau_{i_0}$};

\node at (-4.5,-1.5) {$\triangle_0$};
\node at (-2.5,2) {$\triangle_1$};
\node at (1,2) {$\triangle_2$};

\node at (1.5,-2) {$\triangle_3$};
\node at (4.5,-1.5) {$\triangle_4$};
\node at (-5,-0.2) {$\gamma_0$};

\node at (-2.5,-0.2) {$\gamma_1$};
\node at (-0.5,-0.2) {$\gamma_2$};
\node at (2.5,-0.2) {$\gamma_3$};
\node at (5,-0.2) {$\gamma_4$};
\end{tikzpicture}
\caption{}\label{fig1}
\end{figure}

A {\bf $T$-path} is a path $\alpha$ in $S$ on the triangulation $T$, that is, there exist arcs $\alpha_1,\cdots,\alpha_l\in T$ such that $\alpha$ is the concatenation of paths $\alpha=\alpha_1\alpha_2\cdots\alpha_l$.

A $T$-path $\alpha=\alpha_1\cdots\alpha_l$  is called a {\bf complete $(T,\gamma)$-path} if the
following axioms hold:

(T1) The even arcs are precisely the arcs crossed by $\gamma$ in order, that is, $\alpha_{2k}=\tau_{i_k}$.

(T2) For all $k=0,1,\cdots,d$, the segment $\gamma_k$ is homotopic to the segment of the path $\alpha$ starting
at the point $p_k$ following $\alpha_{2k}=\tau_{i_k}, \alpha_{2k+1}$, and $\alpha_{2k+2}=\tau_{i_{k+1}}$ until the point $p_{k+1}$.

Note that every complete $(T,\gamma)$-path has length $2d+1$ (i.e., $l=2d+1$) and starts and ends at the same point as $\gamma$.

The triangulation $T$ cuts $S$ into triangles. For each triangle $\triangle$ in $T$, fix the orientation of $\triangle$ by clockwise orientation when looking at
it from outside the surface.

\begin{Remark}\label{rmknoteq}
Since $(S,M)$ is an unpunctured surface, the three sides of each triangle $\triangle$ in $T$ are distinct. In particular, $\tau_{i_k}\neq \tau_{i_{k+1}}$ for $k=0,1,\cdots,d$, since they are two sides of the triangle $\triangle_k$, which means that $\alpha_{2k}\neq \alpha_{2k+2}$ in a complete $(T,\gamma)$-path $\alpha=\alpha_1\alpha_2\cdots\alpha_{2d+1}$.
\end{Remark}

\begin{Theorem}\cite[Theorem 6.2]{R} Let $T=\{\tau_1,\cdots,\tau_n,\tau_{n+1},\cdots,\tau_{n+m}\}$ be a triangulation of an unpunctured surface
$(S,M)$, and let $\gamma$ be an arc. Then there is precisely one complete $(T,\gamma)$-path $\alpha_\gamma^T=\alpha_1\cdots\alpha_{2d+1}$ satisfying that the orientation of $\alpha_{2k}=\tau_{i_k}$ induced by $\alpha$ coincides with the orientation of $\tau_{i_k}$ in the triangle $\triangle_k$. We call $\alpha_\gamma^T$ the minimal $(T,\gamma)$-path.
\end{Theorem}

\begin{Remark}\label{rmkmin}
Let $\alpha_\gamma^T=\alpha_1\cdots\alpha_{2d+1}$ be the minimal $(T,\gamma)$-path. There are the following facts from the proof of \cite[Theorem 6.2]{R}.

\begin{itemize}
 \item $\alpha_1=\tau_{i_0}$ and $\alpha_{2d+1}=\tau_{i_{d+1}}$;

\item  $(\alpha_{2k-1},\alpha_{2k},\alpha_{2k+1})$ must be one of the following cases:

$$(\tau_{i_{k-1}},\tau_{i_k},\tau_{i_k}),\;\;(\tau_{i_{k-1}},\tau_{i_k},\tau_{i_{k+1}}),\;\;
(\tau_{i_{k}},\tau_{i_k},\tau_{i_k}),\;\;(\tau_{i_{k}},\tau_{i_k},\tau_{i_{k+1}}),$$
where $k=1,2,\cdots,d$.
 \end{itemize}
\end{Remark}

Keep the above notations, for $k=1,\cdots,d$, we define three subsets $I_{\tau_{i_k}}^+,I_{\tau_{i_k}}^-,I_{\tau_{i_k}}^0$ of the set of crossing points $\{p_1,p_2,\cdots,p_d\}$ identifying it with $[1,d]:=\{1,2,\cdots,d\}$ given by
\begin{eqnarray}
I_{\tau_{i_k}}^+&:=&\{j| \tau_{i_j}=\tau_{i_k},\text{ and } (\alpha_{2j-1},\alpha_{2j},\alpha_{2j+1})=(\tau_{i_j},\tau_{i_j},\tau_{i_j})\}\subseteq [2,d-1],\nonumber\\
I_{\tau_{i_k}}^-&:=&\{j| \tau_{i_j}=\tau_{i_k},\text{ and } (\alpha_{2j-1},\alpha_{2j},\alpha_{2j+1})=(\tau_{i_{j-1}},\tau_{i_j},\tau_{i_{j+1}})\}\subseteq [1,d],\nonumber\\
I_{\tau_{i_k}}^{0,<}&:=&\{j| \tau_{i_j}=\tau_{i_k},\text{ and } (\alpha_{2j-1},\alpha_{2j},\alpha_{2j+1})=(\tau_{i_j},\tau_{i_j},\tau_{i_{j+1}})\}\subseteq [2,d],\nonumber\\
I_{\tau_{i_k}}^{0,>}&:=&\{ j| \tau_{i_j}=\tau_{i_k},\text{ and } (\alpha_{2j-1},\alpha_{2j},\alpha_{2j+1})=(\tau_{i_{j-1}},\tau_{i_j},\tau_{i_{j}})\}\subseteq [1,d-1]
.\nonumber
\end{eqnarray}

The following are the local shapes of the crossings of $\gamma$ and $T$ at $p_j$ and $p_{l}$ for $j\in I_{\tau_k}^+$ and for $l\in I_{\tau_k}^-$.

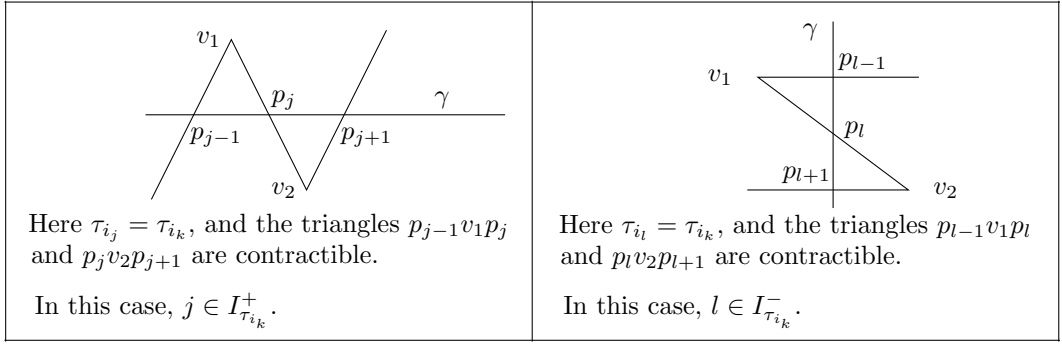
\begin{tikzpicture}
[xscale=1,yscale=1,> = stealth, 
	shorten < = -4pt,
shorten >=-4pt]
\node (v1) at (-7,4) {};
\node (v2) at (0,4) {};
\node (v3) at (7,4) {};
\node (v4) at (-7,-0.5) {};
\node (v5) at (0,-0.5) {};
\node (v6) at (7,-0.5) {};
\draw(v1)--(v3);
\draw(v3)--(v6);
\draw(v6)--(v4);
\draw(v4)--(v1);
\draw(v2)--(v5);

\draw(-5,1.5)--(-4,3.5)--(-3,1.5)--(-2,3.5);
\draw(-5,2.5)--(-0.5,2.5);

\draw(5,3)--(3,3)--(5,1.5)--(3,1.5);
\draw(4,3.6)--(4,1.4);

\node at (-4.2,2.2) {$p_{j-1}$};
\node at (-3.3,2.7) {$p_j$};
\node at (-2.2,2.2) {$p_{j+1}$};
\node at (-1.2,2.7) {$\gamma$};
\node at (-4.3,3.5) {$v_1$};
\node at (-3.3,1.5) {$v_2$};

\node at (-3.5,1) {Here $\tau_{i_j}=\tau_{i_k}$, and the triangles $p_{j-1}v_1p_j$};

\node at (-4.35,0.6) { and $p_jv_2p_{j+1}$ are contractible.};

\node at (-5,-0.1) {In this case, $j\in I_{\tau_{i_k}}^+$.};

\node at (3.7,3.6) {$\gamma$};
\node at (4.4,3.2) {$p_{l-1}$};

\node at (4.3,2.3) {$p_{l}$};

\node at (3.65,1.7) {$p_{l+1}$};
\node at (2.5,3) {$v_1$};
\node at (5.5,1.5) {$v_2$};

\node at (3.5,1) {Here $\tau_{i_l}=\tau_{i_k}$, and the triangles $p_{l-1}v_1p_l$};

\node at (2.65,0.6) { and $p_lv_2p_{l+1}$ are contractible.};

\node at (2,-0.1) {In this case, $l\in I_{\tau_{i_k}}^-$.};

\end{tikzpicture}

Define ${\bf e}_{\tau_1}, {\bf e}_{\tau_2},\cdots, {\bf e}_{\tau_n}$  to be the standard basis vectors of $\mathbb Z^n$, and let ${\bf e}_{\tau_j}$ be the zero vector in $\mathbb Z^n$  if $\tau_j$ is a
boundary arc, i.e., $j=n+1,\cdots,n+m$.

\begin{Definition}(i) Let $\alpha_\gamma^T=\alpha_1\cdots\alpha_{2d+1}$ be the minimal $(T,\gamma)$-path, the
{\bf $g$-vector $g_\gamma^T$ of $\gamma$ with respect to $T$} is defined to be the vector $g_\gamma^T=\sum\limits_{i=1}^{2d+1}(-1)^{i+1}{\bf e}_{\alpha_i}={\bf e}_{\tau_{i_0}}+{\bf e}_{\tau_{i_{d+1}}}+\sum\limits_{i=2}^{2d}(-1)^{i+1}{\bf e}_{\alpha_i}\in\mathbb Z^n$.

(ii) Let $T^\prime=\{\gamma_1,\cdots,\gamma_n,\tau_{n+1},\cdots,\tau_{n+m}\}$ be another triangulation of $(S,M)$, the matrix $$G_{T^\prime}^T=(g_{\gamma_1}^T,\cdots,g_{\gamma_n}^T)$$ is called the {\bf $G$-matrix of $T^\prime$ with respect to $T$}.
\end{Definition}

\begin{Example}
In the Figure \ref{fig1}, $\alpha=\tau_{i_0}\tau_{i_1}\tau_{i_1}\tau_{i_2}\tau_{i_2}\tau_{i_3} \tau_{i_4}\tau_{i_4}\tau_{i_5}$ is the minimal  complete $(T,\gamma)$-paths, and the $g$-vector of $\gamma$ with respect to $T$ is $g_\gamma^T=-{\bf e}_{\tau_{i_3}}$.
\end{Example}

\begin{Proposition}\label{proi+}
Keep the above notations. The coefficient of ${\bf e}_{\tau_{i_k}}$ in $\sum\limits_{i=2}^{2d}(-1)^{i+1}{\bf e}_{\alpha_i}$ is $$|I_{\tau_{i_k}}^+|-|I_{\tau_{i_k}}^-|.$$
\end{Proposition}
\begin{proof}
Consider the set
$$I_{\tau_{i_k}}=\{i\in[2,2d]|\alpha_i=\tau_{i_k}\}.$$
An easy fact from Remark \ref{rmknoteq} and Remark \ref{rmkmin} is that
\begin{eqnarray}
I_{\tau_{i_k}}=\left(\bigsqcup\limits_{j\in I_{\tau_{i_k}}^+}\{2j-1,2j,2j+1\}\right)\bigsqcup\left(\bigsqcup\limits_{j\in I_{\tau_{i_k}}^-}\{2j\}\right)\bigsqcup\left(\bigsqcup\limits_{j\in I_{\tau_{i_k}}^{0,<}}\{2j-1,2j\}\right)\bigsqcup\left(\bigsqcup\limits_{j\in I_{\tau_{i_k}}^{0,>}}\{2j,2j+1\}\right).\nonumber
\end{eqnarray}
We know that $\sum\limits_{i\in[2,d] \text{ and }\alpha_i=\tau_{i_k}}(-1)^{i+1}{\bf e}_{\alpha_i}=\sum\limits_{i\in I_{\tau_{i_k}}}(-1)^{i+1}{\bf e}_{\tau_{i_k}}$ and
\begin{eqnarray}
&&\sum\limits_{i\in I_{\tau_{i_k}}}(-1)^{i+1}{\bf e}_{\tau_{i_k}}\nonumber\\
&=&\sum\limits_{j\in I_{\tau_{i_k}}^+}({\bf e}_{\tau_{i_k}}-{\bf e}_{\tau_{i_k}}+{\bf e}_{\tau_{i_k}})+\sum\limits_{j\in I_{\tau_{i_k}}^-}(-{\bf e}_{\tau_{i_k}})+\sum\limits_{j\in I_{\tau_{i_k}}^{0,<}}({\bf e}_{\tau_{i_k}}-{\bf e}_{\tau_{i_k}})
+\sum\limits_{j\in I_{\tau_{i_k}}^{0,>}}(-{\bf e}_{\tau_{i_k}}+{\bf e}_{\tau_{i_k}})\nonumber\\
&=&\left(|I_{\tau_{i_k}}^+|-|I_{\tau_{i_k}}^-|\right){\bf e}_{\tau_{i_k}}.\nonumber
\end{eqnarray}
So the coefficient of ${\bf e}_{\tau_{i_k}}$ in $\sum\limits_{i=2}^{2d}(-1)^{i+1}{\bf e}_{\alpha_i}$ is $|I_{\tau_{i_k}}^+|-|I_{\tau_{i_k}}^-|$.
\end{proof}
\begin{Corollary}\label{corgexp}
Keep the above notations.
Denote by $I^-(T,\gamma)=\bigcup\limits_{k=1}^d I_{\tau_{i_k}}^-$ and $$I^+(T,\gamma)=\left(\bigcup\limits_{k=1}^d I_{\tau_{i_k}}^+\right)\bigcup\{j|j\in\{0,d+1\} \text{and } \tau_{i_j} \text{ is not a bound arc.}\}.$$
Then the $g$-vector $g_\gamma^T$ of $\gamma$ with respect to $T$ is
$$g_\gamma^T=\sum\limits_{j\in I^+(T,\gamma)}{\bf e}_{\tau_{i_j}}-\sum\limits_{j\in I^-(T,\gamma)}{\bf e}_{\tau_{i_j}}.$$
\end{Corollary}
\begin{proof}
Note that $I_{\tau_{i_k}}^\pm \cap I_{\tau_{i_j}}^\pm=\phi$ for any $\tau_{i_j}\neq \tau_{i_k}$.
The result follows from the definition of $g$-vectors and Proposition \ref{proi+}.
\end{proof}

Let $T^\prime$ be another triangulation of $(S,M)$, we denote by
\begin{eqnarray}
I^+(T,T^\prime):=\bigsqcup\limits_{\gamma\in T^\prime}I^+(T,\gamma)\;\;\text{ and }\;\;I^-(T,T^\prime):=\bigsqcup\limits_{\gamma\in T^\prime}I^-(T,\gamma).\nonumber
\end{eqnarray}
Note that the elements in $I^{\pm}(T,\gamma)$ (respectively, $I^{\pm}(T,T^\prime)$) represents certain crossing points between $\gamma$ (respectively, $T^\prime$) and $T$. For $j\in I^\pm(T,\gamma)$ (respectively, $j\in I^\pm(T,T^\prime)$), denote by $\tau_{i_j}$ be the arc in $T$ crossing with $\gamma$ (respectively, $T^\prime$) at $j$.
\begin{Theorem}\label{thmempty}
Keep the above notations.

(i) for any $j\in I^+(T,\gamma)$ and $l\in I^-(T,\gamma)$, we have $\tau_{i_j}\neq \tau_{i_l}$;

(ii) for any $j\in I^+(T,T^\prime)$ and $l\in I^-(T,T^\prime)$, we have $\tau_{i_j}\neq \tau_{i_l}$.
\end{Theorem}
\begin{proof}
(i) Assume by contradiction that there exist $j\in I^+(T,\gamma)$ and $l\in I^-(T,\gamma)$ such that $\tau_{i_j}=\tau_{i_l}$. There is an unique quadrilateral $\diamondsuit$ in $T$ such that $\tau_{i_j}=\tau_{i_l}$ is one of the diagonals of $\diamondsuit$. Then the picture is the following.
$$
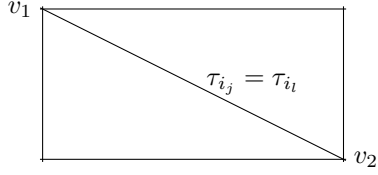
\begin{tikzpicture}
[xscale=1,yscale=0.5,> = stealth, 
	shorten < = -1pt,
shorten >=-1pt]

\draw (0,0)--(0,4);
\draw(0,4)--(4,4);
\draw(4,4)--(4,0);
\draw(4,0)--(0,0);

\draw(0,4)--(4,0);
\node at (2.8,2) {$\tau_{i_j}=\tau_{i_l}$};
\node at (-0.3,4) {$v_1$};
\node at (4.3,0) {$v_2$};
\end{tikzpicture}
$$

By the above picture, we know that $j\in I^+(T,\gamma)$ and $l\in I^-(T,\gamma)$ with $\tau_{i_j}=\tau_{i_l}$ will result in that $\gamma$ has self-crossing, which is a contradiction. So for any $j\in I^+(T,\gamma)$ and $l\in I^-(T,\gamma)$, we have $\tau_{i_j}\neq \tau_{i_l}$.

(ii) The proof is similar to (i). Assume by contradiction that there exist $j\in I^+(T,T^\prime)$ and $l\in I^-(T,T^\prime)$ such that $\tau_{i_j}=\tau_{i_l}$. This will result in the arcs in $T^\prime$ have at least one crossing, which is a contradiction. So for any $j\in I^+(T,T^\prime)$ and $l\in I^-(T,T^\prime)$, we have $\tau_{i_j}\neq \tau_{i_l}$.
\end{proof}
\begin{Corollary}\label{corgcom}
Let $g_\gamma^T=(g_1,\cdots,g_n)^{\rm T}$ be the $g$-vector of $\gamma$ with respect to $T$. Then

(i) $g_k>0$ if and only if there exists a crossing point $j\in I^+(T,\gamma)$ such that $\gamma$ crosses with $\tau_k$ at $j$.

(ii) $g_k<0$ if and only if there exists a crossing point $j\in I^-(T,\gamma)$ such that $\gamma$ crosses with $\tau_k$ at $j$.
\end{Corollary}
\begin{proof}
It follows from Theorem \ref{thmempty} and Corollary \ref{corgexp}.
\end{proof}

\begin{Corollary}(Row sign-coherence of $G$-matrices) Let $T^\prime=\{\gamma_1,\cdots,\gamma_n,\tau_{n+1},\cdots,\tau_{n+m}\}$ be another triangulation of $(S,M)$, and $G_{T^\prime}^T=(g_{ij;T^\prime}^T)$ be the $G$-matrix of $T^\prime$ with respect to $T$. Then
 each row vector of $G_{T^\prime}^T$ is either a non-negative vector or a non-positive vector.
\end{Corollary}
\begin{proof}
Assume by contradiction that there exist $g_{kp;T^\prime}^T$ and $g_{kq;T^\prime}^T$ in the $k$-th row of $G_{T^\prime}^T$ such that $g_{kp;T^\prime}^T>0$ and $g_{kq;T^\prime}^T<0$. By Corollary \ref{corgcom}, there exist
$j\in I^+(T,\gamma_p)$ and $l\in I^-(T,\gamma_q)$ such that $\tau_k$ crosses with $\gamma_p$  at $j\in I^+(T,\gamma_p) \subseteq I^+(T,T^\prime)$ and  crosses with $\gamma_q$  at $l\in I^-(T,\gamma_q)\subseteq I^-(T,T^\prime)$, which contradicts Theorem \ref{thmempty} (ii). So each row vector of $G_{T^\prime}^T=(g_{ij;T^\prime}^T)$ is either a non-negative vector or a non-positive vector.
\end{proof}

\begin{Theorem}\cite[Theorem 6.4]{R} \label{thmarcgvector}Let $T=\{\tau_1,\cdots,\tau_n,\tau_{n+1},\cdots,\tau_{n+m}\}$ be a triangulation of an unpunctured surface $(S,M)$, and let $\gamma$ be an arc. Then the $g$-vector  of $\gamma$ with respect to $T$ is
equal to the $g$-vector of  the cluster variable $x_\gamma$ with respect to the cluster
${\bf x}_T=\{x_{\tau_1},\cdots,x_{\tau_n}\}$,
where $x$ and ${\bf x}$ are the bijections in Theorem \ref{thmbijection}.
\end{Theorem}

\subsection{Co-Bongartz completions on  unpunctured surfaces}
In this subsection, we give the definition of the elementary co-Bongartz completion of an arc $\beta$ with respect to another arc $\alpha$, which is used to construct the co-Bongartz completions on unpunctured surfaces.

 {\bf Definition-Construction:} Let $(S,M)$ be an unpunctured surface, and $\alpha,\beta$ be any two (boundary) arcs of $(S,M)$.

(i) If $\alpha$ and $\beta$ are compatible, the {\bf elementary co-Bongartz completion} of $\beta$ with respect to $\alpha$ is defined to be $\mathcal T_\beta(\alpha):=\{\beta\}\cup\{\alpha\}$.

(ii) If $\alpha$ and $\beta$ are not compatible, then they have some crossings. We use $\beta$ to cut $\alpha$ into several segments $\{\alpha_i\}$. For each segment $\alpha_i$, we deform it a new (boundary) arc $\alpha_i^\prime$. The deformed principle is that we go along $\alpha_i$ and clockwise to $\beta$ when meeting the crossings between $\alpha$ and $\beta$. The obtained route corresponding a curve with endpoints in $M$, which is denoted by $\alpha_i^\prime$. Then the {\bf elementary co-Bongartz completion} of $\beta$ with respect to $\alpha$ is defined to be $\mathcal T_\beta(\alpha):=\{\beta\}\cup\{\alpha_i^\prime\}$.

Now we illustrate the case that $\alpha$ and $\beta$ are not compatible in details. Assume that $\alpha$ crosses $\beta$ exactly $d$ times ($d\geq 1$).
We  fix an orientation for $\alpha$
and we denote its starting point by $s$ and its endpoint by $r$, with $s,r\in M$.  Let
$s=p_0, p_1,\cdots, p_d, p_{d+1}=r$ be the intersection points of $\alpha$ and $\beta$ in order of occurrence on $\alpha$.
For $k=0,1,\cdots,d$, let $\alpha_k$ denote the segment of the path $\alpha$ from the point $p_k$ to the point
$p_{k+1}$, and $E(\beta)$ be the set of endpoints of $\beta$.
One can refer to Figure \ref{fig2} for an illustration. For each $\alpha_k,\;k=0,\cdots,d$, we define an (boundary) arc $\alpha_k^\prime$ associate with $\alpha_k$  given by

\begin{itemize}
\item For $1\leq k\leq d-1$, let $\alpha_k^\prime$ be the (boundary) arc isotopy to the curve
$$(e_1\xrightarrow[]{\text{along }\beta} p_k)\text{ then anticlockwise to } (p_k\xrightarrow[]{\text{along }\alpha}p_{k+1})\text{ then clockwise to } (p_{k+1}\xrightarrow[]{\text{along }\beta}e_2),$$
where $e_1, e_2\in E(\beta)$ and the sub-curves $(e_1\xrightarrow[]{\text{along }\beta} p_k)$ and $(p_{k+1}\xrightarrow[]{\text{along }\beta}e_2)$ are uniquely determined by the ``anticlockwise direction and clockwise direction" in the above curve.

 \item For $k=0$, let $\alpha_0^\prime$ be the (boundary) arc isotopy to the curve

$$(s=p_0\xrightarrow[]{\text{along }\alpha}p_1) \text{ then clockwise to }(p_1\xrightarrow[]{\text{along }\beta} e_2),$$
where $e_2\in E(\beta)$ and the sub-curve $(p_1\xrightarrow[]{\text{along }\beta} e_2)$ are uniquely determined by the ``clockwise direction" in the above curve.

    \item For $k=d$, let $\alpha_d^\prime$ be the (boundary) arc isotopy to the curve

    $$(e_1\xrightarrow[]{\text{along }\beta} p_d) \text{ then anticlockwise to } (p_d\xrightarrow[]{\text{along }\alpha}p_{d+1}=r),$$
    where $e_1\in E(\beta)$ and the sub-curve $(e_1\xrightarrow[]{\text{along }\beta} p_d)$ are uniquely determined by the ``anticlockwise direction" in the above curve.

\end{itemize}
 In the case that $\alpha$ and $\beta$ are not compatible, the {\bf elementary co-Bongartz completion} of $\beta$ with respect to $\alpha$ is defined to be the set of (boundary) arcs  $$\mathcal T_\beta(\alpha):=\{\beta\}\cup \{\alpha_0^\prime,\cdots,\alpha_d^\prime\}.$$
 One can refer to Figure \ref{fig3} for  an illustration.

\begin{figure}
\centering
\begin{tikzpicture}[xscale=0.65,yscale=0.65,> = stealth 
	]

\node (n1)at (-10,0) {$\bullet$};
\node (n2)at (10,0) {$\bullet$};

\draw [red,shorten < =-8 pt,
shorten >=-8 pt] (n1)--(n2);

\draw  plot[smooth, tension=.7] coordinates {(-8.5,4) (-4,-4) (0,4) (4.5,-4) (8.5,4)};
\node at (-8.5,3.8) {$\bullet$};
\node at (8.5,3.8) {$\bullet$};

\node at (-6.9,0) {$\bullet$};
\node at (-7,-0.3) {$p_1$};
\node at (-2,0) {$\bullet$};
\node at (-1.8,-0.3) {$p_2$};
\node at (2.23,0) {$\bullet$};
\node at (2.13,-0.3) {$p_3$};
\node at (7.18,0) {$\bullet$};
\node at (7.3,-0.3) {$p_4$};
\node at (-8.5,4.1) {$s=p_0$};
\node at(8.6,4.1) {$r=p_5$};
\node at (-7.7,3) {$\alpha_0$};
\node at (-3.8,-4.3) {$\alpha_1$};
\node at (0,4.3) {$\alpha_2$};
\node at (4.5,-4.3) {$\alpha_3$};
\node at (8.2,1.5) {$\alpha_4$};
\draw[very thick] (-4.2,-1.5) ellipse (0.5 and 0.5);
\draw[very thick] (0.1,2) ellipse (0.5 and 0.5);
\draw[very thick](4.5,-2.4) ellipse (0.5 and 0.5);
\filldraw[fill=green!20,draw=black]  (-4.2,-1.5) ellipse (0.5 and 0.5);
\filldraw[fill=green!20,draw=black] (0.1,2) ellipse (0.5 and 0.5);
\filldraw[fill=green!20,draw=black]   (4.5,-2.4) ellipse (0.5 and 0.5);

\node at (-9,-0.4) {$\beta$};
\node at (8.5,2.5) {$\alpha$};
\end{tikzpicture}

  \caption{}\label{fig2}
\end{figure}

\begin{figure}
\centering
\begin{tikzpicture}[xscale=0.65,yscale=0.65,> = stealth 
	]

\node (n1)at (-10,0) {$\bullet$};
\node (n2)at (10,0) {$\bullet$};

\draw [red,shorten < = -8pt,
shorten >=-8pt] (n1)--(n2);

\draw  plot[smooth, tension=.7] coordinates {(-8.5,4) (-4,-4) (0,4) (4.5,-4) (8.5,4)};
\node at (-8.5,3.8) {$\bullet$};
\node at (8.5,3.8) {$\bullet$};

\node at (-6.9,0) {$\bullet$};
\node at (-7,-0.3) {$p_1$};
\node at (-2,0) {$\bullet$};
\node at (-1.8,-0.3) {$p_2$};
\node at (2.23,0) {$\bullet$};
\node at (2.13,-0.3) {$p_3$};
\node at (7.18,0) {$\bullet$};
\node at (7.3,-0.3) {$p_4$};
\node at (-8.5,4.1) {$s=p_0$};
\node at(8.6,4.1) {$r=p_5$};
\node at (-7.7,3) {$\alpha_0$};
\node at (-3.8,-4.3) {$\alpha_1$};
\node at (0,4.3) {$\alpha_2$};
\node at (4.5,-4.3) {$\alpha_3$};
\node at (8.2,1.5) {$\alpha_4$};
\draw[very thick] (-4.2,-1.5) ellipse (0.5 and 0.5);
\draw[very thick] (0.1,2) ellipse (0.5 and 0.5);
\draw[very thick](4.5,-2.4) ellipse (0.5 and 0.5);
\filldraw[fill=green!20,draw=black]  (-4.2,-1.5) ellipse (0.5 and 0.5);
\filldraw[fill=green!20,draw=black] (0.1,2) ellipse (0.5 and 0.5);
\filldraw[fill=green!20,draw=black]   (4.5,-2.4) ellipse (0.5 and 0.5);

\node at (-9,-0.4) {$\beta$};
\node at (8.5,2.5) {$\alpha$};

\draw[blue] plot[smooth, tension=.7] coordinates { (-10,0) (-6.5,0.3) (-1,0.35)
(2.9,0.8) (8.52,4)
};

\draw[blue]  plot[smooth, tension=.7] coordinates {(-10,0) (-6,0.7) (-0.5,1) (1,2)
(0,2.8) (-3,2.5)  (-7,1.5)  (-10,0)};

\draw[blue]  plot[smooth, tension=.7] coordinates{(-10,0) (-8.52,3.9)};

\draw[blue]  plot[smooth, tension=.7] coordinates { (10,0) (7.5,-0.5) (1.5,-0.5)
(-4,-0.5)
(-5.5,-1.5) (-4,-2.5) (-1,-1.4) (2,-1) (7.5,-0.8) (10,0)};

\draw[blue]  plot[smooth, tension=.7] coordinates {(10,0) (8.5,-1.5) (5,-1.5)
(3.5,-2) (4,-3.3) (5.2,-3.1)   (9,-2)(10,0)};

\node at (-9.5,2.5) {$\alpha_0^\prime$};
\node at (-4.5,2.7) {$\alpha_2^\prime$};
\node at (0,-1.5) {$\alpha_1^\prime$};
\node at (7,-3.3) {$\alpha_3^\prime$};

\node at (5.5,2.5) {$\alpha_4^\prime$};
\end{tikzpicture}

  \caption{}\label{fig3}
\end{figure}

\begin{Lemma}\label{lemcomp1}
Keep the notations above. Then

(i) each  arc in $\mathcal T_\beta(\alpha)\backslash\{\beta\}$ is compatible with $\beta$;

(ii) any two arcs in $\mathcal T_\beta(\alpha)$ are compatible.
\end{Lemma}
\begin{proof}
(i) If $\alpha$ is an (boundary) arc compatible with $\beta$, then $\mathcal T_\beta(\alpha)=\{\beta\}\cup \{\alpha\}$. In this case, the result is clear. Now we assume that $\alpha$ is not compatible with $\beta$. In this case,
$\mathcal T_\beta(\alpha)=\{\beta\}\cup\{\alpha_0^\prime,\cdots,\alpha_d^\prime\}$ and each $\alpha_k^\prime$ with $1\leq k\leq d$ is  the (boundary) arc isotopy to the curve
$$\rho:=(e_1\xrightarrow[]{\text{along }\beta} p_k)\text{ then anticlockwise to } (p_k\xrightarrow[]{\text{along }\alpha}p_{k+1})\text{ then clockwise to } (p_{k+1}\xrightarrow[]{\text{along }\beta}e_2),$$
where $e_1, e_2\in E(\beta)$ and the sub-curves $(e_1\xrightarrow[]{\text{along }\beta} p_k)$ and $(p_{k+1}\xrightarrow[]{\text{along }\beta}e_2)$ are uniquely determined by the ``anticlockwise direction and clockwise direction" in the above curve.
By the construction, we know that the cure $\rho$ does not cross $\beta$  transversely, so $\alpha_k^\prime$ is compatible with $\beta$.
Similarly, $\alpha_0^\prime$ and $\alpha_d^\prime$ are compatible with $\beta$. So each  arc in $\mathcal T_\beta(\alpha)\backslash\{\beta\}$ is compatible with $\beta$.

(ii) If $\alpha$ is compatible with $\beta$, then $\mathcal T_\beta(\alpha)=\{\beta\}\cup\{\alpha\}$. In this case, the result is clear. So we can assume that $\alpha$ is not compatible with $\beta$. Keep the notations introduced before. Let $\alpha_i^\prime,\alpha_j^\prime$ be two arcs in $\mathcal T_\beta(\alpha)\backslash\{\beta\}$.
 We first consider the case $1\leq i,j\leq d-1$. In this case,
 we know that $\alpha_i^\prime$ is   isotopy to the curve
$$\rho:=(e_1\xrightarrow[]{\text{along }\beta} p_i)\text{ then anticlockwise to } (p_i\xrightarrow[]{\text{along }\alpha}p_{i+1})\text{ then clockwise to } (p_{i+1}\xrightarrow[]{\text{along }\beta}e_2),$$
and $\alpha_j^\prime$ is   isotopy to the curve
$$\rho^\prime:=(e_1^\prime\xrightarrow[]{\text{along }\beta} p_j)\text{ then anticlockwise to } (p_j\xrightarrow[]{\text{along }\alpha}p_{j+1})\text{ then clockwise to } (p_{j+1}\xrightarrow[]{\text{along }\beta}e_2^\prime),$$
where $e_1, e_2, e_1^\prime,e_2^\prime\in E(\beta)$ and the sub-curves $$(e_1\xrightarrow[]{\text{along }\beta} p_i),\;(p_{i+1}\xrightarrow[]{\text{along }\beta}e_2),\;(e_1^\prime\xrightarrow[]{\text{along }\beta} p_j),\;(p_{j+1}\xrightarrow[]{\text{along }\beta}e_2^\prime)$$
are uniquely determined by the ``anticlockwise directions and clockwise directions" in $\rho$ and $\rho^\prime$.

We know that the sub-curve $(e_1\xrightarrow[]{\text{along }\beta} p_i)$ of $\rho$ does not cross the interior of sub-curves $(e_1^\prime\xrightarrow[]{\text{along }\beta} p_j),\;(p_{j+1}\xrightarrow[]{\text{along }\beta}e_2^\prime)$ of $\rho^\prime$ transversely, since $(e_1\xrightarrow[]{\text{along }\beta} p_i)$ and  $$(e_1^\prime\xrightarrow[]{\text{along }\beta} p_j),\;(p_{j+1}\xrightarrow[]{\text{along }\beta}e_2^\prime)$$ are sub-curves of $\beta$, which has no self-crossings.

The sub-curve $(e_1\xrightarrow[]{\text{along }\beta} p_i)$ of $\rho$ also does not cross the interior of sub-curve $(p_j\xrightarrow[]{\text{along }\alpha}p_{j+1})$ transversely, since $(e_1\xrightarrow[]{\text{along }\beta} p_i)$ is a sub-curve of  $\beta$ and $p_j$ and $p_{j+1}$ are  intersection points of $\alpha$ and $\beta$ in order, and there exists no  intersection points of $\alpha$ and $\beta$  in the interior of $(p_j\xrightarrow[]{\text{along }\alpha}p_{j+1})$.

So the sub-curve  $(e_1\xrightarrow[]{\text{along }\beta} p_i)$ of $\rho$ does not cross $\rho^\prime\backslash\{p_j,p_{j+1}\}$ transversely. Similarly, we can show the sub-curve $(p_i\xrightarrow[]{\text{along }\alpha}p_{i+1})$ and the sub-curve $(p_{i+1}\xrightarrow[]{\text{along }\beta}e_2)$ of $\rho$ do not cross $\rho^\prime\backslash\{p_j,p_{j+1}\}$ transversely. Thus $\rho\backslash\{p_i,p_{i+1}\}$ does not cross $\rho^\prime\backslash\{p_j,p_{j+1}\}$ transversely.

If $\{p_i,p_{i+1}\}\cap\{p_j,p_{j+1}\}=\phi$, then $\rho$ does not cross $\rho^\prime$ transversely and thus $\rho$ and $\rho^\prime$ are compatible. So $\alpha_i^\prime$ and $\alpha_j^\prime$ are compatible.

If $\{p_i,p_{i+1}\}\cap\{p_j,p_{j+1}\}\neq \phi$, then we must have $p_i=p_j$ or $p_{i+1}=p_{j+1}$ or $p_i=p_{j+1}$ or $p_j=p_{i+1}$. Since $\alpha$ has no self-crossing, we know that if $1\leq l_1\neq l_2\leq d-1$, then $p_{l_1}\neq p_{l_2}$. So if $\{p_i,p_{i+1}\}\cap\{p_j,p_{j+1}\}\neq \phi$, we can get $i=j-1$ or $i=j$ or $i=j+1$.

If $i=j$, then $\rho=\rho^\prime$. In this case, $\alpha_i^\prime$ and $\alpha_j^\prime$ are compatible.

If $i=j-1$, then $j=i+1$ and $p_j=p_{i+1}$. Recall that,
\begin{eqnarray}
\rho&=&(e_1\xrightarrow[]{\text{along }\beta} p_i)\text{ then anticlockwise to } (p_i\xrightarrow[]{\text{along }\alpha}p_{i+1})\text{ then clockwise to } (p_{i+1}\xrightarrow[]{\text{along }\beta}e_2),\nonumber\\
\rho^\prime&=&(e_1^\prime\xrightarrow[]{\text{along }\beta} p_j)\text{ then anticlockwise to } (p_j\xrightarrow[]{\text{along }\alpha}p_{j+1})\text{ then clockwise to } (p_{j+1}\xrightarrow[]{\text{along }\beta}e_2^\prime),\nonumber\\
&=&(e_1^\prime\xrightarrow[]{\text{along }\beta} p_{i+1})\text{ then anticlockwise to } (p_{i+1}\xrightarrow[]{\text{along }\alpha}p_{i+2})\text{ then clockwise to } (p_{i+2}\xrightarrow[]{\text{along }\beta}e_2^\prime),\nonumber\\
&\xlongequal{\text{as curves}}&
(e_2^\prime\xrightarrow[]{\text{along }\beta} p_{i+2})\text{ then anticlockwise to } (p_{2+1}\xrightarrow[]{\text{along }\alpha}p_{i+1})\text{ then clockwise to } (p_{i+1}\xrightarrow[]{\text{along }\beta}e_1^\prime).\nonumber
\end{eqnarray}
Locally, at the point $p_{i+1}$, the curves $\rho$ and $\rho^\prime$ look like the following.

$$
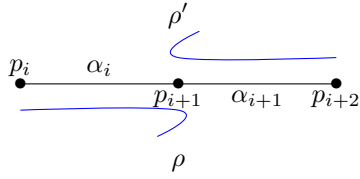
\begin{tikzpicture}[xscale=0.7,yscale=0.7,> = stealth]

\node(v1) at (0,0) {$\bullet$};
\node at (0,-0.3) {$p_{i+1}$};

\node at (-1.5,0.3) {$\alpha_i$};
\node(v2) at (3,0) {$\bullet$};

\node at (3,-0.3) {$p_{i+2}$};

\node at (1.5,-0.3) {$\alpha_{i+1}$};
\node(v3) at (-3,0) {$\bullet$};

\node at (-3,0.3) {$p_{i}$};

\draw[> = stealth, 
	shorten < = 0pt,
shorten >=0pt](-3,0)--(3,0);

\draw[blue]  plot[smooth, tension=.7] coordinates {(-3,-0.5) (0,-0.5) (-0.4,-1)};

\draw[blue]  plot[smooth, tension=.7] coordinates {(3,0.5) (0,0.5) (0.4,1)};
\node at (0,-1.5) {$\rho$};

\node at (0,1.3) {$\rho^\prime$};
\end{tikzpicture}
$$

In this case, $\rho$ can not cross with $\rho^\prime$ transversely at $p_j=p_{i+1}$. Thus $\rho$ can not cross with $\rho^\prime$ transversely by the discussion before. So if $i=j-1$, then $\alpha_i^\prime$ and $\alpha_j^\prime$ are compatible. Similarly, if $i=j+1$, $\alpha_i^\prime$ and $\alpha_j^\prime$ are compatible.

 Hence, if $1\leq i,j\leq d-1$, then $\alpha_i^\prime$ and $\alpha_j^\prime$ are  compatible. For $i$ and $j$ in the remain cases, the proof is similar. So any two arcs in $\mathcal T_\beta(\alpha)$ are compatible.
\end{proof}

\begin{Lemma}\label{lemcomp2}
Let $\alpha,\beta,\gamma$ be three arcs. If  $\gamma$ is compatible with both $\alpha$ and $\beta$, then $\gamma$ is compatible with the arcs in $\mathcal T_\beta(\alpha)$.
\end{Lemma}
\begin{proof}
If $\alpha$ is compatible with $\beta$, then $\mathcal T_\beta(\alpha)=\{\beta\}\cup\{\alpha\}$. In this case, the result is clear.
So we can assume that $\alpha$ is not compatible with $\beta$.

Keep the notations introduced before. Let $\alpha_k^\prime$ be an arc in $\mathcal T_\beta(\alpha)\backslash\{\beta\}$. We first consider $1\leq k\leq d-1$.  In this case£¬ we know that $\alpha_k^\prime$ is   isotopy to the curve
$$\rho:=(e_1\xrightarrow[]{\text{along }\beta} p_k)\text{ then anticlockwise to } (p_k\xrightarrow[]{\text{along }\alpha}p_{k+1})\text{ then clockwise to } (p_{k+1}\xrightarrow[]{\text{along }\beta}e_2),$$
where $e_1, e_2\in E(\beta)$ and the sub-curves $(e_1\xrightarrow[]{\text{along }\beta} p_k)$ and $(p_{k+1}\xrightarrow[]{\text{along }\beta}e_2)$ are uniquely determined by the ``anticlockwise direction and clockwise direction" in the above curve.

Since $\gamma$ is compatible with both $\alpha$ and $\beta$, we know that $\gamma$ has no crossings with $\rho$. So $\gamma$ is compatible with $\alpha_k^\prime$. Similarly, we can show that  $\gamma$ is compatible with $\alpha_0^\prime$ and $\alpha_d^\prime$. Hence, we get that  $\gamma$ is compatible with any arc in $\mathcal T_\beta(\alpha)=\{\beta\}\cup\{\alpha_0^\prime,\cdots,\alpha_d^\prime\}$.
\end{proof}

\begin{Lemma}\label{lemcomp3}
Let $\alpha$ and $\gamma$ be two compatible arcs, then  the arcs in $\mathcal T_\beta(\alpha)$ are compatible with the arcs in $\mathcal T_\beta(\gamma)$.
\end{Lemma}
\begin{proof}
By Lemma \ref{lemcomp1}, it suffices to show that the arcs in $\mathcal T_\beta(\alpha)\backslash\{\beta\}$ are compatible with the arcs in $\mathcal T_\beta(\gamma)\backslash\{\beta\}$.

If $\alpha$ is compatible with $\beta$, then $\mathcal T_\beta(\alpha)=\{\beta\}\cup\{\alpha\}$. In this case, the result follows from Lemma \ref{lemcomp2}. The proof for the case that $\gamma$ is compatible with $\beta$ is similar.

So we can assume that $\alpha$ crosses $\beta$ exactly $d_1$ times with $d_1\geq 1$ and $\gamma$ crosses $\beta$ exactly $d_2$ times with $d_2\geq 1$.
We fix an orientation for $\alpha$ (respectively, $\gamma$)
and we denote its starting point by $s_1$ (respectively, $s_2$) and its endpoint by $r_1$ (respectively, $r_2$), with $s_1,s_2,r_1,r_2\in M$.  Let
$s_1=p_0, p_1,\cdots, p_{d_1}, p_{d_1+1}=r_1$ (respectively, $s_2=p_0^\prime, p_1^\prime,\cdots, p_{d_2}^\prime, p_{d_2+1}^\prime=r_2$) be the intersection points of $\alpha$ (respectively, $\gamma$) and $\beta$ in order of occurrence on $\alpha$ (respectively, $\gamma$). For $k=0,1,\cdots,d_1$ (respectively, $k=0,1,\cdots,d_2$), let $\alpha_k$ (respectively, $\gamma_k$) denote the segment of the path $\alpha$ (respectively, $\gamma$) from the point $p_k$ (respectively, $p_k^\prime$) to the point
$p_{k+1}$ (respectively, $p_{k+1}^\prime$), and $E(\beta)$ be the set of endpoints of $\beta$.

 We know that $\mathcal T_\beta(\alpha)\backslash\{\beta\}=\{\alpha_0^\prime,\cdots,\alpha_{d_1}^\prime\}$ and $\mathcal T_\beta(\gamma)\backslash\{\beta\}=\{\gamma_0^\prime,\cdots,\gamma_{d_2}^\prime\}$. Let $\alpha_i^\prime$ be an arc in $\mathcal T_\beta(\alpha)\backslash\{\beta\}$, and $\gamma_j^\prime$ be an arc in $\mathcal T_\beta(\gamma)\backslash\{\beta\}$. We first consider the case $1\leq i\leq d_1-1$ and $1\leq j\leq d_2-1$. In this case, we know that  $\alpha_i^\prime$ is   isotopy to the curve
$$\rho:=(e_1\xrightarrow[]{\text{along }\beta} p_i)\text{ then anticlockwise to } (p_i\xrightarrow[]{\text{along }\alpha}p_{i+1})\text{ then clockwise to } (p_{i+1}\xrightarrow[]{\text{along }\beta}e_2),$$
and $\gamma_j^\prime$ is   isotopy to the curve
$$\rho^\prime:=(e_1^\prime\xrightarrow[]{\text{along }\beta} p_j^\prime)\text{ then anticlockwise to } (p_j^\prime\xrightarrow[]{\text{along }\gamma}p_{j+1}^\prime)\text{ then clockwise to } (p_{j+1}^\prime\xrightarrow[]{\text{along }\beta}e_2^\prime),$$
where $e_1, e_2, e_1^\prime,e_2^\prime\in E(\beta)$ and the sub-curves $$(e_1\xrightarrow[]{\text{along }\beta} p_i),\;(p_{i+1}\xrightarrow[]{\text{along }\beta}e_2),\;(e_1^\prime\xrightarrow[]{\text{along }\beta} p_j^\prime),\;(p_{j+1}^\prime\xrightarrow[]{\text{along }\beta}e_2^\prime)$$
are uniquely determined by the ``anticlockwise directions and clockwise directions" in $\rho$ and $\rho^\prime$.
Since $\alpha$ and $\gamma$ are compatible, they have no crossings in the interior of the surface. Then by the similar discussion as the proof in Lemma \ref{lemcomp1} (ii), we can obtained that  $\rho\backslash\{p_i,p_{i+1}\}$ does not cross $\rho^\prime\backslash\{p_j^\prime,p_{j+1}^\prime\}$ transversely.
Since  $\alpha$ and $\gamma$  have no crossings in the interior of the surface, we have $\{p_i,p_{i+1}\}\cap \{p_j^\prime,p_{j+1}^\prime\}=\phi$. Thus $\rho$ does not cross $\rho^\prime$ transversely, which means that $\alpha_i^\prime$ and $\gamma_j^\prime$ are compatible for $1\leq i\leq d_1-1$ and $1\leq j\leq d_2-1$. For $i$ and $j$ in the remain cases, the proof is similar.
\end{proof}

\begin{Definition}
Let $(S,M)$ be an unpunctured surface with a triangulation $T$, and $\beta$ be an arc of  $(S,M)$. The {\bf elementary co-Bongartz completion} of $\beta$ with respect to $T$ is defined to be the set of (boundary) arcs
$$\mathcal T_\beta(T):=\bigcup\limits_{i=1}^{n+m}\mathcal T_\beta(\tau_i),$$
where $T=\{\tau_1,\cdots,\tau_n,\tau_{n+1},\cdots,\tau_{n+m}\}$.
\end{Definition}

\begin{Theorem}
Let $(S,M)$ be an unpunctured surface and $T=\{\tau_1,\cdots,\tau_n,\tau_{n+1},\cdots,\tau_{n+m}\}$ be a triangulation of $(S,M)$. Then $\mathcal T_\beta(T)$ is a triangulation of $(S,M)$ for any arc $\beta$.
\end{Theorem}

\begin{proof}If $\beta\in T$, then $\mathcal T_\beta(T)=T$ and the result follows. So we can assume that $\beta\notin T$.

By Lemma \ref{lemcomp1} and \ref{lemcomp3}, any two arcs in $\mathcal T_\beta(T)$ are compatible.
Now it suffices to show that if an arc $\gamma$ is compatible with any arc in $\mathcal T_\beta(T)$, then $\gamma\in \mathcal T_\beta(T)$.

Let $\gamma$ be an arc satisfying that $\gamma$ is compatible with any arc in $\mathcal T_\beta(T)$. In particular, $\gamma$ is compatible with $\beta$, thus $\mathcal T_\beta(\gamma)=\{\beta\}\cup\{\gamma\}$.

If $\gamma\in T$, then $\gamma\in \mathcal T_\beta(\gamma)\subseteq\mathcal T_\beta(T)$. If $\gamma\notin T$, we will show that $\gamma$ and $\beta$ have a common endpoint.

By $\gamma\notin T$,  there exists $\alpha:=\tau_i\in T$ such that $\gamma$ and $\alpha=\tau_i$ are not compatible. If $\alpha$ is compatible with $\beta$, then
 $\alpha=\tau_i\in\mathcal T_\beta(\alpha)=\{\beta\}\cup \{\alpha\}\subseteq\mathcal T_\beta(T)$. This contradicts that $\gamma$ is compatible with any arc in $\mathcal T_\beta(T)$. So $\alpha$ is not compatible with $\beta$.

We can assume that $\alpha$ crosses $\beta$ exactly $d$ times ($d\geq 1$).
We  fix an orientation for $\alpha$
and denote its starting point by $s$ and its endpoint by $r$, with $s,r\in M$.  Let
$s=p_0, p_1,\cdots, p_d, p_{d+1}=r$ be the intersection points of $\alpha$ and $\beta$ in order of occurrence on $\alpha$.
For $k=0,1,\cdots,d$, let $\alpha_k$ denote the segment of the path $\alpha$ from the point $p_k$ to the point
$p_{k+1}$, and $E(\beta)$ be the set of endpoints of $\beta$.

Since $\alpha$ and $\gamma$ are not compatible, they must have crossings. Let $P$ be a crossing between $\alpha$ and $\gamma$, thus $P\in \alpha\cap\gamma$. Since $\gamma$ and $\beta$ are compatible, we must have $P\notin \beta$.
So $P\notin\{p_1,\cdots,p_d\}\subseteq \beta\cap\alpha$. Thus there exists $k\in\{0,1,\cdots,d\}$ such that $P$ is a point in the interior the the curve $\alpha_k$, i.e., $P$ is a crossing between $\gamma$ and $\alpha_k$.
Since $k\in\{0,1,\cdots,d\}$, we know that either $1\leq k\leq d$ or $1\leq k+1\leq d$. Without loss of generality, we will assume $1\leq k\leq d$ in the sequel, i.e., the endpoint $p_k$ of $\alpha_k$ in a interior point in $\alpha$.

Recall that $\alpha_k^\prime$ is isotopy to the curve
$$\rho:=(e_1(\beta)\xrightarrow[]{\text{along }\beta} p_k)\text{ then anticlockwise to } (p_k\xrightarrow[]{\text{along }\alpha}p_{k+1})\text{ then clockwise to } (p_{k+1}\xrightarrow[]{\text{along }\beta}e_2(\beta)),$$
where $e_1(\beta), e_2(\beta)\in E(\beta)$ and the sub-curves $(e_1(\beta)\xrightarrow[]{\text{along }\beta} p_k)$ and $(p_{k+1}\xrightarrow[]{\text{along }\beta}e_2(\beta))$ are uniquely determined by the ``anticlockwise direction and clockwise direction" in the above curve.

Let $e_1(\gamma)$ be the endpoint of $\gamma$ such that $(e_1(\gamma)\xrightarrow[]{\text{along }\gamma} P)$ is anticlockwise to $(P\xrightarrow[]{\text{along }\alpha_k} p_{k+1})$. One can refer to the following picture.
$$
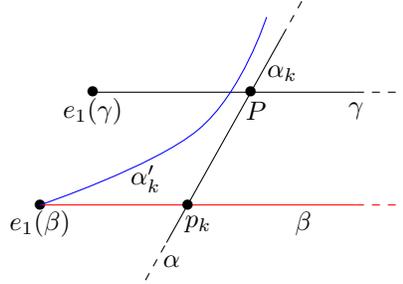
\begin{tikzpicture}[xscale=0.7,yscale=0.5]

\node at (-3,0) {$\bullet$};

\draw[red,> = stealth, 
	shorten < = 0pt,
shorten >=0 pt](-3,0)--(3,0);
\draw[dashed, red,> = stealth, 
	shorten < = 0pt,
shorten >=0 pt](3,0)--(4,0);

\node at (-2,3) {$\bullet$};

\draw [> = stealth, 
	shorten < = 0pt,
shorten >=0 pt](-2,3)--(3,3);
\draw[dashed,> = stealth, 
	shorten < = 0pt,
shorten >=0 pt](3,3)--(4,3);

\draw(-0.6,-1)--(1,3)--(1.6,4.5);
\draw[dashed](1.6,4.5)--(2,5.5);
\draw[dashed](-1,-2)--(-0.6,-1);

\node at (-3,-0.5) {$e_1(\beta)$};

\node at (2,-0.5) {$\beta$};

\node at (-0.2,0) {$\bullet$};
\node at (0,-0.5) {$p_k$};
\node at (1,3) {$\bullet$};

\node at (1.1,2.5) {$P$};
\node at (1.6,3.5) {$\alpha_k$};
\node at (-2,2.5) {$e_1(\gamma)$};
\node at (3,2.5) {$\gamma$};

\node at (-0.5,-1.5) {$\alpha$};

\draw [blue] plot[smooth, tension=.7] coordinates {(-3,0) (-0,2)   (1.3,5)};
\node at (-1,0.7) {$\alpha_k^\prime$};
\end{tikzpicture}
$$

One can see that if $e_1(\gamma)\neq e_1(\beta)$, then $\gamma$ can not be compatible with $\alpha_k^\prime\in\mathcal T_\beta(\alpha)\subseteq\mathcal T_\beta(T)$. This is a contradiction. So we must have $e_1(\gamma)= e_1(\beta)$. Thus $\gamma$ and $\beta$ have a common endpoint.

If $\gamma=\beta$, then $\gamma\in \mathcal T_\beta(T)$. So we can assume that $\gamma\neq \beta$. Since $\gamma$ and $\beta$ have a common endpoint, we know there exists a unique (boundary) arc $\kappa$ such that $\gamma,\beta,\kappa$ form the three distinct sides of a triangle $\triangle$. Without loss of generality, we assume that $\beta$ is clockwise to $\gamma$ in the triangle $\triangle$. One can refer to the following picture.

$$
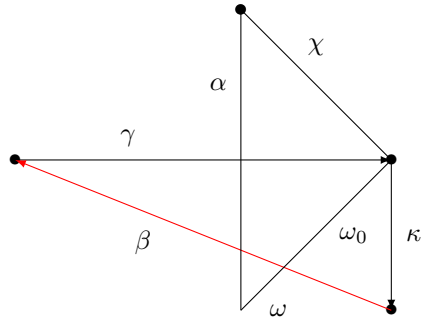
\begin{tikzpicture}

\node at (2,2) {$\bullet$};
\node at (2,0) {$\bullet$};
\node at (-3,2) {$\bullet$};
\draw[-latex] (-3,2)--(2,2);
\draw[-latex](2,2)--(2,0);
\draw[-latex,red](2,0)--(-3,2);

\node at (2.3,1) {$\kappa$};

\node at (-1.3,0.9) {$\beta$};

\node at (-1.5,2.3) {$\gamma$};

\node at (0,4) {$\bullet$};
\node at (0,4) {$\bullet$};
\draw(0,0)--(0,4);
\draw (0,4)--(2,2);
\draw(2,2)--(0,0);

\node at (-0.3,3) {$\alpha$};

\node at (1,3.5) {$\chi$};

\node at (1.5,1) {$\omega_0$};
\node at (0.5,0) {$\omega$};
\end{tikzpicture}
$$

Fix the orientation of $\triangle$ by clockwise orientation when looking at
it from outside the surface. The orientation of $\triangle$ gives an orientation of  $\gamma$. Let $\alpha$ be the arc in $T$ such that the last crossing point between $\gamma$ and the arcs in $T$ lies on $\alpha$. Then there exists a triange $\triangle^\prime$ in $T$ with sides $\alpha,\chi,\omega\in T$ of the shape in the above picture. We have proved before that $\alpha$ is not compatible with $\beta$, so they have crossings. Thus $\omega$ also has the crossings with $\beta$. Let $\omega_0$ be the segment of $\omega$ in the interior of the triangle $\triangle$. The segment $\omega_0$ will deform into $\omega_0^\prime=\gamma\in\mathcal T_\beta(\omega)\subseteq \mathcal T_\beta(T)$. This completes the proof.
\end{proof}

\begin{Theorem}\label{thmlast}
 Fix an unpunctured surface $(S,M)$ with an initial triangulation $$T=\{\tau_1,\cdots,\tau_n, \tau_{n+1},\cdots,\tau_{n+m}\},$$ where $\tau_{n+1},\cdots,\tau_{n+m}$ are boundary arcs.
Let $\mathcal A(B_T)$ be a cluster algebra with the initial exchange matrix $B_T$. Let $x$ and ${\bf x}$ be the bijections given in Theorem \ref{thmbijection}. Then for any arc $\beta$, we have the following diagram.
$$\xymatrix{T\ar@{<->}[d]\ar[r]^{\mathcal T_\beta}&T^\prime\ar@{<->}[d]\\
{\bf x}_T\ar[r]^{\mathcal T_{x_{\beta}}}&{\bf x}_{T^\prime}
}$$
\end{Theorem}
\begin{proof}
Let $T^\prime=\mathcal T_\beta(T)=\{\gamma_1,\cdots,\gamma_n,\tau_{n+1},\cdots,\tau_{n+m}\}$, we will show that ${\bf x}_{T^\prime}=\mathcal T_{x_{\beta}}({\bf x}_T)$.
By Corollary \ref{corclustertwi}, it suffices to show that the $k$-th row vector of $G_{{\bf x}_T}^{{\bf x}_{T^\prime}}$ the $G$-matrix of ${\bf x}_T$ with respect to ${\bf x}_{T^\prime}$ is a nonnegative vector for any $k$ such that $x_{\gamma_k}\neq x_\beta$. This is equivalent to show that $k$-th row vector of $G_T^{T^\prime}=(g_{ij;T}^{T^\prime})$ is a nonnegative vector for any $k$ such that $\gamma_k\neq \beta$, by Theorem \ref{thmarcgvector}.

Without loss of generality, we set $\gamma_n=\beta$. We need to show that the $k$-th row vector of $G_T^{T^\prime}=(g_{ij;T}^{T^\prime})$ is in $\mathbb Z_{\geq0}^n$ for $k=1,2,\cdots,n-1$.

Assume by contradiction that there exists a $g_{ki;T}^{T^\prime}<0$ for some $k\in\{1,\cdots,n-1\}$. We consider the $g$-vector $g_{\tau_i}^{T^\prime}$ of $\tau_{i}$ with respect to $T^\prime$. By Corollary \ref{corgcom} (ii) and $g_{ki;T}^{T^\prime}<0$, we know that there exists $j\in I^-(T^\prime,\tau_i)$ such that $\tau_i$ crosses with $\gamma_k$ at $j$ (identifying it with $p_j$). Then the local picture is the following.
$$
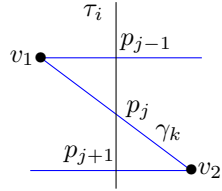
\begin{tikzpicture}[xscale=1,yscale=1,> = stealth, 
	shorten < = -4pt,
shorten >=-4pt]

\draw [blue] (5,3)--(3,3)--(5,1.5)--(3,1.5);
\draw(4,3.6)--(4,1.4);

\node at (3.7,3.6) {$\tau_i$};
\node at (4.4,3.2) {$p_{j-1}$};

\node at (4.3,2.3) {$p_{j}$};

\node at (3.65,1.7) {$p_{j+1}$};
\node at (2.75,3) {$v_1$};
\node at (5.25,1.5) {$v_2$};

\node at (4.7,2) {$\gamma_k$};

\node at (3,3) {$\bullet$};

\node at (5,1.5) {$\bullet$};
\end{tikzpicture}
$$
where the triangles $p_{j-1}v_1p_j$ and $p_jv_2p_{j+1}$ are contractible.

Since $\gamma_k\in T^\prime=\mathcal T_\beta(T)$ and $\gamma_k\neq \gamma_n=\beta$, there exists $\tau_l\in T$ such that $\gamma_k\in \mathcal T_\beta(\tau_l)\backslash\{\beta\}$. We must have that $\tau_l$ is not compatible with $\beta$, otherwise, $ \mathcal T_\beta(\tau_l)\backslash\{\beta\}=\tau_l$ and $\gamma_k=\tau_l\in T$. This contradicts that $\gamma_k$ crosses with $\tau_i$ at $j$ (identifying it with $p_j$).
 So  $\tau_l$ is not compatible with $\beta$. Then by the construction of $\mathcal T_\beta(\tau_l)$, we know that $\gamma_k\in \mathcal T_\beta(\tau_l)\backslash\{\beta\}$ is isotopy to the curve $\rho$ with one of the following two forms.

$$\begin{tikzpicture}[xscale=1,yscale=1,> = stealth, 
	shorten < = -1pt,
shorten >=-1pt]

\draw [red](0,1)--node[above]{subcurve of $\beta$}(3,1);

\draw(3,1)--node[above]{subcurve of $\tau_l$}(6,1);
\draw [red](0,0)--node[above]{subcurve of $\beta$}(3,0);
\draw(3,0)--node[above]{subcurve of $\tau_l$}(6,0);

\draw [red](6,0)--node[above]{subcurve of $\beta$}(9,0);
\node at (-0.5,1) {$\rho=$};
\node at (-2,0.5) {or};
\node at (-0.5,0) {$\rho=$};
\end{tikzpicture}$$
Since $\tau_i$ and $\gamma_k$ are not compatible, we know that $\tau_i$ and $\rho$ are not compatible. Thus $\tau_i\cap \rho\neq\phi$, i.e., they must have crossing points. Since $\tau_i$ and $\tau_l$ are compatible, we know $\tau_i\cap \rho\subseteq \beta$.

Now we identify $\gamma_k$ with $\rho$. If $\rho$ has the first form, then the picture is

$$
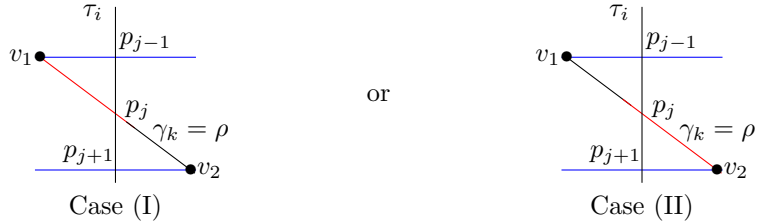
\begin{tikzpicture}[xscale=1,yscale=1,> = stealth, 
	shorten < = -2pt,
shorten >=-2pt]

\draw [blue] (5,3)--(3,3);
\draw[red](3,3)--(4.2,2.1);
\draw(4.2,2.1)--(5,1.5);

\draw[blue](5,1.5)--(3,1.5);
\draw(4,3.6)--(4,1.4);

\node at (3.7,3.6) {$\tau_i$};
\node at (4.4,3.2) {$p_{j-1}$};

\node at (4.3,2.3) {$p_{j}$};

\node at (3.65,1.7) {$p_{j+1}$};
\node at (2.75,3) {$v_1$};
\node at (5.25,1.5) {$v_2$};

\node at (5,2) {$\gamma_k=\rho$};

\node at (3,3) {$\bullet$};

\node at (5,1.5) {$\bullet$};

\draw [blue] (12,3)--(10,3);
\draw(10,3)--(10.8,2.4);
\draw[red](10.8,2.4)--(12,1.5);

\draw[blue](12,1.5)--(10,1.5);
\draw(11,3.6)--(11,1.4);

\node at (10.7,3.6) {$\tau_i$};
\node at (11.4,3.2) {$p_{j-1}$};

\node at (11.3,2.3) {$p_{j}$};

\node at (10.65,1.7) {$p_{j+1}$};
\node at (9.75,3) {$v_1$};
\node at (12.25,1.5) {$v_2$};

\node at (12,2) {$\gamma_k=\rho$};

\node at (10,3) {$\bullet$};

\node at (12,1.5) {$\bullet$};

\node at (7.5,2.5) {or};
\node at (4,1) {Case (I)};

\node at (11,1) {Case (II)};
\end{tikzpicture}
$$

Without loss of generality, we assume that we are in Case (I). By the construction of $\gamma_k\in\mathcal T_\beta(\tau_l)\backslash \{\beta\}$, we know that $\beta$ is clockwise to $\gamma_k=\rho$ at the endpoint $v_1$ in Case (I).
Thus the other part of $\beta$ is contained in the triangle $p_jv_2p_{j+1}$. Since the triangle $p_jv_2p_{j+1}$ is contractible, we must have $\beta$ is isotopy to $\gamma_k=\rho$, i.e., $\gamma_k=\beta=\gamma_n$. Thus $k=n$, which  contradicts that $k\in \{1,\cdots,n-1\}$. So if $\rho$ has the first form, we have $g_{ki;T}^{T^\prime}\geq 0$ for any $k=1,\cdots,n-1$ and $i=1,\cdots,n$.

If $\rho$ has the second form, by the similar arguments, we can also conclude that $\gamma_k=\beta=\gamma_n$ and $k=n$, which is a contradiction. So in this case, we also have $g_{ki;T}^{T^\prime}\geq 0$ for any $k=1,\cdots,n-1$ and $i=1,\cdots,n$.

 So the $k$-th row vector of $G_T^{T^\prime}$ is in $\mathbb Z_{\geq0}^n$ for $k=1,2,\cdots,n-1$ and this completes the proof.
\end{proof}
The following result follows direct from Theorem \ref{thmlast} and Corollary \ref{corlast}.
\begin{Corollary}
Let $(S,M)$ be an unpunctured surface with a triangulation $T$ and  $\beta_1$, $\beta_2$ be any two compatible arcs. Then  $\mathcal T_{\beta_2}\mathcal T_{\beta_1}(T)=\mathcal T_{\beta_1}\mathcal T_{\beta_2}(T)$, i.e., the following commutative diagram holds.
$$\xymatrix{T\ar[r]^{\mathcal T_{\beta_1}}\ar[d]_{\mathcal T_{\beta_2}}&T^\prime\ar[d]^{\mathcal T_{\beta_2}}\\
T^{\prime\prime}\ar[r]^{\mathcal T_{\beta_1}}&T^{\prime\prime\prime}
}$$
\end{Corollary}

Thanks to the above corollary, we can give the definition of co-Bongartz completions on unpunctured surfaces.

Let $(S,M)$ be an unpunctured surface with a triangulation $T$, and $U=\{\beta_1,\cdots,\beta_s\}$ be a compatible set of arcs. The {\bf co-Bongartz completion} of $U$ with respect to $T$ is defined to be the triangulation given by
$$\mathcal T_U(T):=\mathcal T_{\beta_{i_s}}\cdots\mathcal T_{\beta_{i_2}}\mathcal T_{\beta_{i_1}}(T),$$
where $i_1,\cdots,i_s$ is any permutation of $1,\cdots,s$.

Finally, we give an example for Theorem \ref{thmlast}.
\begin{Example}
Keep the notations in Example \ref{example1}. Let $T_{t_0}$ be the triangulation corresponding to the initial seed $({\bf x}_{t_0},B_{t_0})$, and $T_{t}=\mu_3\mu_1\mu_3\mu_2(T_{t_0})$ be the triangulation corresponding to the seed $({\bf x}_t,B_t)=\mu_3\mu_1\mu_3\mu_2({\bf x}_{t_0},B_{t_0})$, where $T_{t_0}$ and $T_t$ are given in the following figure.

$$
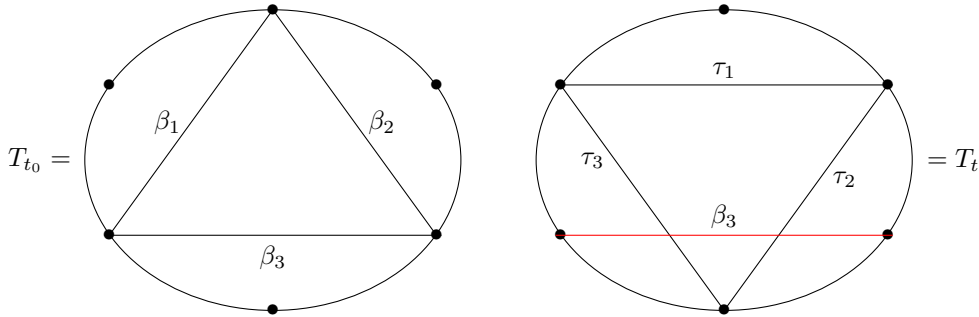
\begin{tikzpicture}[xscale=0.5,yscale=0.5,> = stealth, 
	shorten < = -8pt,shorten >=-8pt]
\draw  (-6,0) ellipse (5 and 4);
\node(n1) at (-6,4) {$\bullet$};
\node(n2) at (-6,-4) {$\bullet$};

\node(n3) at (-10.35,2) {$\bullet$};

\node(n4) at (-1.65,2) {$\bullet$};

\node(n5) at (-10.35,-2) {$\bullet$};

\node(n6) at (-1.65,-2) {$\bullet$};

\draw  (6,0) ellipse (5 and 4);

\node(v1) at (6,4) {$\bullet$};
\node (v2)at (6,-4) {$\bullet$};

\node(v3) at (10.35,2) {$\bullet$};

\node(v4) at (1.65,2) {$\bullet$};

\node(v5) at (10.35,-2) {$\bullet$};

\node(v6) at (1.65,-2) {$\bullet$};

\draw (n1)--(n5);
\draw(n1)--(n6);
\draw(n5)--(n6);

\node at (-8.8,1) {$\beta_1$};
\node at (-3.1,1) {$\beta_2$};
\node at (-6,-2.6) {$\beta_3$};

\draw (v3)--(v4);
\draw(v2)--(v3);
\draw(v2)--(v4);
\draw[red](v5)--(v6);

\node at (6,2.4) {$\tau_1$};

\node at (9.2,-0.5) {$\tau_2$};
\node at (2.5,0) {$\tau_3$};
\node at (6,-1.5) {$\beta_3$};

\node at (-12.2,0) {$T_{t_0}=$};

\node at (12.1,0) {$=T_{t}$};
\end{tikzpicture}$$
Then by the construction of $\mathcal T_{\beta_3}(T_{t})$, we know that

$$
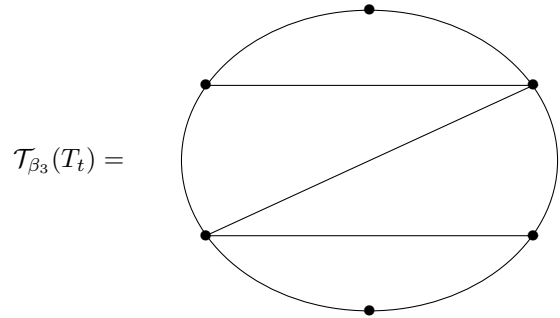
\begin{tikzpicture}[xscale=0.5,yscale=0.5,> = stealth, 
	shorten < = -8pt,shorten >=-8pt]
\draw  (-6,0) ellipse (5 and 4);
\node(n1) at (-6,4) {$\bullet$};
\node(n2) at (-6,-4) {$\bullet$};

\node(n3) at (-10.35,2) {$\bullet$};

\node(n4) at (-1.65,2) {$\bullet$};

\node(n5) at (-10.35,-2) {$\bullet$};

\node(n6) at (-1.65,-2) {$\bullet$};

\draw (n3)--(n4);
\draw(n4)--(n5);
\draw(n5)--(n6);

\node at (-14,0) {$\mathcal T_{\beta_3}(T_t)=$};

\end{tikzpicture}$$
It can be seen that $\mathcal T_{\beta_3}(T_t)=\mu_1\mu_2(T_{t_0})$, which corresponds to  ${\bf x}_{t^\prime}=\mathcal T_{x_3}({\bf x}_t)$, where ${\bf x}_{t^\prime}$ is the cluster in the seed $({\bf x}_{t^\prime},B_{t^\prime})=\mu_1\mu_2({\bf x}_{t_0},B_{t_0})$.
\end{Example}

\vspace{5mm}
{\bf Acknowledgements:}\;I would like to thank  my supervisor Professor Fang Li for introducing me to this topic and for his encouragement these years. This project is supported by the National Natural Science Foundation of China (No.11671350 and No.11571173).

\def\cprime{$'$} \def\cprime{$'$}
\providecommand{\bysame}{\leavevmode\hbox to3em{\hrulefill}\thinspace}
\providecommand{\MR}{\relax\ifhmode\unskip\space\fi MR }
\providecommand{\MRhref}[2]{%
  \href{http://www.ams.org/mathscinet-getitem?mr=#1}{#2}
}
\providecommand{\href}[2]{#2}

\end{document}